\documentclass[reqno]{amsart}
\DeclareFontSeriesDefault[rm]{bf}{b}

\usepackage[colorlinks,linkcolor={darkblue}]{hyperref}
\usepackage{amsmath}
\usepackage{amsfonts}
\usepackage{amsthm}
\usepackage{amssymb}
\usepackage{graphicx}
\usepackage{bm}
\usepackage{dsfont}
\usepackage[font=footnotesize]{caption}
\usepackage[all]{xy}
\usepackage{tikz-cd}
\usepackage{mathtools}

\definecolor{darkred}{rgb}{1,0,0} 
\definecolor{darkgreen}{rgb}{0,0.4,0}
\definecolor{darkblue}{rgb}{0,0,0.8}

\hypersetup{colorlinks,
linkcolor=darkblue,
filecolor=darkgreen,
urlcolor=darkred,
citecolor=darkgreen}

\numberwithin{equation}{section}
\newtheorem{thm}{Theorem}[section]
\newtheorem{cor}[thm]{Corollary}
\newtheorem{lem}[thm]{Lemma}
\newtheorem{prop}[thm]{Proposition}
\theoremstyle{definition}
\newtheorem{rem}[thm]{Remark}

\newcommand{\nablah}{\nabla^{\mathrm{h}}}
\newcommand{\nablav}{\nabla^{\mathrm{v}}}
\newcommand{\N}{\mathds{N}}
\newcommand{\Z}{\mathds{Z}}
\newcommand{\R}{\mathds{R}}
\newcommand{\F}{\mathds{F}}
\newcommand{\PPP}{\mathds{P}}
\newcommand{\RP}{\mathds{RP}}
\newcommand{\CP}{\mathds{CP}}
\newcommand{\HP}{\mathds{HP}}
\newcommand{\CaP}{\mathrm{Ca}\mathds{P}}
\newcommand{\PT}{\mathds{P}\Tan}
\newcommand{\Q}{\mathds{Q}}

\newcommand{\K}{\mathds{K}}

\newcommand{\Sp}{\mathrm{Sp}}
\newcommand{\ev}{\mathrm{ev}}
\newcommand{\pr}{\mathrm{pr}}
\newcommand{\diff}{\mathrm{d}}
\newcommand{\dist}{\mathrm{dist}}
\newcommand{\Tan}{\mathrm{T}}
\newcommand{\Ver}{\mathrm{Ver}}

\newcommand{\Emb}{\mathrm{Emb}}
\newcommand{\End}{\mathrm{End}}

\newcommand{\WW}{\mathcal{W}}

\newcommand{\UU}{\mathcal{U}}
\newcommand{\BB}{\mathcal{B}}
\newcommand{\VV}{\mathcal{V}}
\newcommand{\W}{W^{1,2}}

\newcommand{\injrad}{\mathrm{injrad}}
\newcommand{\sys}{\mathrm{sys}}

\newcommand{\xx}{\bm{x}}

\newcommand{\crit}{{\mathrm{crit}}}

\DeclareMathOperator*{\toup}{\longrightarrow} 
\DeclareMathOperator*{\otup}{\longleftarrow}

\renewcommand{\twoheadrightarrow}{\mathrel{\mathrlap{\rightarrow}\mkern-2.3mu\rightarrow}}

\DeclareMathOperator*{\epi}{\twoheadrightarrow}

\DeclareMathOperator*{\diam}{\mathrm{diam}}
\newcommand{\llangle}{\langle\!\langle}
\newcommand{\rrangle}{\rangle\!\rangle}
\newcommand{\lllangle}{\langle\!\langle\!\langle}
\newcommand{\rrrangle}{\rangle\!\rangle\!\rangle}
\newcommand{\Diff}{\mathrm{Diff}}
\newcommand{\sigmas}{\sigma_\mathrm{s}}
\newcommand{\supp}{\mathrm{supp}}
\newcommand{\ind}{\mathrm{ind}}
\newcommand{\nul}{\mathrm{nul}}
\newcommand{\equaldown}{\arrow[d, no tail, no head, shift left =0.4]\arrow[d, no tail, no head, shift right =0.4]}
\usepackage{todonotes}

\begin{document}

\title[Closed geodesics on reversible Finsler 2-spheres]{Closed geodesics on reversible Finsler 2-spheres}

\author[G. De Philippis]{Guido De Philippis}
\address{Guido De Philippis\newline\indent Courant Institute of Mathematical Sciences, New York University\newline\indent 251 Mercer Street, New York, N.Y. 10012-1185, United States}
\email{guido@cims.nyu.edu}

\author[M. Marini]{Michele Marini}
\address{Michele Marini\newline\indent Scuola Internazionale Superiore di Studi Avanzati\newline\indent via Bonomea 265, 34136 Trieste, Italy}
\email{mmarini@sissa.it}

\author[M. Mazzucchelli]{Marco Mazzucchelli}
\address{Marco Mazzucchelli\newline\indent CNRS, \'Ecole Normale Sup\'erieure de Lyon, UMPA\newline\indent  46 all\'ee d'Italie, 69364 Lyon Cedex 07, France}
\email{marco.mazzucchelli@ens-lyon.fr}

\author[S. Suhr]{Stefan Suhr}
\address{Stefan Suhr\newline\indent Ruhr Universit\"at Bochum, Fakult\"at f\"ur Mathematik\newline\indent Universit\"atsstra{\ss}e 150, 44801 Bochum, Germany}
\email{stefan.suhr@rub.de}

\subjclass[2010]{53C22, 58E10, 53C44}
\keywords{Closed geodesics, reversible Finsler metrics, curve shortening flow, Lusternik-Schnirelmann theory}

\date{February 2, 2020. \emph{Revised}: January 14, 2022}

\dedicatory{Dedicated to Claude Viterbo on the occasion of his 60th birthday.}

\begin{abstract}
We extend two celebrated theorems on closed geodesics of Riemannian 2-spheres to the larger class of reversible Finsler 2-spheres: Lusternik-Schnirelmann's theorem asserting the existence of three simple closed geodesics, and Bangert-Franks-Hingston's theorem asserting the existence of infinitely many closed geodesics. In order to prove the first theorem, we employ the generalization of Grayson's curve shortening flow developed by Angenent-Oaks.

\tableofcontents
\end{abstract}

\maketitle

%\vspace{-30pt}

\section{Introduction}

Since the seminal work of Hadamard \cite{Hadamard:1898ww}, Poincar\'e \cite{Poincare:1905jp}, Birkhoff \cite{Birkhoff:1966xb}, and Morse \cite{Morse:1996ua}, it became evident that closed Riemannian manifolds of dimension at least 2 tend to have many closed geodesics (that is, periodic orbits of the geodesic flow). This evidence was supported by celebrated theorems of Gromoll-Meyer \cite{Gromoll:1969gh} and Vigu\'e Poirrier-Sullivan \cite{Vigue-Poirrier:1976ug}, which together assert that simply connected closed Riemannian manifolds with a non-monogenic rational cohomology ring always have infinitely many closed geodesics. This statement covers a large class of simply connected closed manifolds, leaving out those with the cohomology of a compact rank-one symmetric space: $S^n$, $\CP^n, \HP^n$, and $\CaP^2$. As of 2019, it is an open conjecture whether these spaces admit infinitely many closed geodesics for any choice of the Riemannian metric. The only known case is the one of $S^2$, for which the proof required a combination of spectacular work by Bangert \cite{Bangert:1993wo}, Franks \cite{Franks:1992jt}, and Hingston \cite{Hingston:1993ou} (either Franks' or Hingston's work, together with Bangert's one, provide the full result). The starting point for this work is another celebrated result due to Lusternik-Schnirelmann \cite{Lusternik:1929wa}, asserting that every Riemannian 2-sphere possesses at least three simple closed geodesics (that is, closed geodesics that are embedded circle in the Riemannian manifold). For many decades Lusternik-Schnirelmann's theorem was considered controversial due to a gap in their construction of a pseudo-gradient flow for the length function of simple closed curves, that have been subsequently addressed by many authors \cite{Ballmann:1978rw, Jost:1989pi, Jost:1991jt, Taimanov:1992wo, Hass:1994vx, Klingenberg:1995aa}. Nowadays, the gap is considered filled, for instance thanks to the work of Grayson \cite{Grayson:1989ec} on the curve shortening flow.

The closed geodesic problem can be studied on closed Finsler manifolds as well. A Finsler metric on a manifold $M$ is a continuous function $F:\Tan M\to[0,\infty)$, smooth outside the zero-section of $\Tan M$, positively homogeneous of degree $1$ (i.e. $F(x,\lambda v)=\lambda F(x,v)$ for all $(x,v)\in\Tan M$ and $\lambda\geq0$), and such that the restriction of its square $F^2$ to any fiber of $\Tan M$ has positive definite Hessian everywhere outside the origin. In the literature, a more general notion of Finsler metric is sometimes employed, but the one given here is the most appropriate for the study of geodesic flows. Many results, such as Gromoll-Meyer's one, remain valid essentially with the same proof in the Finsler category (see \cite{Lu:2015aa} and references therein). However, a striking example due to Katok \cite{Katok:1973mw}, and further explored by Ziller \cite{Ziller:1983aa}, shows that Lusternik-Schnirelmann's and Bangert-Franks-Hingston's theorems fail: there exists a Finsler metric on $S^2$ having only two closed geodesics. 

A Finsler metric $F:\Tan M\to[0,\infty)$ is called reversible when $F(x,-v)=F(x,v)$ for all $(x,v)\in\Tan M$. The Katok's Finsler metric does not satisfy this property. In the current paper, we show that all the above mentioned results valid for Riemannian 2-spheres remain valid for reversible Finsler 2-spheres.

\subsection{The curve shortening semi-flow}
In \cite{Oaks:1994aa}, Oaks provided a generalization of Grayson's curve shortening flow \cite{Grayson:1989ec}. As remarked by Angenent \cite{Angenent:2008aa}, such generalization allows to provide a curve shortening flow on orientable reversible Finsler surfaces: a tool to shrink embedded circles without creating self-inter\-sections. In this section, we state a theorem that summarizes all the properties of this flow (actually, a semi-flow) that we will need for the application to the closed geodesics problem.

Let $M$ be a closed oriented surface, equipped with a reversible Finsler metric $F$. We denote by $S^1:=\R/\Z$ the 1-periodic circle, and by $\Emb(S^1,M)$ the space of smooth embedded loops $\gamma:S^1\hookrightarrow M$, endowed with the $C^\infty$ topology (that is, the topology whose basis is given by the open sets $\UU\subset\Emb(S^1,M)$ of the $C^k$ topology, for all $k\in\N$).
We consider the Finsler length functional 
\begin{align}
\label{l:length}
L:\Emb(S^1,M)\to(0,\infty),\qquad 
L(\gamma) = \int_0^1 F(\gamma(u),\dot\gamma(u))\,\diff u.
\end{align}
The group of diffeomorphisms $\Diff(S^1)$ acts freely on $\Emb(S^1,M)$ by reparametrizations. Notice that, since the Finsler metric $F$ is homogeneous of degree 1 and reversible, the length functional is invariant by the $\Diff(S^1)$-action, i.e.\
 $L(\gamma)=L(\gamma\circ\theta)$ for all $\gamma\in\Emb(S^1,M)$ and $\theta\in\Diff(S^1)$.

We fix an auxiliary Riemannian metric $g$ on $M$, and we will simply write  $\|\cdot\|$ or $\|\cdot\|_g$ for its associated norm on tangent vectors. Since $(M,g)$ is an orientable Riemannian surface, it admits a canonical complex structure $J\in\End(\Tan M)$, i.e.\ $Jv$ is obtained by rotating $v\in\Tan_xM$ of a positive angle $\pi/2$ measured with $g$.  The positive normal to $\gamma\in\Emb(S^1,M)$ is the vector field
\begin{align*}
N_\gamma(u):=\frac{1}{\|\dot\gamma(u)\|}J\dot\gamma(u),
\end{align*}
where $\|\cdot\|$ is the Riemannian norm associated to $g$.  We set 
\begin{align}
\label{e:V_gamma}
 V_\gamma(u):=\frac{\big(\tfrac{\diff}{\diff u}F_v(\gamma(u),\dot\gamma(u))-F_x(\gamma(u),\dot\gamma(u)) \big)N_\gamma(u)}{\|\dot\gamma(u)\|}.
\end{align}
In the expression of $V_\gamma$, the terms $F_x$ and $F_v$ denote the partial derivatives of $F$ with respect to some local coordinates on $M$ (or, more precisely, local coordinates on $\Tan M$ induced by local coordinates on $M$). However, the covector 
\begin{align}
\label{e:EL_covector} 
\tfrac{\diff}{\diff u}F_v(\gamma(u),\dot\gamma(u)) - F_x(\gamma(u),\dot\gamma(u))\in\Tan_{\gamma(u)}^*M
\end{align} 
is independent of the choice of local coordinates, and vanishes identically if and only if $\gamma$ is a closed geodesic of $(M,F)$. Since $\dot\gamma$ is always in the kernel of this covector, we actually conclude that $V_\gamma$ vanishes identically if and only if $\gamma$ is a closed geodesic of $(M,F)$. Therefore, for each $\ell\geq\injrad(M,F)$ and $\epsilon>0$,  the open subset
\begin{align}
\label{e:U_ell_epsilon}
\UU(\ell,\epsilon):=
\Big\{\gamma\in\Emb(S^1,M)\ \Big|\ L(\gamma)\in(\ell-\epsilon^2,\ell+\epsilon^2),\ \max_{s\in S^1}|V_\gamma(s)|<\epsilon\Big\}.
\end{align}
is a neighborhood of the set of simple closed geodesics of $(M,F)$ with length $\ell$. We will employ the notation
\begin{align*}
\Emb(S^1,M)^{<\ell}:=\{\gamma\in\Emb(S^1,M)\ |\ L(\gamma)<\ell\}
\end{align*}
to denote the sublevel sets of the length functional.

\begin{rem}
In the literature, closed geodesics $\gamma$ are usually required to be pa\-ram\-e\-trized with constant positive speed, that is, the function $t\mapsto F(\gamma(t),\dot\gamma(t))$ is required to be constant and positive. 
In this paper, instead, we allow closed geodesics to be parametrized arbitrarily as immersed curves. Indeed, the equation $V_\gamma\equiv0$ is independent of the parametrization of $\gamma$ (Lemma~\ref{l:invariance}). From Section~\ref{s:critical_point_theory} on, we will often consider closed geodesics that are non-trivial critical points of the energy function, and thus parametrized with constant speed.
\hfill\qed
\end{rem}

We consider the evolution equation
\begin{align*}
\partial_t \gamma_t&=V_{\gamma_t}N_{\gamma_t}
\end{align*}
with prescribed initial condition $\gamma_0\in\Emb(S^1,M)$, where $\gamma_t\in C^\infty(S^1,M)$ for all $t$ for which it is defined. Up to slowing down the time evolution when $\gamma_t$ becomes short, the solutions of this equation give rise to a curve shortening semi-flow, whose properties are summarized in the next statement.

\begin{thm}\label{t:curve_shortening}
Let $(M,F)$ be a closed, orientable, reversible Finsler manifold, and $\rho_0>0$.
There exists a continuous map 
\[\psi:[0,\infty)\times\Emb(S^1,M)\to\Emb(S^1,M),\qquad\psi(t,\gamma)=\psi_t(\gamma),\]
with the following properties:
\begin{itemize}
\item[(i)] It is a semi-flow, i.e. $\psi_0=\mathrm{id}$ and $\psi_{t_2}\circ\psi_{t_1}=\psi_{t_1+t_2}$ for all $t_1,t_2\geq0$.

\item[(ii)] It is equivariant with respect to the action of circle diffeomorphisms, i.e.\ $\psi_t(\gamma\circ\theta)=\psi_t(\gamma)\circ\theta$ for all $\gamma\in\Emb(S^1,M)$ and $\theta\in\Diff(S^1)$.

\item[(iii)] The length function is not increasing along the trajectories of $\psi_t$. More precisely, $\tfrac{\diff}{\diff t}L(\psi_t(\gamma))\leq 0$ with equality if and only if $\gamma$ is a closed geodesic or $L(\gamma)\leq \rho_0$. 

\item[(iv)] For each $\ell>2\rho_0$ and $\epsilon>0$ there exists $\delta>0$ and a continuous function
\[\tau:\Emb(S^1,M)^{<\ell+\delta}\to(0,\infty)\] 
such that
\begin{align*}
\qquad\quad\
 \psi_{t}(\gamma)\in \UU(\ell,\epsilon)\cup\Emb(S^1,M)^{<\ell-\delta},\quad
 \forall \gamma\in \Emb(S^1,M)^{<\ell+\delta},\, t\geq\tau(\gamma).
\end{align*}

\item[(v)] If there are no simple closed geodesics with length in $[\ell_1,\ell_2]\subset(2\rho_0,\infty)$, then there exists a continuous function 
\[\tau:\Emb(S^1,M)^{<\ell_2}\to(0,\infty)\] 
such that
\begin{align*}
 \psi_t(\gamma)\subset \Emb(S^1,M)^{<\ell_1},
 \qquad
 \forall \gamma\in\Emb(S^1,M)^{<\ell_2},\,t\geq\tau(\gamma).
\end{align*}

\end{itemize}
\end{thm}

Most of the points in this theorem follow from Oaks \cite{Oaks:1994aa}, except point (iv), which is crucial for the applications.

\subsection{Closed geodesics on Finsler 2-spheres}

We already anticipated that the semi-flow of Theorem~\ref{t:curve_shortening} allows to extend the celebrated Lusternik-Schnirelmann's theorem \cite{Lusternik:1929wa} to the reversible Finsler setting. Actually, it will also allow to extend the characterization of simple Zoll geodesic flows on the 2-sphere, originally claimed in the Riemannian case by Lusternik \cite{Ljusternik:1966tk} and rigorously proved in \cite{Mazzucchelli:2018ek}. We recall that a Finsler manifold is called Zoll when all its unit-speed geodesics are closed with the same minimal period, and simple Zoll if, in addition, all the geodesics are simple closed. We denote by $\sigmas(S^2,F)$ the simple length spectrum of a Finsler 2-sphere, which is the set of lengths of its simple closed geodesics.

\begin{thm}
\label{t:LS}
Every reversible Finsler 2-sphere $(S^2,F)$ has at least three geometrically distinct simple closed geodesics. More precisely: 
\begin{itemize}
\item[(i)] If $\sigmas(S^2,F)$ is a singleton, then $(S^2,F)$ is simple Zoll.
\item[(ii)] If $\sigmas(S^2,F)$ contains exactly two elements, then there exists $\ell\in\sigmas(S^2,F)$ such that every point of $S^2$ lies on a simple closed geodesic  of length $\ell$.
\item[(iii)] Assume that, for any compact interval $[\ell_1,\ell_2]\subset(0,\infty)$, $(S^2,F)$ has only finitely many simple closed geodesics with length in $[\ell_1,\ell_2]$. Then, $(S^2,F)$ has three simple closed geodesics $\gamma_1,\gamma_2,\gamma_3$ with lengths $L(\gamma_1)< L(\gamma_2)< L(\gamma_3)$ and such that, for each $i=1,2,3$, $\gamma_i$ has non-trivial local homology in degree $i$ with $\Z_2$ coefficients.
\end{itemize}
\end{thm}

For the definition of the local homology of a closed geodesic, we refer the reader to Section~\ref{ss:local_homology}. Point (iii) in Theorem~\ref{t:LS} may look technical, but it is a crucial ingredient for the proof of Theorem~\ref{t:multiplicity} here below. Even though it is claimed in \cite{Hingston:1993ou}, it does not have a proper proof in the published literature.

\begin{rem}
As it was pointed out in \cite{Gromoll:1981ty} and \cite[Remark~5.3]{Frauenfelder:2019wh}, a Zoll reversible Finsler 2-sphere is actually simple Zoll provided it has a simple closed geodesic. This, together with Theorem~\ref{t:LS}, implies that any Zoll reversible Finsler 2-sphere is automatically simple Zoll.
\hfill\qed
\end{rem}

Finally, we can state the last result, that generalizes Bangert-Franks-Hingston's theorem.

\begin{thm}
\label{t:multiplicity}
Every reversible Finsler 2-sphere $(S^2,F)$ has infinitely many geometrically distinct closed geodesics. 
\end{thm}

The main ideas for this theorem remain the same as in the Riemannian case, but nevertheless we provide a full and rather self-contained account, which insures that certain arguments of the original proof that looked Riemannian can indeed be carried over in the Finsler case. At the same time, our treatment fills some expository gaps present in the literature.

Finally, we would like to mention a related problem that saw major advances in recent years. Closed geodesics on Riemannian surfaces are in particular minimal hypersurfaces. In 1982, Yau \cite{Yau:1982ta} conjectured that every closed Riemannian 3-manifold has infinitely many  smooth, closed, immersed minimal hypersurfaces. An even stronger statement was proved by Irie-Marques-Neves \cite{Irie:2018vd}: on any closed $n$-manifold, with $3\leq n\leq 7$, equipped with a $C^\infty$-generic Riemannian metric, the union of all smooth, closed, embedded minimal hypersurfaces is dense. We refer the reader to the survey \cite{Marques:2019ww} for more background and details.

\subsection{Organization of the paper}
In Section~\ref{s:curve_shortening} we provide a construction of the curve shortening semi-flow, and prove Theorem~\ref{t:curve_shortening}. In Section~\ref{s:simply_closed_geodesics} we prove Theorem~\ref{t:LS}, except the technical point~(iii). In Section~\ref{s:critical_point_theory}, we provide the background on the classical critical point theory for the Finsler energy function, and we will prove Theorem~\ref{t:LS}(iii) at the end of the section. Finally, in Section~\ref{s:infinitely_many} we prove Theorem~\ref{t:multiplicity}.

\subsection{Acknowledgments}
We are grateful to Egor Shelukhin and Vuka\v{s}in Stojisavljevi\'c for pointing out a mistake in the original proof of Theorem~\ref{t:Hingston}. We are also grateful to the anonymous referee, who did a tremendous work with the first draft of the manuscript: in particular, for pointing out a mistake in the original proof of Theorem~\ref{t:curve_shortening}(iv), and for providing many valuable corrections and suggestions throughout the whole manuscript.

Part of this work was carried out while Marco Mazzucchelli was visiting the Scuola Internazionale Superiore di Studi Avanzati (Trieste, Italy) in May 2018, and the Ruhr-Universit\"at Bochum (Germany) in September-November 2019. We wish to thank both institutions for providing an excellent working environment. Marco Mazzucchelli and Stefan Suhr are supported by the SFB/TRR 191 “Symplectic Structures in Geometry, Algebra and Dynamics”, funded by the Deutsche Forschungsgemeinschaft.

\section{The curve shortening semi-flow}
\label{s:curve_shortening}

\subsection{The evolution equation}
We consider a 1-parameter family of curves $\gamma_t\in\Emb(S^1,M)$ evolving according to the partial differential equation
\begin{align}
\label{e:PDE}
\partial_t \gamma_t(u)&=w_{t}(u)n_{t}(u) 
\end{align}
where $w_{t}:=V_{\gamma_t}$ and $n_t:= N_{\gamma_t}$. For every $\gamma_0\in\Emb(S^1,M)$, we denote by $\tau_{\gamma_0}\in[0,\infty]$ the largest extended real number such that there is a well defined solution $\gamma_t\in\Emb(S^1,M)$ of~\eqref{e:PDE} for all $t\in[0,\tau_{\gamma_0})$, with $\gamma_t|_{t=0}=\gamma_0$. We set
\begin{align*}
 \UU:=\big\{(t,\gamma_0)\ \big|\ \gamma_0\in\Emb(S^1,M),\ t\in[0,\tau_{\gamma_0})\big\}.
\end{align*}

\begin{thm}
\label{t:properties}
There is a unique map 
\begin{align*}
\phi:\UU\to\Emb(S^1,M),
\qquad
\phi(t,\gamma_0)=\phi_t(\gamma_0)=\gamma_t,
\end{align*}
where $\gamma_t$ is the solution of~\eqref{e:PDE} with initial condition $\gamma_0$, satisfying the following properties:

\begin{itemize}
\item[(i)] The subset $\UU\subset[0,\infty)\times\Emb(S^1,M)$ is an open neighborhood of $\{0\}\times\Emb(S^1,M)$, and $\phi$ is continuous.

\item[(ii)] The map $\phi$ is equivariant under the action of $\Diff(S^1)$ on $\Emb(S^1,M)$, i.e.\ $\phi_t(\gamma\circ\theta)=\phi_t(\gamma)\circ\theta$ for all $\gamma\in\Emb(S^1,M)$ and $\theta\in\Diff(S^1)$.

\item[(iii)] For each $\gamma\in\Emb(S^1,M)$ we have $\tfrac{\diff}{\diff t} L(\phi_t(\gamma))\leq 0$, with equality if and only if $\gamma$ is a closed geodesic of $(M,F)$.

\item[(iv)] For each $\gamma\in\Emb(S^1,M)$, if 
\[\ell_\gamma:=\lim_{t\to\tau_\gamma^-} L(\phi_t(\gamma))>0\]
then $\tau_\gamma=\infty$.
\end{itemize}
\end{thm}

The proof of this theorem will be carried out in the rest of the section: point~(i) will be proved in Subsection~\ref{ss:short_time}; point~(ii) is a consequence of Lemma~\ref{l:invariance}; point~(iii) will be proved in Subsection~\ref{ss:antigradient}. The fact that $\phi$ is well-defined as a map of the above form (i.e.\ mapping the space $\Emb(S^1,M)$ into itself) and point~(iv) will be proved in Section~\ref{ss:long_time}.
In analogy with the Riemannian case, we call $\phi_t$ the \textbf{curve shortening semi-flow} of $(M,F)$. Notice that $\phi_t$ is not a flow (despite in the Riemannian literature it is often called a flow): indeed, it is only defined for $t\geq0$, and thus satisfies $\phi_{t_1}\circ\phi_{t_2}=\phi_{t_1+t_2}$ only for $t_1,t_2\geq0$.

All closed geodesics of a closed Finsler surface $(M,F)$ have length strictly larger than the injectivity radius $\injrad(M,F)$. It is sometimes convenient to have a well defined curve shortening semi-flow defined for all positive times even for those trajectories that are not converging to a closed geodesic. We can achieve this by slowing down the curve shortening semi-flow lines in the sublevel set $\{L<\injrad(M,F)\}$, as follows. We fix a constant 
\begin{align}
\label{e:rho_0}
\rho_0>0,
\end{align}
which will be chosen smaller than $\injrad(M,F)$ in the applications. We consider a monotone increasing smooth function $\chi:[0,\infty)\to[0,1]$ such that $\mathrm{supp}(\chi)=[\rho_0,\infty)$ and $\chi(\ell)=1$ for all $\ell\in[2\rho_0,\infty)$. We define 
\begin{align*}
\psi:[0,\infty)\times\Emb(S^1,M)\to\Emb(S^1,M),
\qquad
\psi(t,\gamma_0)=\psi_t(\gamma_0)=\gamma_t, 
\end{align*}
where $\gamma_t$ is the solution of the partial differential equation
\begin{align}\label{e:PDE_slowed_down}
\partial_t \gamma_t(u)=\chi(L(\gamma_t))V_{\gamma_t}(u)N_{\gamma_t}(u)
\end{align}
The semi-flow $\psi_t$ is the one that we employ for Theorem~\ref{t:curve_shortening}. Its properties, except Theorem~\ref{t:curve_shortening}(iv) and (v), will be direct consequences of the above Theorem~\ref{t:properties} by means of the following lemma.

\begin{lem}
There exists a smooth function 
$T:\Emb(S^1,M)\times[0,\infty)\to[0,\infty)$ 
monotone increasing in the second variable such that $T(\gamma,\cdot)<\tau_\gamma$ and  
\[\psi_t(\gamma)=\phi_{T(\gamma,t)}(\gamma),\qquad \forall \gamma\in\Emb(S^1,M),\ t\in[0,\infty).\] 
Moreover
\begin{itemize}
\item[(i)] $T(\gamma,t_1+t_2)=T(\phi_{T(\gamma,t_1)}(\gamma),t_2)$,

\item[(ii)] $T(\gamma,t)=t$ if $L(\phi_t(\gamma))\geq2\rho_0$,

\item[(iii)] $T(\gamma,t)=0$ if $L(\gamma)\leq\rho_0$,

\item[(iv)] $T(\gamma\circ\theta,t)=T(\gamma,t)$ for all $\theta\in\Diff(S^1)$.
\end{itemize}
\end{lem}

\begin{proof}
We denote $\gamma_0:=\gamma$ and $\gamma_t:=\phi_t(\gamma_0)$. The smooth map $(s,t)\mapsto\gamma_{T(\gamma,t)}(s)$ is a solution of~\eqref{e:PDE_slowed_down} if and only if
\begin{align*}
\chi(L(\gamma_{T(\gamma,t)}))V_{\gamma_{T(\gamma,t)}}N_{\gamma_{T(\gamma,t)}}
=
\partial_t \gamma_{T(\gamma,t)}
=
(\partial_t T(\gamma,t))V_{\gamma_{T(\gamma,t)}}N_{\gamma_{T(\gamma,t)}}.
\end{align*}
Therefore, the desired function $t\mapsto T(\gamma,t)$ is a solution of the ordinary differential equation
\begin{equation}
\label{e:ODE_T}
\begin{split}
\partial_t T(\gamma,t)
&=
\chi(L(\gamma_{T(\gamma,t)})),\\
T(\gamma,0)&=0.
\end{split}
\end{equation}
This readily implies that $T$ is smooth as a function of $(\gamma,t)$, and not decreasing. Point~(i) readily follows from the semi-flow property $\phi_{t_1+t_2}=\phi_{t_1}\circ\phi_{t_2}$ of the curve shortening.
If $L(\gamma_{T(\gamma,t)})\geq2\rho_0$, then $L(\gamma_{T(\gamma,t')})\geq2\rho_0$ and $\chi(L(\gamma_{T(\gamma,t')}))=1$ for all $t'\in[0,t]$, which implies point~(ii). If $L(\gamma)\leq\rho_0$, then $L(\gamma_t)\leq\rho_0$ and $\chi(L(\gamma_{T(\gamma,t)}))=0$ for all $t\in(0,\tau_\gamma)$, which implies point~(iii). Finally, if we set $T_\theta(\gamma,t):=T(\gamma\circ\theta,t)$ for some $\theta\in\Diff(S^1)$, we readily see that $T_\theta$ is also a solution of the ordinary differential equation~\eqref{e:ODE_T}. Since such equation has a unique solution, we have point~(iv).
\end{proof}

The function $V_\gamma$ is a generalization of the Riemannian curvature of immersed curves in oriented Riemannian surfaces. Theorem~\ref{t:properties}(ii) readily follows from the following statement.

\begin{lem}
\label{l:invariance}
For each $\theta\in\Diff(S^1)$, we have 
\[N_{\gamma\circ\theta}=\mathrm{sign}(\dot\theta)N_\gamma\circ\theta,
\qquad
V_{\gamma\circ\theta}=\mathrm{sign}(\dot\theta)V_\gamma\circ\theta.\]
\end{lem}

\begin{proof}
The statement concerning the normal vector $N_\gamma$ is clear.
Since the Finsler metric $F$ is 1-homogeneous in the fibers $\Tan_xM$, we have  $F_v(x,\lambda v)=F_v(x,v)$ for all $\lambda>0$.
Moreover, we have $F_x(x,\lambda v)=\lambda F_x(x,v)$, $F_{xv}(x,\lambda v)=F_{xv}(x,v)$, $F_{vv}(x,v)=\lambda F_{vv}(x,\lambda v)$. Therefore, if we set $r=\theta(u)$, 
\begin{align*}
V_{\gamma\circ\theta}(u)
:=\,
&
\frac{\big(\tfrac{\diff}{\diff u}F_v(\gamma(\theta(u)),\dot\gamma(\theta(u))) - \dot\theta(u)F_x(\gamma(\theta(u)),\dot\gamma(\theta(u))) \big)N_{\gamma\circ\theta}(u)}{\|\dot\gamma(\theta(u))\|\,|\dot\theta(u)|}
\\
=\,
&
\frac{\dot\theta(u)\big( \tfrac{\diff}{\diff r}F_v(\gamma(r),\dot\gamma(r)) - F_x(\gamma(r),\dot\gamma(r)) \big)N_{\gamma\circ\theta}(u)}{\|\dot\gamma(r)\|\,|\dot\theta(u)|}
\\
=\,
&
\mathrm{sign}(\dot\theta(u))\frac{\big(\tfrac{\diff}{\diff r}F_v(\gamma(r),\dot\gamma(r)) - F_x(\gamma(r),\dot\gamma(r)) \big)N_{\gamma}(r)}{\|\dot\gamma(r)\|}
\\
=\,
&
\mathrm{sign}(\dot\theta(s))V_{\gamma}(\theta(u)).
\qedhere
\end{align*}
\end{proof}

\subsection{The anti-gradient of the length}
\label{ss:antigradient}

The space $\Emb(S^1,M)$, equipped with the $C^\infty$ topology, is a Fr\'echet manifold (indeed, it is an open subset of the Fr\'echet manifold $C^\infty(S^1,M)$). The tangent space $\Tan_\gamma\Emb(S^1,M)$ is precisely the space of smooth 1-periodic vector field $X$ along $\gamma$. The length function 
\[
L:\Emb(S^1,M)\to(0,\infty),\qquad 
L(\gamma) = \int_0^1 F(\gamma(u),\dot\gamma(u))\,\diff u
\]
is Gateaux differentiable (it is actually smooth, but we will not need it throughout this paper). Its differential can be computed as
\begin{align}
\label{e:diff_L}
\diff L(\gamma)X
=
\int_0^1 \big(F_x(\gamma(u),\dot\gamma(u))-\tfrac{\diff}{\diff u}F_v(\gamma(u),\dot\gamma(u)) \big)\,X(u)\,\diff u.
\end{align}

\begin{lem}
\label{l:dot_gamma}
For each $a\in C^{\infty}(S^1,\R)$, we have $\diff L(\gamma)a\dot\gamma=0$.
\end{lem}
\begin{proof}
If we set $\gamma_\epsilon(u):=\gamma(u+\epsilon a(u))$, we have $a\dot\gamma=\partial_\epsilon\gamma_\epsilon|_{\epsilon=0}$. Since, for all $|\epsilon|$ small enough, $\gamma_\epsilon$ is an embedded curve obtained by reparametrization of $\gamma$, we have $L(\gamma_\epsilon)=L(\gamma)$ and
$\diff L(\gamma)a\dot\gamma=\tfrac{\diff}{\diff \epsilon}\big|_{\epsilon=0}L(\gamma_\epsilon)=0$.
\end{proof}
The Riemannian metric $g$ introduces an $L^2$ Riemannian metric on $\Emb(S^1,M)$ given by
\begin{align}\label{e:L2_inner_product}
\llangle X,Y\rrangle_\gamma
=
\int_{S^1}
g(X(u),Y(u))\|\dot\gamma(u)\|\,\diff u,\qquad\forall X,Y\in\Tan_\gamma\Emb(S^1,M).
\end{align}
Thanks to the factor $\|\dot\gamma(u)\|$ in the integrand, the inner product is invariant under the action of $\Diff(S^1)$, i.e.
\begin{align}
\label{e:invariance}
\llangle X\circ\theta,Y\circ\theta\rrangle_{\gamma\circ\theta}
=
\llangle X,Y\rrangle_{\gamma},
\qquad
\forall \theta\in\Diff(S^1).
\end{align}
We denote by $\nabla L$ the gradient of the length functional with respect to this inner product. Namely, $\nabla L(\gamma)$ is the 1-periodic vector field along $\gamma$ defined by
\begin{align*}
 \diff L(\gamma)X = \llangle \nabla L(\gamma),X\rrangle_\gamma.
\end{align*}
\begin{lem}
\label{l:gradient}
$\nabla L(\gamma)=-V_{\gamma}N_{\gamma}$.
\end{lem}

\begin{proof}
Consider an arbitrary  $X\in\Tan_\gamma\Emb(S^1,M)$, which we can uniquely write as $X(u)=a(u)\dot\gamma+b(u)N_\gamma$, where 
$b(u)=g(X(u),N_\gamma(u))$.
By Lemma~\ref{l:dot_gamma} and Equation~\eqref{e:diff_L}, we compute
\begin{align*}
\diff L(\gamma)X
&=
\diff L(\gamma)a\dot\gamma+\diff L(\gamma)bN_\gamma
=
\diff L(\gamma)bN_\gamma\\
&=
\int_0^1 \big(F_x(\gamma(u),\dot\gamma(u))-\tfrac{\diff}{\diff u}F_v(\gamma(u),\dot\gamma(u)) \big)\,b(u)\,N(u)\,\diff u\\
&=
\int_0^1  g( -V_\gamma(u)N_\gamma(u), X(u))\,\|\dot\gamma(u)\|\,\diff u. 
\qedhere
\end{align*}
\end{proof}
Therefore, the curve shortening equation \eqref{e:PDE} can be seen as the anti-gradient flow equation of $L$ associated to the $L^2$-Riemannian metric on $\Emb(S^1,M)$, i.e.
\begin{align}\label{e:anti_grad_flow}
\partial_t\gamma_t=-\nabla L(\gamma_t).
\end{align} 
The invariance~\eqref{e:invariance}, together with Lemma~\ref{l:gradient}, provides an alternative proof of Lemma~\ref{l:invariance}. Moreover, if a solution $\gamma_t$ is well defined for $t\in[a,b]$, then
\begin{align}
\label{e:length_decreasing}
L(\gamma_a)-L(\gamma_b)
=
\int_a^b \|\nabla L(\gamma_t)\|^2\diff t
=
\int_{a}^b \!\!\int_{S^1} V_{\gamma_t}(u)^2\,\|\dot\gamma_t(u)\|\,\diff u\,\diff t.
\end{align}
It is well known that the closed geodesics of $(M,F)$ are critical points of $L$, that is, those $\gamma$ such that $V_\gamma\equiv0$. Therefore, $\partial_t L(\gamma_t)\leq0$ with equality if and only if $\gamma_t$ is a closed geodesic of $(M,F)$. This settles Theorem~\ref{t:properties}(iii).

\begin{rem}[Alternative curve shortening]
The PDE \eqref{e:PDE} of the curve shortening is not canonically associated to the Finsler metric $F$, as it also involves the auxiliary Riemannian metric $g$. This choice of curve shortening semi-flow turns out to be the most convenient for the later computations. Alternatively, one could also study a curve shortening semi-flow whose definition does not involve an auxiliary Riemannian metric: this is done by replacing, in~\eqref{e:anti_grad_flow}, the gradient $\nabla$ with the one induced by the following Riemannian metric on $\Emb(S^1,M)$
\begin{align*}
\llangle X,Y\rrangle_\gamma' = \int_{S^1}  F(\gamma(u),\dot\gamma(u)) (\tfrac12 F^2)_{vv}(\gamma(u),\dot\gamma(u))[X(u),Y(u)]\,\diff u,\\
\forall X,Y\in\Tan_\gamma\Emb(S^1,M).
\end{align*}
For each $v\in\Tan_qM$, we define $v^F$ to be the positive orthogonal to $v$ with respect to the inner product $(F^2)_{vv}(q,v)[\cdot,\cdot]$ with norm $F(q,v^F)=F(q,v)$. If we set 
\begin{align*}
 Z_{\gamma}(u):=\frac{\big(\tfrac{\diff}{\diff u}F_v(\gamma(u),\dot\gamma(u))-F_x(\gamma(u),\dot\gamma(u)) \big)\dot\gamma(u)^F}{F(\gamma(u),\dot\gamma(u))},
\end{align*}
the alternative curve shortening semi-flow is precisely given by
\[
\tag*{\qed}
\partial_t \gamma_t(u)= \frac{Z_{\gamma_t}(u)}{F(\gamma_t(u),\dot\gamma_t(u))} \dot\gamma_t(u)^F.
\]
\end{rem}

\subsection{Short-time existence} 
\label{ss:short_time}
In order to prove Theorem~\ref{t:properties}(i), it is convenient to work in suitable local coordinates around a fixed curve $\gamma_0\in\Emb(S^1,M)$. We denote by $\exp:\Tan M\to M$ the exponential map of $(M,g)$. There exists $\rho>0$ and an open set $U\subset M$ of $\gamma_0(S^1)$ such that the map 
\begin{align*}
\xi:S^1\times(-\rho,\rho)\to U,\qquad
\xi(u,r)=\exp_{\gamma_0(u)}(r\,N_{\gamma_0}(u))
\end{align*}
is a diffeomorphism. 

We define the smooth map
\begin{align*}
\Xi: C^\infty(S^1,(-\rho,\rho))\to\Emb(S^1,M),\qquad
\Xi(z)(u)=\xi(u,z(u)).
\end{align*}
Let us show that this map is open and injective. We first define the vector field $N$ on $U$ by
\begin{align*}
 N(\xi(u,r))= \frac{\diff}{\diff r}\xi(u,r) = \diff \exp_{\gamma_0(u)}(r\,N_{\gamma_0}(u))N_{\gamma_0}(u),
\end{align*}
and notice that $\|N(q)\|=1$ for all $q\in U$. Thus, we have
$\diff\Xi(z)w=W$,
where
\[W(u)=w(u)N(\Xi(z)(u)),\]
and this latter vector field  along $\Xi(z)$ is non-zero provided the function $w$ is non-zero. Hence $\Xi$ is an immersion. Clearly, $\Xi$ is injective, for $\xi$ is a diffeomorphism. Finally, the equality
\begin{align*}
\dist(\Xi(z)(u),\gamma_0(u))
=|z(u)|,\qquad
\forall z\in C^\infty(S^1,(-\rho,\rho)),\ u\in S^1
\end{align*}
implies that $\Xi$ is an open map onto its image. 

Since $\Diff(S^1)$ acts freely on $\Emb(S^1,M)$, the map 
\begin{gather*}
\Psi:C^\infty(S^1,(-\rho,\rho))\times\Diff(S^1)\to\Emb(S^1,M),\\
\Psi(z,\theta)(u)=\Xi(z)(\theta(u))=\exp_{\gamma_0(\theta(u))}\big(z(\theta(u)) N_{\gamma_0}(\theta(u))\big)
\end{gather*}
is open and injective onto a neighborhood of $\gamma_0$. The differential of $\Psi$ is given by
\[\diff\Psi(z,\theta)(v,\tau)=V,\]
where
\begin{align*}
V(u)=v(\theta(u))N(\Psi(z,\theta)(u)) + \tau(\theta(u)) \Xi(z)^{\cdot}(\theta(u)).
\end{align*}
Here, we have denoted $\Xi(z)^\cdot(u):=\tfrac{\partial}{\partial u}\Xi(z)(u)$

The map $\Psi$ pulls-back the $L^2$ inner product~\eqref{e:L2_inner_product} to
\begin{align*}
&\lllangle (v,\tau),(w,\sigma)\rrrangle_{(z,\theta)}\\
&\quad:= \llangle \diff\Psi(z,\theta)(v,\tau),\diff\Psi(z,\theta)(w,\sigma)\rrangle_{\Psi(z,\theta)}\\
 &\quad=  \int_{S^1} 
\Big( v(u)\,w(u) + v(u)\,\sigma(u)\, a_z(u) + w(u)\,\tau(u)\, a_z(u)\\
 &\qquad\qquad + \tau(u)\,\sigma(u)\, b_z(u)^2\Big)\,b_z(u)\,\diff u,
\end{align*}
where
\begin{align*}
a_z(u) & :=g( N(\Xi(z)(u)),\Xi(z)^\cdot(u)),\\
b_z(u) & :=\|\Xi(z)^\cdot(u)\|.
\end{align*}
Notice that this inner product is actually independent of $\theta\in\Diff(S^1)$, and therefore we will simply write
\begin{align}\label{e:local_L2}
 \lllangle (v,\tau),(w,\sigma)\rrrangle_{(z,\theta)}
 =
 \lllangle (v,\tau),(w,\sigma)\rrrangle_{z}
\end{align}

In order to write expressions in local coordinates, let us pull-back the Finsler metric $F$ by $\xi$. We obtain the Finsler metric $G:=\xi^*F$ on $S^1\times(-\epsilon,\epsilon)$ given by
\begin{align*}
G(q,v)=F(\xi(q),\diff\xi(q)v),\qquad\forall q\in S^1\times(-\rho,\rho),\ v\in\R^2.
\end{align*}
The composition of the length functional $L$ with $\Psi$ reads
\begin{align*}
L\circ\Psi(z,\theta)
=
L\circ\Xi(z)
=
\int_{S^1}
F\big(\tfrac{\diff}{\diff u}\Xi(z)(u)\big)\,\diff u
=
\int_{S^1}
G(\underbrace{(u,z(u))}_q,\underbrace{(1,\dot z(u))}_v)\,\diff u.
\end{align*}
Let us compute the derivative
\begin{align}
\label{e:derivative_L_Phi}
\begin{split}
\diff (L\circ\Xi)(z)w
&=
\int_{S^1}
\big(G_{q_2}\,w+\partial_{v_2}G\,\dot w\big)\,\diff u\\
&=
\int_{S^1}
\big(G_{q_2}-G_{q_1v_2}-G_{q_2v_2}\,\dot z-G_{v_2v_2}\,\ddot z\big) w\,\diff u. 
\end{split}
\end{align}
We denote by $(v_z,\tau_z):=\nabla(L\circ\Psi)(z)$ the gradient of $L\circ\Psi$ with respect to the inner product~\eqref{e:local_L2}, i.e.
\begin{align*}
\lllangle (v_z,\tau_z),(w,\sigma)\rrrangle_{z}=\diff(L\circ\Psi)(z,\theta)(w,\sigma).
\end{align*}
Since $L\circ\Psi(z,\theta)$ is independent of $\theta\in\Diff(S^1)$, we have
\begin{align*}
0  
= \lllangle(v_z,\tau_z),(0,\sigma)\rrrangle_{z}
=  \int_{S^1} \Big(v_z(u)\, a_z(u)\, b_z(u)
+  \tau_z(u)\,b_z(u)^3\Big)\,\sigma(u)\,\diff u,
\end{align*}
which implies that
\begin{align*}
\tau_z(u)= - v_z(u)\,\frac{a_z(u)}{b_z(u)^2},
\qquad\forall u\in S^1.
\end{align*}
On the other hand, we have
\begin{equation}
\begin{split}\label{e:v_z}
\lllangle(v_z,\tau_z),(w,0)\rrrangle_{z}
&=
\int_{S^1} \Big(v_z(u)+\tau_z(u)\,a_z(u)\Big)\,b_z(u)\,w(u)\,\diff u\\
&=
\int_{S^1} \Big( 1-\tfrac{a_z(u)^2}{b_z(u)^2} \Big)v_z(u)\,b_z(u)\,w(u)\,\diff u\\
&=
\diff (L\circ\Xi)(z)w.
\end{split} 
\end{equation}
Notice that the quotient $a_z(u)/b_z(u)$ is well defined. Indeed, the curve $s\mapsto \Xi(z)(u)$ is transverse to the vector field $N$, and therefore
\begin{align*}
\tfrac{a_z(u)^2}{b_z(u)^2}
=
g\big( N(\Xi(z)(u)),\tfrac{\Xi(z)^\cdot(u)}{|\Xi(z)^\cdot(u)|_{\Xi(z)(u)}}\big)^2
<1.
\end{align*}
Equations~\eqref{e:derivative_L_Phi} and~\eqref{e:v_z} imply that
\begin{align*}
v_z= \Big( 1-\tfrac{a_z^2}{b_z^2} \Big)^{-1}\,b_z(u)^{-1} \big(
G_{q_2}-G_{q_1v_2}-G_{q_2v_2}\,\dot z-G_{v_2v_2}\,\ddot z\big).
\end{align*}
The integral curves of the anti-gradient $-\nabla(L\circ\Xi)$ are solutions \[(z,\theta):[0,T)\times S^1\to (-\rho,\rho)\times S^1\] of the partial differential equation 
\begin{align}\label{e:curve_shortening_local}
\partial_t(z,\theta)=(-v_z,v_z a/b^2). 
\end{align}
In particular, $z$ is a solution of the partial differential equation
\begin{equation}\label{e:curve_shortening_z}
\begin{split}
\partial_t z
=\,&
\tfrac{b_z}{b_z^2-a_z^2}
\Big(
G_{v_2v_2}\,\partial^2_u z
+
G_{q_2v_2}\,\partial_u z
+
G_{q_1v_2}
-
G_{q_2}\Big).
\end{split}
\end{equation}
Since $G$ is a Finsler metric, the second derivative $G_{vv}(q,v)$ is positive semidefinite and its kernel is generated by $v$. Therefore, $G_{v_2v_2}((u,z(u)),(1,\dot z(u)))\neq0$, and~\eqref{e:curve_shortening_z} is a parabolic partial differential equation. The local theory for this class of equations (see, e.g., \cite{Mantegazza:2012zp}) provides the following statement.

\begin{prop}
\label{p:local_existence_normal}
For each $z_0\in C^\infty(S^1,(-\rho,\rho))$ there exists $\epsilon>0$ and a unique smooth solution $z:[0,\epsilon)\times S^1\to (-\rho,\rho)$ of~\eqref{e:curve_shortening_z} such that $z(0,\cdot)=z_0$. Moreover, $z$ depends continuously on the initial condition $z_0$ in the $C^\infty$ topology.
\hfill\qed
\end{prop}

Assume that $z:[0,\epsilon)\times S^1\to (-\rho,\rho)$ is the smooth solution given by Proposition~\ref{p:local_existence_normal}. Up to reducing $\epsilon>0$, we can easily find a smooth $\theta:[0,\epsilon)\times S^1\to S^1$ such that $(z,\theta)$ is a solution of the curve shortening equation~\eqref{e:curve_shortening_local} with $\theta(0,\cdot)=\mathrm{id}$. Indeed, for each $s\in S^1$, such a $\theta$ is the unique smooth solution of the ordinary differential equation
\begin{align*}
\partial_t\theta(t,s)=-\tau_z(\theta(t,s)).
\end{align*}
The smooth map
\begin{align*}
\gamma:[0,\epsilon)\times S^1\to M,
\qquad
\gamma(t,s)=\Psi(z(t,\cdot),\theta(t,\cdot))(s)=\xi(\theta(t,s),z(t,\theta(t,s)))
\end{align*}
is thus the unique smooth solution of the curve shortening equation~\eqref{e:PDE} with $\gamma(0,\cdot)=\Xi(z)=\gamma_0$. Summing up, we have proved the following statement, which implies Theorem~\ref{t:properties}(i).

\begin{thm}[Local existence and uniqueness]
For each $\gamma_0\in \Emb(S^1,M)$ there exists $\epsilon>0$ and a unique smooth solution $\gamma:[0,\epsilon)\times S^1\to M$ of the curve shortening equation~\eqref{e:PDE} such that $\gamma(0,\cdot)=\gamma_0$. Moreover, $\gamma$ depends continuously on the initial condition $\gamma_0$ in the $C^\infty$ topology.
\hfill\qed
\end{thm}

\subsection{Long-time existence}
\label{ss:long_time}
We denote by $SM$ the unit tangent bundle of $M$ with respect to the auxiliary Riemannian metric $g$, i.e.
\begin{align}
\label{e:SM}
SM=\big\{(x,v)\in \Tan M\  \big|\ \|v\|=1  \big\}.
\end{align}
In order to prove that $\phi$ is well-defined as a map onto $\Emb(S^1,M)$, and that there is long-time existence of solutions of the curve shortening equation (Theorem~\ref{t:properties}(iv)), it suffices to show that the factor $w_t$ in the right-hand side of~\eqref{e:PDE} can be expressed by means of a suitable smooth function  
\begin{align*}
 V:\R\times SM\to\R,
 \qquad
 V(\kappa_t(u),\gamma_t(u),\tau_t(u)),
\end{align*}
and invoke the general results of Angenent \cite{Angenent:1990aa} and Oaks \cite{Oaks:1994aa}. Here, $\kappa_t$ denotes the Riemannian curvature of $\gamma_t$ measured with respect to the auxiliary Riemannian metric $g$, and $\tau_t(u):=\dot\gamma_t(u)/\|\dot\gamma_t(u)\|$ its unit tangent vector. By expanding the definition of $w_t$, we have
\begin{align*}
w_t
& =
\frac{\big(\tfrac{\diff}{\diff u}F_v(\gamma_t,\dot\gamma_t)-F_x(\gamma_t,\dot\gamma_t) \big)n_t}{\|\dot\gamma_t\|}\\
& =
F_{vv}(\gamma_t,\tau_t)[\ddot\gamma_t/\|\dot\gamma_t\|^2,n_t ] + F_{xv}(\gamma_t,\tau_t)[\tau_t,n_t] - F_x(\gamma_t,\tau_t)n_t.
\end{align*}
Since $F_{vv}(x,v)v=0$, the first summand in the last line can be rewritten as
\begin{align*}
F_{vv}(\gamma_t,\tau_t)[\ddot\gamma_t/\|\dot\gamma_t\|^2,n_t ]
&=
F_{vv}(\gamma_t,\tau_t)[n_t,n_t]\, g(\ddot\gamma_t/\|\dot\gamma_t\|^2,n_t)\\
&=
F_{vv}(\gamma_t,\tau_t)[n_t,n_t] \kappa_t
- F_{vv}(\gamma_t,\tau_t)[n_t,n_t] g(\Gamma_{\gamma_t}[\tau_t,\tau_t],n_t).
\end{align*}
Here, 
\[\Gamma_x[v,w]=\Gamma_{ij}^k(x)v_iw_j\partial_{x_k},\]
where $\Gamma_{ij}^k$ are the Christoffel symbols of the metric $g$ with respect to the local coordinates employed in the above expression. Inserting this into the expression of $w_t$, we obtain
\begin{align*}
w_t
& =
F_{vv}(\gamma_t,\tau_t)[n_t,n_t] \kappa_t - F_{vv}(\gamma_t,\tau_t)[n_t,n_t] g(\Gamma_{\gamma_t}[\tau_t,\tau_t],n_t)\\
& \quad 
+  F_{xv}(\gamma_t,\tau_t)[\tau_t,n_t] - F_x(\gamma_t,\tau_t)n_t.
\end{align*}
Notice that the first summand $F_{vv}(\gamma_t,\tau_t)[n_t,n_t] \kappa_t$ is well defined independently of the local coordinates, as $F_{vv}$ is simply the fiberwise Hessian of $F$. Therefore, since $w_t$ is also well defined, the remaining summands 
\[-F_{vv}(\gamma_t,\tau_t)[J\tau_t,J\tau_t] g(\Gamma_{\gamma_t}[\tau_t,\tau_t],J\tau_t) +F_{xv}(\gamma_t,\tau_t)[\tau_t,J\tau_t] - F_x(\gamma_t,\tau_t)\] 
are well defined independently of the local coordinates as well. The expression above shows that $w_t$ is of the form $w_t=V(\kappa_t,\gamma_t,\tau_t)$, where $V:\R\times SM\to\R$ is the smooth function
\begin{equation}
\label{e:V(kappa,x,v)}
\begin{split}
 V(\kappa,x,v)&=F_{vv}(x,v)[Jv,Jv] \kappa - F_{vv}(x,v)[Jv,Jv] g(\Gamma_{x}[v,v],Jv)\\
& \quad 
+  F_{xv}(x,v)[v,Jv] - F_x(x,v)Jv\\
&=: A(x,v)\kappa + B(x,v).
\end{split}
\end{equation}
The reversibility of $F$ readily imply that $V(\kappa,x,v)=-V(-\kappa,x,-v)$. The function $V$ thus satisfies in particular the assumptions required in \cite{Angenent:1990aa,Oaks:1994aa}. By \cite[Theorem~1.3]{Angenent:1991aa}, the map $\phi$ takes values inside $\Emb(S^1,M)$. Finally, Theorem~\ref{t:properties}(iv) follows from \cite[Corollary~6.2]{Oaks:1994aa}.

\subsection{$L^\infty$ bounds on $V_\gamma$}
For any $\gamma_0\in\Emb(S^1,M)$, we will write the corresponding solution of~\eqref{e:PDE} by
\[\gamma_t=\phi_t(\gamma_0)\] 
and its length by 
\[\ell_t:=L(\gamma_t).\] 
We denote by $\nabla_t$, $\nabla_u$, and $\nabla_s$ the covariant derivatives associated with the Levi-Civita connection of $g$.
It is convenient to introduce the vector field 
\[D_s=\|\dot\gamma_t(u)\|^{-1}\partial_u\] 
on $\R\times S^1$, which acts on smooth real-valued functions $f:\R\times S^1\to\R$, $f(t,u)=f_t(u)$ by
\begin{align*}
D_sf_t(u) = \frac{\partial_u f_t(u)}{\|\dot\gamma_t(u)\|}.
\end{align*}
We recall the classical Frenet formulas from plane Riemannian geometry:
\begin{align*}
 \nabla_u\tau_t = \kappa_t \|\dot\gamma_t\|n_t,
 \qquad
 \nabla_u n_t = -\kappa_t \dot\gamma_t.
\end{align*}
By means of the PDE~\eqref{e:PDE}, we also have the following formulas.

\begin{lem}
$\nabla_t\tau_t = (D_s w_t)n_t$, $\nabla_tn_t = -(D_s w_t)\tau_t$.
\end{lem}

\begin{proof}
 Let us compute the covariant derivative $\nabla_t\tau_t$. Since $\|\tau_t\|=\|n_t\|\equiv1$, we have
\begin{align*}
g(\nabla_t\tau_t,\tau_t)=g(\nabla_t n_t,n_t)=g(\nabla_u n_t,n_t)=g(\nabla_u\tau_t,\tau_t)=0.
\end{align*}
Moreover
\begin{align*}
g(\nabla_t\dot\gamma_t, n_t) 
=
g(\nabla_u\partial_t\gamma_t,n_t)
=
g(\nabla_u(w_tn_t),n_t)
=
\dot w_t,
\end{align*}
which readily implies
\begin{align*}
\nabla_t\tau_t&=g(\nabla_t\tau_t,n_t)n_t=
\frac{\dot w_t}{\|\dot\gamma_t\|} n_t = (D_s w_t)n_t,\\
\nabla_tn_t&
=g(\nabla_t n_t,\tau_t)\tau_t
=-\frac{g( n_t, \nabla_t\dot\gamma_t)}{\|\dot\gamma_t\|}\tau_t
=
-\frac{\dot w_t}{\|\dot\gamma_t\|} \tau_t = -(D_s w_t)\tau_t.
\qedhere 
\end{align*}
\end{proof}

\begin{lem}
$\partial_t \|\dot\gamma_t(u)\|=-\kappa_t(u)w_t(u)\|\dot\gamma_t(u)\|$.
\end{lem}

\begin{proof}
By means of the commutativity $\nabla_t\partial_u=\nabla_u\partial_t$ and of the PDE~\eqref{e:PDE}, we compute
\begin{align*}
 \partial_t \|\dot\gamma_t(u)\|
 & = 
 \frac{g(\nabla_t\dot\gamma_t(u),\dot\gamma_t(u))}{\|\dot\gamma_t(u)\|}
 = 
 \frac{g(\nabla_u\partial_t\gamma_t(u),\dot\gamma_t(u))}{\|\dot\gamma_t(u)\|}
 = 
 \frac{g(\nabla_u (w_t n_t),\dot\gamma_t(u))}{\|\dot\gamma_t(u)\|}\\
 & =
 w_t(u)\frac{g(\nabla_u n_t,\dot\gamma_t(u))}{\|\dot\gamma_t(u)\|}
 =
 -w_t(u)\frac{g(\nabla_u\dot\gamma_t,n_t(u))}{\|\dot\gamma_t(u)\|}
 =-\kappa_t(u)w_t(u)\|\dot\gamma_t(u)\|.
\end{align*}
\end{proof}

\begin{lem}
The curvature $\kappa_t$ evolves according to the PDE
\begin{align*}
\partial_t \kappa_t(u) = D_s^2 w_t(u) + w_t(u)\kappa_t^2(u)+w_t(u)k_g(\gamma_t(u)),
\end{align*}
where $k_g$ denotes the Gaussian curvature of $(M,g)$, i.e.\ $k_g(x)=g(R(v,Jv)v,Jv)$ for all 
$v\in S_xM$.
\end{lem}

\begin{proof}
The lemma follows by direct computation:
\begin{align*}
\partial_t\kappa_t
&=
\partial_t\frac{g(\nabla_u\tau_t,n_t)}{\|\dot\gamma_t\|}\\
&=
\left(\partial_t \frac1{\|\dot\gamma_t\|}\right)\kappa_t\|\dot\gamma_t\| + \frac1{\|\dot\gamma_t\|} g(\nabla_t\nabla_u\tau_t,n_t) + \frac1{\|\dot\gamma_t\|} \underbrace{g(\nabla_u\tau_t,\nabla_t n_t)}_{=0}\\
&=
\frac{\kappa_t w_t \|\dot\gamma_t\|}{\|\dot\gamma_t\|^2}\kappa_t\|\dot\gamma_t\|
+ \frac1{\|\dot\gamma_t\|} g(\nabla_u\nabla_t\tau_t,n_t)
+ \frac1{\|\dot\gamma_t\|} g(R(\dot\gamma_t,\partial_t\gamma_t)\tau_t,n_t)\\
&= 
w_t\kappa_t^2 + D_s^2 w_t + w_t k_g\circ\gamma_t.
\qedhere
\end{align*}
\end{proof}

 We set $\ell_t:=L(\gamma_t)$, and denote by $\Gamma_t:\R/\ell_t\Z\to M$ the reparametrized $\gamma_t$ with unit speed with respect to the auxiliary Riemannian metric $g$. Namely $\Gamma_t(s)=\gamma_t\circ\sigma_t^{-1}(s)$, where
\begin{align*}
 \sigma_t(u) = \int_0^u \|\dot\gamma_t(r)\|\,\diff r,
\end{align*}
and therefore $\ell_t=\sigma_t(1)$.
We also set 
\begin{align}
\label{e:W_N_K}
\begin{split}
 W_t(s) &:= w_t\circ\sigma_t^{-1}(s),\\
 N_t(s) &:= n_t\circ\sigma_t^{-1}(s),\\
 K_t(s) &:= \kappa_t\circ\sigma_t^{-1}(s).
\end{split} 
\end{align}
Notice that
\begin{align*}
K_t(s)=g(\nabla_s\dot\Gamma_t,N_t),\qquad
W_t(s) = V(K_t(s),\dot\Gamma_t(s)).
\end{align*}
Moreover
\begin{align*}
\dot W_t \circ\sigma_t = D_s w_t,
\qquad
\dot K_t \circ\sigma_t = D_s \kappa_t.
\end{align*}

If $f:SM\to\R$ is a smooth function, we will denote by $\nablah f$ and $\nablav f$ the duals of the horizontal and vertical projections respectively of its gradient with respect to the Sasaki metric of $SM$ induced by $g$. These operators allow to express $\partial_t w_t$ as
\begin{align*}
\partial_tw_t 
& = 
(\partial_\kappa V) \partial_t\kappa_t + (\nablah V)\partial_t\gamma_t + (\nablav V)\nabla_t\tau_t\\
& =
(\partial_\kappa V) \big( D_s^2 w_t + w_t\kappa_t^2+w_t k_g(\gamma_t) \big) + (\nablah V)n_t\,w_t + (\nablav V)n_t\, D_sw_t.
\end{align*}
We set 
\[A(x,v):=\partial_\kappa V(x,v)=F_{vv}(x,v)[Jv,Jv].\] 
Notice that $A$ is uniformly bounded from below by a positive constant. From now on, we will consider it evaluated at $(\gamma_t(u),\tau_t(u))$. Notice that $A\, D_s^2w_t=D_s(A\, D_sw_t)-(D_sA)(D_sw_t)$, and 
\begin{align*}
D_s A
=
\tfrac{1}{\|\dot\gamma_t\|} \big((\nablah A)\dot\gamma_t + (\nablav A) \nabla_u\tau_t\big)
=
(\nablah A)\tau_t + (\nablav A) n_t\,\kappa_t.
\end{align*}
Therefore, $\partial_tw_t$ can be written as
\begin{align*}
 \partial_tw_t  
 &=
A\, D_s^2 w_t + A w_t\kappa_t^2 + B\,D_sw_t + C\, w_t\\
 &= 
 D_s(A\, D_s w_t)+   A w_t\kappa_t^2 + E\,D_sw_t + H\,\kappa_t D_sw_t + C\, w_t,
\end{align*}
where $B$, $C$, $E$, and $H$ are smooth functions on $SM$ evaluated at $(\gamma_t(u),\tau_t(u))$.

We are now going to employ the open sets $\UU(\ell,\epsilon)$ defined in~\eqref{e:U_ell_epsilon}.

\begin{lem}
\label{l:L2_estimates}
For all $\ell>2\rho_0$ there exists a constant $c\geq1$ with the following properties: for all $\epsilon>0$ small enough,  $\gamma_0\in\Emb(S^1,M)$, and $t\geq0$ such that 
\begin{align*}
 \|W_0\|_{L^2} \leq \epsilon,
 \qquad
 \ell-\epsilon^2 \leq \ell_t \leq \ell_0 \leq \ell+\epsilon^2,
\end{align*}
we have $\|W_t\|_{L^2}\leq c\,\epsilon$.
\end{lem}

\begin{proof}
We require $\epsilon>0$ to be small enough so that $\ell-\epsilon^2>2\rho_0$, where $\rho_0>0$ is the constant fixed in~\eqref{e:rho_0}.
We consider an arbitrary $\gamma_0\in\Emb(S^1,M)$ such that $\|W_0\|_{L^2}\leq \epsilon$ and $\ell_0\in(\ell-\epsilon^2,\ell+\epsilon^2)$, and its evolution $\gamma_t$. We recall that the corresponding $w_t$ has the form $w_t=V(\kappa_t,\gamma_t,\tau_t)$, where $\tau_t:=\dot\gamma_t/\|\dot\gamma_t\|$. 

We compute
\begin{align*}
\partial_t \|W_t\|_{L^2}^2
&=
\int_0^1 \Big( 2 w_t \,(\partial_t w_t) \, \dot\sigma_t + w_t^2 (\partial_t\dot \sigma_t) \Big)\diff u\\
&=
\int_0^1 \Big(
2  w_t\,D_s(AD_s w_t) +  2 A w_t^2\kappa_t^2 + 2E\,w_t\,D_sw_t + 2H\,\kappa_t w_t D_sw_t\\
&\qquad\qquad + 2C w_t^2  
- \kappa_t w_t^3
\Big)
\dot\sigma_t\,\diff u\\
&=
\int_0^{\ell_t} \!\! \Big(
-2 A (\dot W_t)^2 +  2 A W_t^2K_t^2 + 2EW_t \dot W_t + 2H K_t W_t\dot W_t\\
&\qquad\qquad + 2C W_t^2  
- K_t W_t^3
\Big)
\diff s.
\end{align*}
From now on, we will denote by $c\geq1$ a positive constant (independent of $\gamma_t$), that may increase along the different inequalities. The above expression for $\partial_t \|W_t\|_{L^2}^2$ readily implies
\begin{equation}
\label{e:estimate_W_t_1}
\begin{split}
\partial_t \|W_t\|_{L^2}^2
& \leq 
- c^{-1} \|\dot W_t\|_{L^2}^2 + c \,\big( \|W_t\dot W_t\|_{L^1}+\|K_t W_t\dot W_t\|_{L^1}+\|W_t^2K_t^2\|_{L^1}\\
& 
\quad\
+\|K_tW_t^3\|_{L^1}+\|W_t\|_{L^2}^2 \big). 
\end{split}
\end{equation}

By the Peter-Paul inequality, for every $\rho>0$, the term $\|W_t\dot W_t\|_{L^1}$ can be bounded as
\begin{align*}
 \|W_t\dot W_t\|_{L^1} \leq \rho^2 \|\dot W_t\|_{L^2}^2 + \tfrac{1}{4\rho^2} \|W_t\|_{L^2}^2,
\end{align*}
and the term $\|K_tW_t\dot W_t\|_{L^1}$  as
\begin{align*}
 \|K_t W_t\dot W_t\|_{L^1} 
 \leq 
 \rho^2 \|\dot W_t\|_{L^2}^2 + \tfrac{1}{4\rho^2} \|K_tW_t\|_{L^2}^2
 \leq
 \rho^2 \|\dot W_t\|_{L^2}^2 + \tfrac{1}{4\rho^2} \|K_t\|_{L^\infty}^2 \|W_t\|_{L^2}^2.
\end{align*}
We will fix a sufficiently small constant $\rho>0$ so that, in the inequality~\eqref{e:estimate_W_t_1}, the term $- c^{-1} \|\dot W_t\|_{L^2}^2$ will be able to absorb the terms $\rho^2 \|\dot W_t\|_{L^2}^2$, still producing a negative factor in front of $\|\dot W_t\|_{L^2}^2$.

Equation~\eqref{e:V(kappa,x,v)} readily implies that the curvature $K_t$ is related to $W_t$ by
$K_t=A^{-1}W_t + P$, where, once again, $P$ is a smooth function on $SM$ evaluated at $(\Gamma_t(s),\dot\Gamma_t(s))$. Therefore $\|K_t\|_{L^\infty} \leq
c\, (\|W_t\|^2_{L^\infty}+1)$.

Inserting these estimates in~\eqref{e:estimate_W_t_1}, we obtain
\begin{align*}
\partial_t \|W_t\|_{L^2}^2 \leq  c\|W_t\|_{L^2}^2  + c \|W_t\|^2_{L^\infty}\| W_t\|_{L^2}^2 - c^{-1}\|\dot W_t\|_{L^2}^2 .
\end{align*}
We require $c>0$ to be large enough so that, since $\ell-\epsilon^2\leq\ell_t\leq\ell+\epsilon^2$, we have
\[
c^{-1} \leq \ell_t \leq c.
\]
If we bound from above the term $- c^{-1}\|\dot W_t\|_{L^2}^2$ by means of the inequality
\begin{equation*}
\|W_t\|_{L^\infty}^2 
\leq 
2 \ell_t^{-1} \|W_t\|_{L^2}^2 + 2\ell_t\|\dot W_t\|_{L^2}^2 
\leq 
c \|W_t\|_{L^2}^2 + c\|\dot W_t\|_{L^2}^2,
\end{equation*}
we further obtain
\begin{align*}
 \partial_t \|W_t\|_{L^2}^2 
 \leq 
 c\| W_t\|_{L^2}^2
 +
 c \|W_t\|^2_{L^\infty} ( \|W_t\|_{L^2}^2 - c^{-1} ).
\end{align*}

We claim that, if $\epsilon>0$ is small enough (independently of  $\gamma$), then $\|W_t\|_{L^2}^2 < c^{-1}$ for all $t\geq0$ such that $\ell_t\geq \ell-\epsilon^2$. Indeed, assume that this is not the case. If $\epsilon^2<1/c$, since $\|W_0\|_{L^2}^2 \leq \epsilon^2 < c^{-1}$, there must be $\tau>0$ such that $\|W_t\|_{L^2}^2 < c^{-1}$ for all $t\in[0,\tau)$, $\|W_\tau\|_{L^2}^2 = c^{-1}$, and $\ell_\tau\geq\ell-\epsilon^2$. For all $t\in[0,\tau]$ we have the inequality $\partial_t \|W_t\|_{L^2}^2 \leq c\|W_t\|_{L^2}^2$, and thus
\begin{align*}
c^{-1} = \|W_{\tau}\|_{L^2}^2 \leq e^{c\tau} \|W_{0}\|_{L^2}^2 \leq e^{c\tau} \epsilon^2 .
\end{align*}
If $\epsilon^2\leq e^{-c} c^{-1}$, then $\tau\geq1$. Therefore, since $\|W_{\tau}\|_{L^2}^2 \leq e^{c(\tau-t)} \|W_{t}\|_{L^2}$ for all $t\in[0,\tau]$, by~\eqref{e:length_decreasing} we have
\begin{align*}
c^{-1}=\|W_{\tau}\|_{L^2}^2 
\leq 
e^{c} \int_{\tau-1}^\tau \|W_{t}\|_{L^2}^2\diff t
=
e^{c} (\ell_{\tau-1}-\ell_\tau) \leq 2\epsilon^2e^{c},
\end{align*}
which is impossible if $\epsilon^2<1/(2e^{c} c)$.

Summing up, we showed that $\partial_t\|W_t\|_{L^2}^2\leq c\|W_t\|_{L^2}^2$ provided $\ell_t\geq\ell-\epsilon^2$, and therefore 
\begin{align*}
\|W_t\|_{L^2}^2\leq e^{ct}\int_{0}^{t} \|W_r\|_{L^2}^2\,\diff r \leq 
e^{c} (\ell_0-\ell_t)\leq 2\epsilon^2e^{c},&\\
\forall t\geq0\mbox{ such that }\ell_t\geq\ell-\epsilon^2.& 
\qedhere
\end{align*}
\end{proof}

\begin{lem}
\label{l:L2_estimates_derivative}
For all $\ell>2\rho_0$ there exists a constant $c\geq1$ with the following properties:
for all $\epsilon>0$ small enough,  $\gamma_0\in\Emb(S^1,M)$, and $t\geq c\log(\|\dot W_0\|_{L^2}^2 \epsilon^{-2})$ such that 
\begin{align*}
 \|W_0\|_{L^2} \leq \epsilon,
 \qquad
 \ell-\epsilon^2 \leq \ell_t \leq \ell_0 \leq \ell+\epsilon^2,
\end{align*}
we have $\|\dot W_t\|_{L^2}\leq c\,\epsilon$.
\end{lem}

\begin{proof}
We require $\epsilon>0$ to be small enough so that Lemma~\ref{l:L2_estimates} holds, and in particular so that $\ell-\epsilon^2>2\rho_0$.
We consider an arbitrary $\gamma_0\in\Emb(S^1,M)$ such that $\|W_0\|_{L^2}\leq \epsilon$ and $\ell_0\in(\ell-\epsilon^2,\ell+\epsilon^2)$, and its evolution $\gamma_t$. Once again, we will denote by $c\geq 1$ a large enough constant independent of $\gamma_t$ and $\epsilon$, possibly growing throughout the computations.

We estimate
\begin{align*}
\partial_t \|\dot W_t\|_{L^2}^2
&=
\int_0^1 \Big(2 (D_s w_t) \partial_t (D_s w_t) \, \dot\sigma_t - (D_s w_t)^2 \kappa_t w_t\dot\sigma_t\Big)\diff u\\
&=
\int_0^1 \Big(2 (D_s w_t) \partial_t \dot w_t + (D_s w_t)^2 \kappa_t w_t\dot\sigma_t\Big)\diff u\\
&=
\int_0^1 \Big(2 (D_s w_t) D_s\big(A\, D_s^2 w_t + A w_t\kappa_t^2 + B\,D_sw_t + C\, w_t\big)\\
&\qquad\qquad + (D_s w_t)^2 \kappa_t w_t\Big) \dot\sigma_t\diff u\\
&=
\int_0^{\ell_t}\!\! \Big(-2A (\ddot W_t)^2 - 2A \ddot W_t W_t K_t^2 -2 B\ddot W_t\dot W_t - C \ddot W_t W_t \\
&\qquad\qquad + (\dot W_t)^2 K_t W_t\Big) \diff s\\
&\leq
-c^{-1} \|\ddot W_t\|_{L^2}^2 + c\big( \|\ddot W_t W_t K_t^2\|_{L^1} + \|\ddot W_t\dot W_t\|_{L^1} +\|\ddot W_t W_t\|_{L^1}\\
&\qquad\qquad + \|(\dot W_t)^2 K_t W_t\|_{L^1}  \big).
\end{align*}
By employing the Peter-Paul inequality and expressing $K_t$ as an affine function of $W_t$, we have
\begin{equation}
\label{e:terms_to_bound_dot_W_t}
\begin{split}
\partial_t \|\dot W_t\|_{L^2}^2
& \leq 
-c^{-1} \|\ddot W_t\|_{L^2}^2 + c\big( \|W_t^3\|_{L^2}^2 +\|W_t^2\|_{L^2}^2 + \|W_t\|_{L^2}^2 + \|\dot W_t\|_{L^2}^2 \\
&\qquad  + \|(\dot W_t)^2W_t^2\|_{L^1} + \|(\dot W_t)^2W_t\|_{L^1}\big)\\
& = 
-c^{-1} \|\ddot W_t\|_{L^2}^2 + c\big( \|W_t^6\|_{L^1} +\|W_t^4\|_{L^1} + \|W_t\|_{L^2}^2 + \|\dot W_t\|_{L^2}^2 \\
&\qquad  + \|(\dot W_t)^2W_t^2\|_{L^1} + \|(\dot W_t)^2W_t\|_{L^1}\big). 
\end{split}
\end{equation}
We require $c>1$ to be large enough so that $c^{-1}\leq \ell_t\leq c$.

By Lemma~\ref{l:L2_estimates}, we have $\|W_t\|_{L^2}^2<c\,\epsilon^2$ for all $t>0$ such that $\ell_t>\ell-\epsilon^2$. We introduce a large constant $d\geq 1$ that we will fix later.
We introduce the set
\begin{align*}
I
:=
\big\{t\in[0,\infty)\ \big|\ \ell_t\geq \ell-\epsilon^2,\ \|\dot W_t\|_{L^2}^2\geq d \|W_t\|_{L^2}^2 \big\}.
\end{align*}
and consider $t\in I$. By means of an integration by parts and Cauchy-Schwarz's inequality, we have  
\[
\|\dot W_t\|_{L^2}^2 
\leq
-\int_0^{\ell_t} W_t\ddot W_t\,\diff s
\leq
\|W_t\|_{L^2} \|\ddot W_t\|_{L^2}
\leq
d^{-1}
\|\dot W_t\|_{L^2} \|\ddot W_t\|_{L^2}
,
\]
and thus
\begin{align*}
 \|\dot W_t\|_{L^2}\leq d^{-1} \|\ddot W_t\|_{L^2}.
\end{align*}
We employ this inequality to bound from above the positive terms in~\eqref{e:terms_to_bound_dot_W_t} as follows.
\begin{align*}
\|W_t^4\|_{L^1}
&\leq
\|W_t\|_{L^2}^2\|W_t\|_{L^\infty}^2
\leq
c\,\epsilon^2 \big( \|W_t\|_{L^1}^2 + \|\dot W_t\|_{L^1}^2 \big)
\leq
c\,\epsilon^2 \|\dot W_t\|_{L^2}^2,\\
\|W_t^6\|_{L^1}
&\leq
\|W_t\|_{L^2}^2\|W_t\|_{L^\infty}^4
\leq
c\,\epsilon^2 \big( \|W_t\|_{L^1}^4 + \|\dot W_t\|_{L^1}^4 \big)\\
&
\leq
c\,\epsilon^2 \big( \|W_t\|_{L^1}^4 + \|W_t\ddot W_t\|_{L^1}^2 \big)
\leq
c\,\epsilon^2 \big( \|W_t\|_{L^1}^4 + \|W_t\|_{L^2}^2\|\ddot W_t\|_{L^2}^2 \big)\\
&
\leq
c\,\epsilon^4 \big( \|\dot W_t\|_{L^2}^2 + \|\ddot W_t\|_{L^2}^2 \big),\\
\|(\dot W_t)^2W_t\|_{L^1}
&\leq
\|W_t\|_{L^1} \|\dot W_t\|_{L^\infty}^2
\leq
c \|W_t\|_{L^2} \|\ddot W_t\|_{L^2}^2
\leq
c\,\epsilon \|\ddot W_t\|_{L^2}^2,\\
\|(\dot W_t)^2W_t^2\|_{L^1}
&\leq
\|W_t\|_{L^2}^2 \|\dot W_t\|_{L^\infty}^2
\leq
c\,\epsilon^2 \|\ddot W_t\|_{L^2}^2.
\end{align*}
We require $\epsilon>0$ to be small enough so that the negative term $-c^{-1} \|\ddot W_t\|_{L^2}^2$ can absorb the terms $c\,\epsilon \|\ddot W_t\|_{L^2}^2$, $c\,\epsilon^2 \|\ddot W_t\|_{L^2}^2$, $c\,\epsilon^4 \|\ddot W_t\|_{L^2}^2$, thus obtaining
\begin{align*}
\partial_t \|\dot W_t\|_{L^2}^2 
\leq -c^{-1} \|\ddot W_t\|_{L^2}^2 + c \|\dot W_t\|_{L^2}^2
\leq (-c^{-1}d + c) \|\dot W_t\|_{L^2}^2,
\qquad\forall t\in I.
\end{align*}
We now fix $d>c^2$, so that $-b:=(cd^{-1}-c^{-1})<0$ and
\begin{align}
\label{e:Gronwall}
 \partial_t \|\dot W_t\|_{L^2}^2 \leq -b \|\dot W_t\|_{L^2}^2,\qquad\forall t\in I.
\end{align}

Now, consider a time value $t\geq b^{-1}\log(\|\dot W_0\|_{L^2}^2 \epsilon^{-2})$ such that $\ell_t\geq\ell-\epsilon^2$. If $[0,t]\subset I$, then~\eqref{e:Gronwall} and Gronwall's lemma imply
\begin{align*}
 \|\dot W_t\|_{L^2}^2 
 \leq 
 e^{-bt} \|\dot W_0\|_{L^2}^2 
 \leq
 \epsilon^2.
\end{align*}
If instead $[0,t]\setminus I\neq\varnothing$, we set $t_0:=\sup ([0,t]\setminus I)$, and we have
\[
 \|\dot W_t\|_{L^2}^2 
 \leq 
 e^{-b(t-t_0)} \|\dot W_{t_0}\|_{L^2}^2
 \leq
 \|\dot W_{t_0}\|_{L^2}^2
 \leq 
 d \| W_{t_0}\|_{L^2}^2
 \leq dc^2\epsilon^2.
 \qedhere
\]
\end{proof}

The analogue of Theorem~\ref{t:curve_shortening}(iv) is well known in the classical setting of Lusternik-Schnirelmann theory: it is based on the fact that a smooth function drops of at least a fixed amount along a gradient flow line that crosses a given shell around a critical set at a given level. In the context of the Riemannian curve shortening semi-flow, the analogous property is claimed\footnote{Quoted from the last sentence in \cite[page~109]{Grayson:1989ec}: \emph{``Any curve leaving a small neighborhood of a geodesic shortens a fixed amount before moving very far.''}} by Grayson \cite{Grayson:1989ec}. We now employ the bounds of Lemmas~\ref{l:L2_estimates} and~\ref{l:L2_estimates_derivative} to provide a complete proof of Theorem~\ref{t:curve_shortening}(iv) in our general reversible Finsler setting.

\begin{proof}[Proof of Theorem~\ref{t:curve_shortening}(iv)]
Let $\ell>2\rho_0$ and $\epsilon>0$ be given. For any embedded circle $\gamma_0\in\Emb(S^1,M)$, we denote by $\gamma_t=\phi_t(\gamma_0)$ the corresponding evolution under the curve shortening semi-flow, by $\ell_t=L(\gamma_t)$ the length, and by $W_t$ the associated functions as defined in~\eqref{e:W_N_K}. We denote by $c\geq1$ the maximum among the two constants $c$ given by Lemmas~\ref{l:L2_estimates} and~\ref{l:L2_estimates_derivative}.

Assume now that $\gamma_0\in\Emb(S^1,M)$ has length $\ell_0\in(\ell-\delta,\ell+\delta)$ for some constant $\delta\in(0,\ell-2\rho_0)$ that we will fix later. We claim that there exists $t_0\in[0,2]$ such that 
\[\ell_{t_0}<\ell-\delta\quad\mbox{or}\quad \|W_{t_0}\|_{L^2}\leq\sqrt\delta.\] 
Indeed, if $\|W_t\|_{L^2}>\sqrt\delta$ for all $t\in[0,2]$, we  have
\begin{align*}
 \ell_2 
 = 
 \ell_0 - \int_{0}^2 \|W_t\|_{L^2}^2\diff t
 <
 \ell_0 - 2\delta
 <
 \ell - \delta.
\end{align*}
By Lemma~\ref{l:L2_estimates}, for each $t\geq t_0$ such that $\ell_{t}\geq\ell-\delta$, we have
\[\|W_{t}\|_{L^2}\leq c\,\sqrt\delta.\] 
In particular, since $t_0\leq2$, this holds for $t=2$. Therefore, we can apply Lemma~\ref{l:L2_estimates_derivative}: for each $t\geq 2+c\log(\|\dot W_{2}\|_{L^2}^2 c^{-2}\delta^{-1})$  such that $\ell_{t}\geq\ell-\delta$, we have
\[\|\dot W_{t}\|_{L^2}\leq c^2\sqrt\delta,\] 
and in particular
\begin{align*}
 \|W_t\|_{L^\infty}
 &
 \leq
 \|\dot W_t\|_{L^1}+\min|W_t|
 \leq
 \|\dot W_t\|_{L^1} + \ell_t^{-1}\| W_t\|_{L^1}\\
 &
 \leq
 \ell_t^{1/2}\|\dot W_t\|_{L^2} + \ell_t^{-1/2}\| W_t\|_{L^2}
 \leq \big( (\ell+\delta)^{1/2} + (\ell-\delta)^{-1/2} \big)c^2\sqrt\delta;
\end{align*}
we now fix $\delta\in(0,\ell-2\rho_0)$ small enough so that $\epsilon\geq\big( (\ell-\delta)^{-1/2} + (\ell+\delta)^{1/2} \big)c^2\sqrt\delta$, which together with the previous $L^\infty$ bound implies $\gamma_t\in\UU(\ell,\epsilon)$.
The desired positive function $\tau$ is therefore given by
\[
 \tau:\Emb(S^1,M)^{<\ell+\delta}\to(0,\infty),\qquad
 \tau(\gamma_0) = 2+c\log(\|\dot W_{2}\|_{L^2}^2 c^{-2}\delta^{-1}).
 \qedhere
\]
\end{proof}

\subsection{Compactness}
Finally Theorem~\ref{t:curve_shortening}(v) will be a consequence of the following compactness result. We denote by $\PT M\to M$ the projectivized tangent bundle of $M$, whose fiber over any $x\in M$ is the real projective space  
\[\PPP(\Tan_xM)=\frac{\Tan_xM\setminus\{0\}}{\sim},\]
where $v\sim\lambda v$ for all $v\in\Tan_xM\setminus\{0\}$ and $\lambda\in\R$.

\begin{lem}
\label{l:neighborhoods}
Let $K\subseteq \PT M$ be a compact subset. If no element of $K$ is tangent to a simple closed geodesic of $(M,F)$ of length $\ell$, then for all $\epsilon>0$ small enough no element in $K$ is tangent to some curve $\gamma\in\UU(\ell,\epsilon)$.
\end{lem}

\begin{proof}
If the Lemma does not hold, then there exists a sequence  $\gamma_n\in\UU(\ell,1/n)$ such that
$F(\gamma_n,\dot\gamma_n)\equiv L(\gamma_n)$ and $[\dot\gamma_n(0)]\in K$.
The lifted curves $(\gamma_n,\dot\gamma_n/L(\gamma_n))$ are contained in the Finsler unit tangent bundle $\{(x,v)\in\Tan M\ |\ F(x,v)=1\}$, which is a compact subset of $\Tan M$. 
We consider the function $V_\gamma$ defined in~\eqref{e:V_gamma}.
Since $\|V_{\gamma_n}\|_{L^\infty}<1/n\to 0$ as $n\to\infty$, the sequence $\gamma_n$ is bounded in the $C^2$-topology. Therefore, up to a subsequence, $\gamma_n$ converges in the $C^1$-topology to some $\gamma$ such that $F(\gamma,\dot\gamma)\equiv L(\gamma)=\ell$ and $\dot\gamma(0)\in K$. Now, consider the Finsler energy
\begin{align*}
E:\W(S^1,M)\to[0,\infty),\qquad E(\zeta)= \int_{S^1}\!\! F(\zeta(u),\dot\zeta(u))^2\,\diff u,
\end{align*}
Since each $\gamma_n$ has constant speed, we have
\begin{align*}
\diff E(\gamma_n) X & = 2 \int_{S^1} F(\gamma_n(u),\dot\gamma_n(u)) \big(F_x(\gamma_n(u),\dot\gamma_n(u))-\tfrac{\diff}{\diff u}F_v(\gamma_n(u),\dot\gamma_n(u)) \big)\,X(u)\,\diff u\\
& = 2 L(\gamma_n)\,\diff L(\gamma_n)X
= 2L(\gamma_n) \int_{S^1} V_{\gamma_n}(u)\,g( N_{\gamma_n}(u),X(u))\,\|\dot\gamma
_n(u)\|_{g}\,\diff u.
\end{align*}
This, together with $\|V_{\gamma_n}\|_{L^\infty}\to0$ and the fact that $E$ is a $C^{1,1}$ function, readily implies that the limit curve $\gamma$ is a critical point of $E$, and therefore a closed geodesic. In order to reach a contradiction, we simply have to show that $\gamma$ is simple closed.

On an orientable reversible Finsler surface, a closed geodesic that is the $C^1$-limit of embedded circles is itself simple. Indeed, $\gamma$ cannot have a transverse self-intersection, for the same would be true for $\gamma_n$ for $n$ large enough. Moreover, $\gamma$ cannot have a self-tangency with opposite orientation, i.e. of the form $\gamma(u_1)=\gamma(u_2)$ and $\dot\gamma(u_1)=-\dot\gamma(u_2)$ for some $u_1<u_2$; indeed, since $F$ is reversible, we would have $\dot\gamma(u_1+r)=\dot\gamma(u_2-r)$ for all $r>0$, and then $\dot\gamma(\tfrac{u_1+u_2}2)=0$, contradicting the fact that $\gamma$ is a geodesic. Finally, $\gamma$ cannot be an iterated curve, i.e.\ of the form $\gamma(u)=\zeta(mu)$ for some simple closed geodesic $\zeta:S^1\to M$ and $m\geq2$; otherwise, since $M$ is an orientable surface, a tubular neighborhood of $\zeta$ would be diffeomorphic to the annulus $S^1\times(-1,1)$, $\zeta$ being its zero section $S^1\times\{0\}$; any closed curve sufficiently $C^1$-close to $\gamma$ would wind $m\geq2$ times around the annulus $S^1\times(-1,1)$, and therefore would have self-intersections.
\end{proof}

\begin{proof}[Proof of Theorem~\ref{t:curve_shortening}(v)]
By choosing $K=\PT M$ in Lemma~\ref{l:neighborhoods}, we infer that, for each $\ell\in[\ell_1,\ell_2]$, there exists $\epsilon>0$ such that $\UU(\ell,\epsilon)=\varnothing$;  notice that this readily implies that $\UU(\ell',\epsilon/\sqrt2)=\varnothing$ for all $\ell'\in[\ell-\tfrac12\epsilon^2,\ell+\tfrac12\epsilon^2]$.  Theorem~\ref{t:curve_shortening}(v) is a direct consequence of this fact, together with Theorem~\ref{t:curve_shortening}(iv) and the compactness of the interval $[\ell_1,\ell_2]$.
\end{proof}

\section{Existence of simple closed geodesics}
\label{s:simply_closed_geodesics}

\subsection{Lusternik-Schnirelmann theory}
Let $(M,F)$ be a closed, orientable, reversible, Finsler surface. We consider the space of embedded loops $\Emb(S^1,M)$ and the space of circle diffeomorphisms $\Diff(S^1)$, both endowed with the $C^\infty$ topology. We introduce the space of unparametrized embedded loops
\begin{align*}
 \Pi=\frac{\Emb(S^1,M)}{\Diff(S^1)},
\end{align*}
endowed with the quotient topology.
Here, $\Diff(S^1)$ acts by reparametrization on $\Emb(S^1,M)$. 
From now on, the length functional~\eqref{l:length} will be considered as a continuous function on $\Pi$, i.e.
\begin{align*}
L:\Pi\to[0,\infty).
\end{align*}
For any subset $\WW\subset\Pi$ and $\ell\in(0,\infty]$, we set 
\begin{align}
\label{e:W<ell}
\WW^{<\ell}:=\{ \gamma\in\WW\ |\ L(\gamma)<\ell \}. 
\end{align}
Throughout this paper, we shall denote by $H_*(\,\cdot\,;\F)$ the singular homology with coefficients in a field $\F$; we shall remove $\F$ from the notation whenever the arguments will not require a specific field. If $\sigma$ is a singular chain in $\Pi$, we denote by $\supp(\sigma)$ its support, which is a compact subset of $\Pi$. 
Each non-zero homology class $h\in H_*(\Pi^{<b},\Pi^{<a})$, where $0<a<b\leq\infty$, defines a min-max value
\begin{align*}
 \ell(h) := \inf_{[\sigma]=h}\ \max L|_{\supp(\sigma)}\in[a,b).
\end{align*}
Such value turns out to be the (positive) length of a simple closed geodesic of $(M,F)$. This will be a rather direct consequence of the existence of the semi-flow of Theorem~\ref{t:curve_shortening} and of the following statement. We will employ the open subsets $\UU(\ell,\epsilon)\subset\Emb(S^1,M)$ defined in~\eqref{e:U_ell_epsilon}, which depend on an auxiliary Riemannian metric $g$ on $M$. Since such open subsets are invariant under the action of $\Diff(S^1)$, we can consider their quotients 
\[\WW(\ell,\epsilon):=\frac{\UU(\ell,\epsilon)}{\Diff(S^1)},\] 
which are open subset in $\Pi$.

\begin{lem}
\label{l:l_critical}
For each non-zero $h\in H_*(\Pi^{<b},\Pi^{<a})$, the associated min-max $\ell=\ell(h)$ is the length of a simple closed geodesic of $(M,F)$. For each $\epsilon>0$ there exists $\delta\in(0,\epsilon^2)$ such that $h$ can be represented by a relative cycle $\sigma$ with 
\begin{align}
\label{e:supp_cycle}
\supp(\sigma)\subset\Pi^{<\ell-\delta}\cup\WW(\ell,\epsilon).
\end{align} 
Moreover, if there are only finitely many simple closed geodesics with length in $(\ell-\epsilon^2,\ell+\epsilon^2)$, there exists a simple closed geodesic $\gamma$ of length $\ell$ such that, if we denote by $\VV(\gamma,\epsilon)$ the connected component of $\WW(\ell,\epsilon)$ containing $\gamma$, the inclusion induces a non-zero homomorphism
$H_*(\VV(\gamma,\epsilon),\VV(\gamma,\epsilon)^{<\ell-\delta})\to H_*(\Pi,\Pi^{<\ell})$.
\end{lem}

\begin{proof}
We set $\rho_0:=a/3$, and consider the semi-flow $\psi_t$ of Theorem~\ref{t:curve_shortening}. Since, by Theorem~\ref{t:curve_shortening}(ii), $\psi_t$ is equivariant with respect to the action of $\Diff(S^1)$, it induces a well-defined continuous semi-flow on the quotient of its domain, which we still denote by $\psi_t:\Pi\to\Pi$. Given $\epsilon>0$, we consider the associated $\delta\in(0,\epsilon^2)$ provided by Theorem~\ref{t:curve_shortening}(iv). By the definition of the min-max value $\ell:=\ell(h)$, we can find a relative cycle $\sigma$ representing $h$ and such that $\max L|_{\supp(\sigma)}<\ell+\delta$. For each $t>0$, the relative cycle $(\psi_t)_*\sigma$ still represents $h$. Since $\supp(\sigma)$ is compact, Theorem~\ref{t:curve_shortening}(iv) implies that, if we choose $t>0$ large enough, we have $\supp((\psi_t)_*\sigma)\subset\Pi^{<\ell-\delta}\cup\WW(\ell,\epsilon)$. This proves~\eqref{e:supp_cycle}.

Now, assuming by contradiction that $\ell$ is not the length of a simple closed geodesic of $(M,F)$, by choosing $K=\PT M$ in Lemma~\ref{l:neighborhoods} we would have that $\WW(\ell,\epsilon)=\varnothing$ for all $\epsilon>0$ small enough. However, by the result of the previous paragraph, this would allow us to find a relative cycle $\sigma$ representing $h$ and such that $\supp(\sigma)\subset\Pi^{<\ell-\delta}$, contradicting the definition of $\ell=\ell(h)$.

We are left to prove the moreover part of the statement. For that, notice that we can assume that $\epsilon>0$ is arbitrarily small (if the theorem holds for some $\epsilon$, it also holds for larger values of $\epsilon$). In particular, we assume that $\epsilon$ is small enough so that, by our assumptions, there are only finitely many simple closed geodesics $\gamma_1,...,\gamma_k$ with length in the interval $[\ell-\epsilon^2,\ell+\epsilon^2]$, and they all have length $\ell$. We denote by $\VV_{\epsilon,i}$ the connected component of $\WW_\epsilon:=\WW(\ell,\epsilon)$ containing $\gamma_i$. If needed, we further lower $\epsilon>0$, so that $\VV_{\epsilon,i}\cap \VV_{\epsilon,j}=\varnothing$ if $i\neq j$. We set
\begin{align*}
\VV_\epsilon:= \VV_{\epsilon,1}\cup...\cup\VV_{\epsilon,k}.
\end{align*}
We can also lower $\delta>0$ so that $a\leq\ell-\delta$. The inclusions induce the commutative diagram
\begin{equation*}
\begin{tikzcd}  
  H_*(\Pi^{<\ell-\delta}\cup\WW_\epsilon,\Pi^{<a}) 
  \arrow[r,"\quad i_*"]
  \arrow[d] &
  H_*(\Pi,\Pi^{<a})
  \arrow[d,"j_*"]
  \\
  H_*(\Pi^{<\ell-\delta}\cup\WW_\epsilon,\Pi^{<\ell-\delta}) 
  \arrow[r] &
  H_*(\Pi,\Pi^{<\ell})\\
  H_*(\WW_\epsilon,\WW_\epsilon^{<\ell-\delta}) 
  \arrow[u,"\cong"]
  \arrow[ur,"k_*"'] &
\end{tikzcd}
\end{equation*}
The homology class $h$ is contained in the image of $i_*$ according to~\eqref{e:supp_cycle}, and the lower vertical arrow is an isomorphism by excision. Moreover, by the very definition of the min-max value $\ell=\ell(h)$, we have that $j_*(h)\neq0$. Overall, this shows that the homomorphism $k_*$ is non-zero.

We set $\WW_\epsilon':=\WW_\epsilon\setminus(\VV_{\epsilon,1}\cup...\cup \VV_{\epsilon,k})$, and claim that 
\[
\epsilon':=\inf\big\{ \|V_{\gamma}\|_{L^\infty}\ \big|\  \gamma\in\WW_\epsilon'\big\}>0,
\]
where $V_\gamma$ is the function defined in~\eqref{e:V_gamma}. Otherwise, we could find a sequence $\gamma_n\in \WW_\epsilon'$ with $\|V_{\gamma_n}\|_{L^\infty}<1/n$. As in the proof of Lemma~\ref{l:neighborhoods}, one can easily show that, up to extracting a subsequence, $\gamma_n$ converges to a simple closed geodesic $\gamma$ with length $L(\gamma)\in[\ell-\epsilon^2,\ell+\epsilon^2]$. But this would imply that $V_\gamma\equiv0$ and $L(\gamma)=\ell$, and thus that $\gamma\in\VV_{\epsilon,1}\cup...\cup \VV_{\epsilon,k}$, which is impossible since $\gamma_n\in\WW_\epsilon'$ for all $n\in\N$.

Notice that $\WW_{\epsilon'}\subset\VV_\epsilon$, and once again the inclusion induces a non-zero homomorphism $H_*(\WW_{\epsilon'},\WW_{\epsilon'}^{<\ell-\delta})\to H_*(\Pi,\Pi^{<\ell})$, and therefore a non-zero homomorphism
\begin{align*}
\bigoplus_{i=1}^k H_d(\VV_{\epsilon,i},\VV_{\epsilon,i}^{<\ell-\delta}) \to
H_d(\Pi,\Pi^{<\ell}).
\end{align*}
We denote by $I_\epsilon\subseteq\{1,...,k\}$ the subset of those $i$ such that the homomorphism $H_d(\VV_{\epsilon,i},\VV_{\epsilon,i}^{<\ell-\delta}) \to H_d(\Pi,\Pi^{<\ell})$ is non-zero. Notice that $I_{\epsilon_1}\subseteq I_{\epsilon_2}$ if $0<\epsilon_1<\epsilon_2$. Therefore, there exists
\begin{align*}
 i\in\bigcap_{\epsilon\in(0,\epsilon_0]} I_\epsilon,
\end{align*}
and the simple closed geodesic $\gamma_i$ satisfies the desired properties.
\end{proof}

Assume now to have a homology class $h\in H_{d+i}(\Pi,\Pi^{<\rho})$ and a cohomology class $w\in H^{i}(\Pi)$ whose cap product $w\smallfrown h\in H_{d}(\Pi,\Pi^{<\rho})$ is non-zero. Given any relative cycle $\sigma$ representing $h$ we can produce a relative cycle $\sigma'$ representing $w\smallfrown h$ and such that $\supp(\sigma')\subset\supp(\sigma)$. This readily implies that \[\ell(w\smallfrown h)\leq\ell(h).\]
We can now state a version of the classical Lusternik-Schnirelmann theorem for the length functional.  

\begin{thm}
\label{t:LS_minmax}
If $w\smallfrown h\neq0$ and $\ell(w\smallfrown h)=\ell(h)$, then for every $\epsilon>0$ we have $w|_{\WW(\ell(h),\epsilon)}\neq 0$ in $H^{*}(\WW(\ell(h),\epsilon))$ .
\end{thm}

\begin{proof}
Let $\epsilon>0$ be given, and set $\ell:=\ell(h)$ and $\WW:=\WW(\ell,\epsilon)$. By Lemma~\ref{l:l_critical}, $h$ can be represented by a relative cycle $\sigma$ such that $\supp(\sigma)\subset\WW\cup\Pi^{<\ell}$.
By applying sufficiently many barycentric subdivisions to the singular simplexes in $\sigma$, we can assume that the relative cycle decomposes as $\sigma=\sigma'+\sigma''$, where $\sigma'$ and $\sigma''$ are chains with $\supp(\sigma')\subset\Pi^{<\ell}$ and $\supp(\sigma'')\subset\WW$. Let $w\in H^{*}(\Pi)$ be a cohomology class such that $w|_{\WW}=0$ in $H^{*}(\WW)$ and $w\smallfrown h\neq0$. The cohomology long exact sequence of the pair $\WW\subset\Pi$ provides a relative cocycle $\mu$ representing $w$ that vanishes on all singular simplexes contained in $\WW$. This implies
\begin{align*}
w\smallfrown h=[\mu\smallfrown(\sigma'+\sigma'')]=[\mu\smallfrown\sigma'].
\end{align*}
Namely, $w\smallfrown h$ is represented by the relative cycle $\mu\smallfrown\sigma'$ whose support is contained in the sublevel set $\Pi^{<\ell}$, which implies that $\ell(w\smallfrown h)<\ell$.
\end{proof}

\subsection{Topology of the space of embedded circles on the 2-sphere}
Once the results of the previous subsection are established, the proofs of points (i) and (ii) in Theorem~\ref{t:LS} are  analogous to ones of the Riemannian case in \cite{Mazzucchelli:2018ek}. In this subsection, we provide the arguments for the reader's convenience. We will adopt the notation of the previous section, with $M$ equal to the unit sphere $S^2\subset\R^3$. It will be crucial to consider the singular homology $H_*$ with coefficients in $\Z_2=\Z/2\Z$, and therefore we will specify the coefficients in the notation.

We first recall, from \cite{Ballmann:1978rw,Mazzucchelli:2018ek}, some basic information concerning the topology of the space of its unparametrized embedded loops $\Pi$. It is convenient to slightly enlarge this space as follows: we denote by $\Pi_0$ the space of constant loops on $S^2$, and set $\overline\Pi:=\Pi\cup\Pi_0$. We endow $\overline\Pi$ with the quotient $C^\infty$-topology as a subspace of $C^\infty(S^1,M)/\Diff(S^1)$. The relevant topology of $\overline\Pi$, at least for what concerns the application to Theorem~\ref{t:LS}, is provided by the subspace of round circles. More precisely, let 
\begin{align*}
E=\big\{([x],\lambda x)\in\RP^2\times\R^3\ |\ x\in S^2,\ \lambda\in[-1,1]\big\}.
\end{align*}
Namely, $E$ is the total space of the canonical line bundle $\pi:E\to\RP^2$ with fiber $[-1,1]$. We consider the embedding 
\begin{align*}
 \iota:E\to\overline\Pi, \qquad\iota(e)=\gamma_{e},
\end{align*}
where, if $e=([x],y)$,  $\gamma_{e}\in\overline\Pi$ is the (possibly constant) loop in the intersection of $S^2$ with the affine plane orthogonal to $x$ and passing through $y$. The fundamental group of this space is 
\[\pi_1(E)\cong\pi_1(\RP^2)\cong\Z_2.\] 
Its cohomology ring with $\Z_2$ coefficient is given by 
\[H^*(E;\Z_2)=\Z_2[u]/(u^3),\] 
where $u$ is the generator of $H^1(E;\Z_2)\cong H^1(\RP^2;\Z_2)\cong\Z_2$. Moreover, by the Thom isomorphism theorem, 
\[H^*(E,\partial E;\Z_2)\cong H^{*-1}(E;\Z_2)=\langle v,v\smallsmile u,v\smallsmile u^2\rangle,\]
where $v\in H^1(E,\partial E)$ denotes the Thom class of the bundle $\pi:E\to\RP^2$. Since we work with $\Z_2$ coefficients, the homology is simply the dual of the cohomology, and in particular there exists $k_3\in H_3(E,\partial E;\Z_2)$ such that $(v\smallsmile u^2)k_3=1$. Therefore, we also have the classes $k_2:=u\smallfrown k_3$ and $k_1:=u \smallfrown k_2$, and overall we have 
\[H_*(E,\partial E;\Z_2)=\langle k_1,k_2,k_3\rangle.\]

\begin{lem}
\label{l:injectivity_pi1}
The map $\iota$ induces injective homomorphisms
\begin{align*}
\iota_*:\pi_1(E)\hookrightarrow\pi_1(\overline\Pi),
\qquad
\iota_*:H_1(E;\Z_2)\hookrightarrow H_1(\overline\Pi;\Z_2),
\end{align*}
and a surjective homomorphism
\begin{align*}
 \iota^*:H^1(\overline\Pi;\Z_2)\epi H^1(E;\Z_2).
\end{align*}
\end{lem}

\begin{proof}
We introduce the following double covering map $p:C\to\overline\Pi$. The preimage of an embedded circle $\gamma\in\Pi$ is the two-element set $p^{-1}(\gamma)=\{D_1,D_2\}$, where $D_1$ and $D_2$ are the connected components of $S^2\setminus\gamma$; namely, $p^{-1}(\Pi)$ is the space of the interiors of embedded compact disks in $S^2$. The preimage of a constant $\gamma\in\Pi_0$ is the two-element set $p^{-1}(\gamma)=\{\varnothing_\gamma,S^2\setminus\gamma\}$; intuitively, $\varnothing_\gamma$ is the ``empty filling disk'' of the constant $\gamma$. The space $C$, endowed with the obvious topology, makes $p:C\to\overline\Pi$ a covering map.

Notice that both $E$ and $\overline\Pi$ are path-connected. We fix an arbitrary base-point $e_0=([x_0],0)\in E$ and a corresponding base-point $\gamma_{e_0}:=\iota(e_0)\in\Pi$ for the fundamental groups of $E$ and $\overline\Pi$ respectively. We define a homomorphism
\begin{align*}
 A:\pi_1(\overline\Pi)\to\Z_2
\end{align*}
as follows. For any continuous path $\mu:[0,1]\to\overline\Pi$ such that $\mu(0)=\mu(1)=\gamma_{e_0}$, we consider an arbitrary continuous lift $\widetilde\mu:[0,1]\to C$, i.e.~$p\circ\widetilde\mu=\mu$. The homomorphism $A$ is defined by
\begin{align*}
 A([\mu])=
 \left\{  
  \begin{array}{@{}cc}
    0, & \mbox{if }\widetilde\mu(0)=\widetilde\mu(1),\vspace{5pt} \\ 
    1, & \mbox{if }\widetilde\mu(0)\neq\widetilde\mu(1). \\ 
  \end{array}
 \right.
\end{align*}
For the generator $k\in\pi_1(E)\cong\Z_2$, we readily see that $A\circ\iota_*(k)=1$. Namely, the composition 
\[\displaystyle A\circ\iota_*:\pi_1(E)\toup^{\cong}\Z_2\] 
is an isomorphism. In particular, $\iota$ induces an injective homomorphism
\begin{align*}
 \iota_*:\pi_1(E)\hookrightarrow \pi_1(\overline\Pi).
\end{align*}
Since $\Z_2$ is abelian, the commutator subgroup $[\pi_1(\overline\Pi),\pi_1(\overline\Pi)]$ is contained in the kernel of $A$. This readily implies that $\iota$ induces an injective homomorphism
\begin{align*}
 \iota_*:H_1(E;\Z_2)\hookrightarrow H_1(\overline\Pi;\Z_2).
\end{align*}
Indeed, we have
\begin{align*}
H_1(E;\Z_2)\cong \pi_1(E,e_0)\cong\Z_2,
\qquad
H_1(\overline\Pi;\Z_2)\cong\frac{\pi_1(\overline\Pi)}{[\pi_1(\overline\Pi),\pi_1(\overline\Pi)]}\otimes\Z_2.
\end{align*}
Assume by contradiction that $\iota_*(k)=0$ in $H_1(\overline\Pi;\Z_2)$, where $k$ is the generator of $\pi_1(E,e_0)\cong H_1(E;\Z_2)$. This is equivalent to $\iota_*(k)=a*a*b$ in $\pi_1(\overline\Pi)$ for some $a\in\pi_1(\overline\Pi)$ and $b\in [\pi_1(\overline\Pi),\pi_1(\overline\Pi)]$; here, $*$ denotes the group multiplication in $\pi_1(\overline\Pi)$. However, this implies $A\circ\iota_*(k)=2A(a)+A(b)=0$, contradicting the fact that $A\circ\iota_*$ is an isomorphism.

Finally, since the homomorphism $\iota_*:H_1(E;\Z_2)\hookrightarrow H_1(\overline\Pi;\Z_2)$ is injective, its dual $\iota^*:H^1(\overline\Pi;\Z_2)\epi H^1(E;\Z_2)$ is surjective.  
\end{proof}

We fix, once for all, a cohomology class $w\in H^1(\overline\Pi;\Z_2)$ such that $\iota^*w=u$.
Since $\partial E=\Pi_0\cong S^2$ is simply connected, the long exact sequence of homology groups readily implies that $H_1(E,\partial E;\Z_2)\cong H_1(E;\Z_2)$ and $H_1(\overline\Pi,\Pi_0;\Z_2)\cong H_1(\overline\Pi;\Z_2)$.
This, together with Lemma~\ref{l:injectivity_pi1}, implies that $\iota$ induces an injective homomorphism of relative homology groups \[\iota_*:H_1(E,\partial E;\Z_2)\hookrightarrow H_1(\overline\Pi,\Pi_0;\Z_2).\] Therefore, $\iota_*k_2$ and $\iota_*k_3$ are both non-zero in $H_*(\overline\Pi,\Pi_0;\Z_2)$, since
\begin{align*}
w^2\smallfrown \iota_*k_3
=
w\smallfrown\iota_*(u\smallfrown k_3)
=
w\smallfrown\iota_*k_2
=
\iota_*(u\smallfrown k_2)
=
\iota_*k_1\neq0.
\end{align*}

Now, let us get rid of the space of constant curves $\Pi_0$. We recall that the systole $\sys(S^2,F)$ is the length of the shortest closed geodesic of $(S^2,F)$.

\begin{lem}
\label{l:remove_bullet}
For all $\rho\in(0,\sys(S^2,F))$, the inclusion $\Pi_0\subset\overline\Pi{}^{<\rho}$ is a homotopy equivalence. Therefore, the inclusions induce the homology isomorphisms
\begin{align*}
H_*(\overline\Pi,\Pi_0;\Z_2)
\toup^{j_*}_{\cong}
H_*(\overline\Pi,\overline\Pi{}^{<\rho};\Z_2)
\otup^{l_*}_{\cong}
H_*(\Pi,\Pi^{<\rho};\Z_2).
\end{align*}
\end{lem}
\begin{proof}
We fix $\rho=\rho_2\in(0,\sys(S^2,F))$. We claim that, for any $\rho_1\in(0,\rho_2)$, the inclusion 
$\overline\Pi{}^{<\rho_1} \subset \overline\Pi{}^{<\rho_2}$ is a homotopy equivalence. Indeed, since there are no simple closed geodesics of $(S^2,F)$ with length in $[\rho_1,\rho_2]$, Theorem~\ref{t:curve_shortening}(v) with choice of parameter $\rho_0\in(0,\rho_1/2)$ implies that there exists a continuous function $\tau:\overline\Pi{}^{<\rho_2}\to(0,\infty)$ and a continuous map
\begin{align*}
\kappa:\overline\Pi{}^{<\rho_2}\to\overline\Pi{}^{<\rho_1},
\qquad
\kappa(\gamma):=\psi_{\tau(\gamma)}(\gamma).
\end{align*}
The map $\kappa$ is a homotopy inverse of the inclusion $\overline\Pi{}^{<\rho_1} \subset \overline\Pi{}^{<\rho_2}$.

We consider $S^2$ as the unit sphere in $\R^3$, equipped with its round metric $g_0$ induced by the Euclidean metric of $\R^3$. We denote by $\|\cdot\|$ the Euclidean norm on $\R^3$, and by $\exp_x$ the exponential maps of $(S^2,g_0)$.
We fix a sufficiently large constant $a\geq1$ so that
\begin{align*}
a^{-1}\|v\|\leq F(x,v) \leq a\|v\|,
\qquad\forall (x,v)\in\Tan S^2.
\end{align*}
If $\rho_1>0$ is small enough, each $\gamma\in\overline\Pi{}^{<\rho_2}$ has average outside the origin; namely, if we parametrize $\gamma$ with constant speed $F(\gamma,\dot\gamma)\equiv L(\gamma)$, we have
\begin{align*}
\int_{S^1} \gamma(u)\,\diff u\neq0.
\end{align*}
We denote by $\pi:\R^3\setminus\{0\}\to S^2$ the radial projection $\pi(x)=x/\|x\|$, and set
\begin{align*}
\widehat\gamma
:=
\pi\left(\int_{S^1} \gamma(u)\,\diff u\right)\in S^2. 
\end{align*}
Notice that 
\begin{align}
\label{e:bound_distance_barycenter}
 \max_{u\in S^1}\|\gamma(u)-\widehat\gamma\|\leq a L(\gamma)< a\rho_1.
\end{align}
For each $x\in S^2$, we set 
\[B_x:=\big\{v\in T_xS^2\ \big|\ \|v\|<\pi\big\},
\qquad
U_x:=\exp_x(B_x)\subset S^2.\]
From now on we will restrict the exponential map of $(S^2,g_0)$ as a diffeomorphism of the form $\exp_x:B_x\to U_x$. We require $\rho_1<\pi/a$ so that, by~\eqref{e:bound_distance_barycenter}, we have $\gamma\subset U_{\widehat\gamma}$ whenever $L(\gamma)<\rho_1$.

We  define the continuous homotopy
\begin{align*}
r_t:\overline\Pi{}^{<\rho_1}\to\overline\Pi,
\qquad
r_t(\gamma)(u)=\exp_{\widehat\gamma}\big( (1-t) \exp_{\widehat\gamma}^{-1}(\gamma(u)) \big).
\end{align*}
Notice that the time-1 map is a retraction \[r_1:\overline\Pi{}^{<\rho_1}\to\Pi_0.\] 
Moreover, if $b>0$ is a constant larger than $\|\diff\exp_x(v)\|$ and $\|\diff\exp_x^{-1}(y)\|$ for all $x\in S^2$, $v\in B_x$, and $y\in U_x$, we have
\begin{align*}
L(r_t(\gamma))
\leq
a^2 b^2
\int_{S^1} \|\dot\gamma(u)\|\,\diff u
\leq
a^2 b^2  L(\gamma)
<
a^2b^2\rho_1.
\end{align*}
We require $\rho_1<\rho_2 a^{-2}b^{-2}$, so that every $r_t$ is a map of the form \[r_t:\overline\Pi{}^{<\rho_1}\to\overline\Pi{}^{<\rho_2}.\]

Overall, this implies that the inclusion $\Pi_0\subset\overline\Pi{}^{<\rho}$ is a homotopy equivalence. The homology long exact sequence of the triple $\Pi_0\subset\overline\Pi{}^{<\rho}\subset\overline\Pi$ readily implies that $j_*$ is an isomorphism. Finally, the excision property implies that $l_*$ is an isomorphism as well.
\end{proof}

We consider the isomorphisms $j_*$ and $l_*$ provided by Lemma~\ref{l:remove_bullet}, and define the non-zero relative homology classes
\begin{align}
\label{e:h_i}
h_i:=l_*^{-1}j_*\iota_* k_i\in H_*(\Pi,\Pi^{<\rho};\Z_2),
\qquad i=1,2,3.
\end{align}
Notice that
\begin{align*}
 w|_{\Pi} \smallfrown h_{i+1} = h_i.
\end{align*}
We denote by $E_0\subset E$ the zero-section of the line bundle $\pi:E\to\RP^2$. Notice that $\iota$ restricts as a map of the form 
$\iota_0:=\iota|_{E_0}:E_0\to\Pi$.

\begin{lem}
\label{l:z}
For each $z\in H^2(\Pi;\Z_2)$ such that $\iota_0^*z\neq0$ in $H^2(E_0;\Z_2)$, we have \[z\smallfrown h_3= h_1.\]
\end{lem}

\begin{proof}
For each $r\in[0,1]$, we consider the subset
\begin{align*}
 E_r=\big\{ ([x],\lambda x)\in E\ |\ \lambda\in[-r,r] \big\}.\subset E,
\end{align*}
Notice that this notation agrees with the definition of $E_0$ as the zero-section of the line bundle $\pi:E\to\RP^2$. We fix $r\in(0,1)$ sufficiently close to $1$ so that $\iota(\partial E_r)\subset\Pi^{<\rho}$.
By deformation and excision, we have that the inclusions induce isomorphisms
\begin{align*}
 H_*(E,\partial E;\Z_2)\toup^{\cong} H_*(E,E\setminus E_0;\Z_2)\otup^{\cong} H_*(E_r,E_r\setminus E_0;\Z_2)\otup^{\cong} H_*(E_r,\partial E_r;\Z_2).
\end{align*}
We denote by $k_i'\in H_*(E_r,\partial E_r;\Z_2)$ the image of $k_i\in H_*(E,\partial E;\Z_2)$ under the composition of the above isomorphisms. Notice that $k_1'=u|_{E_r}\smallfrown k_3'$.
The restriction $\iota_r=\iota|_{E_r}:E_r\to\Pi$ induces a homomorphism 
\[(\iota_r)_*:H_*(E_r,\partial E_r;\Z_2)\to H_*(\Pi,\Pi^{<\rho};\Z_2),\] 
which allows to express the homology classes~\eqref{e:h_i} as $h_i=(\iota_r)_*k_i'$.
Since the inclusion $E_0\subset E_r$ is a homotopy equivalence, a cohomology class $z\in H^2(\Pi;\Z_2)$ satisfies $\iota_0^*z\neq0$ if and only if $\iota_r^*z\neq0$, and thus if and only if $\iota_r^*z\smallfrown k_3'=k_1'$. Therefore, if this is the case, we have
\[
z\smallfrown h_3= z\smallfrown (\iota_r)_*k_3'= (\iota_r)_*(\iota_r^*z\smallfrown k_3') = (\iota_r)_*k_1'=h_1.
\qedhere
\]

\end{proof}

We set
\begin{align}
\label{e:ell_h_i} 
\ell_i:=\ell(h_i),\qquad i=1,2,3.
\end{align}
Since every $\ell_i$ is the length of a simple closed geodesic of $(S^2,F)$, if the simple length spectrum $\sigmas(S^2,F)$ is a singleton we have $\ell_1=\ell_2=\ell_3$. In this case Theorem~\ref{t:LS}(i) is a consequence of the following statement.

\begin{thm}
If $\ell_1=\ell_2=\ell_3$, then $(S^2,F)$ is simple Zoll.
\end{thm}
\begin{proof}
We consider a circle bundle $\pr:P\to\Pi$, whose total space is given by $P=\{(\gamma,x)\in\Pi\times S^2\ |\ x\in\gamma\}$ and whose projection is $\pr(\gamma,x)=\gamma$. We consider the projectivized tangent bundle 
\[\PT S^2=\big\{V_x\ \big|\ x\in S^2,\ V_x\mbox{ 1-dimensional vector subspace of }\Tan_x S^2\big\},\]
and define the continuous evaluation map $\ev:P\to \PT S^2$, $\ev(\gamma,x)=\Tan_x\gamma$. Since $\PT S^2$ is a closed 3-manifold, we have $H^3(\PT S^2;\Z_2)\cong\Z_2$, and we denote by $m$ a generator of $H^3(\PT S^2;\Z_2)$. We consider the pull-back bundle 
\[
P_0=\iota_0^*P=\big\{ (e,p)\in E_0\times P\ |\ \iota_0(e)=\pr(p) \big\}
,
\] and the commutative diagram
\begin{equation*}
\begin{tikzcd}
  P_0 \arrow[r,"\tilde \iota_0", hookrightarrow] \arrow[d,"\pr|_{P_0}"] &
  P \arrow[r,"\ev"] \arrow[d,"\pr"] &
  \PT S^2\\
  E_0\, \arrow[r,"\iota_0", hookrightarrow] & \Pi  &  
\end{tikzcd}
\end{equation*}
Here, $\tilde\iota_0(e,p)=p$ is the projection onto the second factor. Notice that $\ev\circ\tilde\iota_0$ is a homeomorphism. Moreover, since $H^3(E_0;\Z_2)$ and $H^4(E_0;\Z_2)$ are trivial, the Gysin sequence of the pull-back bundle $\pr|_{P_0}:P_0\to E_0$ readily implies that \[(\pr|_{P_0})_*:H^3(P_0;\Z_2)\to H^2(E_0;\Z_2)\] 
is an isomorphism. This implies that $(\pr|_{P_0})_*\tilde\iota_0^*\ev^*m\neq 0$ in $H^2(E_0;\Z_2)$. We set 
\[z:=\pr_*\ev^*m\in H^2(\Pi;\Z_2).\] 
Since $\iota_0^*z=(\pr|_{P_0})_*\tilde\iota_0^*\ev^*m\neq 0$,  Lemma~\ref{l:z} implies that $h_1=z\smallfrown h_3$.

Now, assume by contradiction that $\ell_1=\ell_2=\ell_3=:\ell$, but there exists $(x,v)\in SS^2$ such that the geodesic $\gamma_{x,v}(t):=\exp_{x}(tv)$ is not a simple closed geodesic of minimal period $\ell$ (namely, $\gamma_{x,v}$ is not a closed geodesic, or it is closed but not simple closed, or it is simple closed but its length is not $\ell$). By Lemma~\ref{l:neighborhoods} there exists $\epsilon>0$ small enough so that $v$ is not tangent to any curve $\gamma\in\WW:=\WW(\ell,\epsilon)$ passing through $x$. Namely, if we set $P':=\pr^{-1}(\WW)$, the restriction $\ev|_{P'}:P'\to\PT S^2$ is not surjective. Since $\ell_1=\ell_3$ and  $h_1=z\smallfrown h_3$,  Theorem~\ref{t:LS_minmax} implies that $z|_{\WW}\neq 0$ in $H^2(\WW;\Z_2)$. However, since $z|_{\WW}=(\pr|_{P'})_*\ev|_{P'}^*m$, this implies that the homomorphism \[\ev|_{P'}^*:H^3(\PT S^2;\Z_2)\to H^3(P';\Z_2)\] is non-zero, which is impossible since $\ev|_{P'}$ is not surjective.
\end{proof}

If the simple length spectrum $\sigmas(S^2,F)$ contains exactly two elements, we must have $\ell_1=\ell_2$ or $\ell_2=\ell_3$. In this case Theorem \ref{t:LS}(i) is a consequence of the following statement.

\begin{thm}
If $\ell_i=\ell_{i+1}$ for some $i\in\{1,2\}$, then every point of $S^2$ lies on a simple closed geodesic of $(S^2,F)$ of length  $\ell_i$.
\end{thm}

\begin{proof}
Assume by contradiction that $\ell:=\ell_i=\ell_{i+1}$, but that some point $x\in S^2$ does not lie on a simple closed geodesic of length $\ell$. We consider the subset $\UU=\{\gamma\in\overline\Pi\ |\ x\not\in\gamma\}$. It is easy to see that $\UU$ is contractible: if we denote by $B^2\subset\R^2$ the unit open ball, and we consider a homeomorphism $\theta:S^2\setminus\{x\}\to B^2$, the homotopy $r_t:\UU\to\UU$, $t\in[0,1]$, given by
\[\qquad r_t(\gamma)=\theta^{-1}((1-t)\theta(\gamma))\]
defines a contraction of $\UU$ onto a point curve in $\Pi_0\cap\UU$. In particular $H^1(\UU;\Z_2)$ is trivial.
 
By applying Lemma~\ref{l:neighborhoods} with $K=\PPP(\Tan_xS^2)$, we infer that there exists $\WW=\WW(\ell,\epsilon)$, for $\epsilon>0$ small enough, such that none of the curves $\gamma\in\WW$ passes through $x$. Since $h_i=w|_{\Pi}\smallfrown h_{i+1}$ and $\ell_i=\ell_{i+1}$, Theorem~\ref{t:LS_minmax} implies that $w|_{\WW}\neq0$ in $H^1(\WW;\Z_2)$. However, since $\WW\subset\UU$, we have $w|_{\UU}\neq0$ in $H^1(\UU;\Z_2)$ as well, contradicting the conclusion of the previous paragraph.
\end{proof}

\section{Critical point theory of the energy functional}
\label{s:critical_point_theory}

In this section we shall recall the background on the variational theory of Finsler closed geodesics. The reader can find more details and proofs in \cite{Rademacher:1992sw,Bangert:2010ak,Caponio:2011aa,Asselle:2018aa} and references therein. Throughout the section, we shall consider a closed Finsler manifold $(M,F)$ of arbitrary dimension, except in certain statements where we will assume $M$ to be a surface. The Finsler metric $F$ is not required to be reversible, unless specifically stated.

\subsection{The energy functional}
\label{ss:energy}
We denote by $\Lambda=W^{1,2}(S^1,M)$ the free loop space of $M$ of regularity $W^{1,2}$, and consider the energy functional
\begin{align*}
E:\Lambda\to[0,\infty),\qquad E(\gamma)=\int_{S^1} F(\gamma(u),\dot\gamma(u))^2\,\diff u.
\end{align*}
Unlike in the Riemannian case, in the Finsler setting $E$ is $C^{1,1}$, but possibly not $C^2$. Its critical points with positive critical value are precisely those $\gamma\in\Lambda $ that are closed geodesics of $(M,F)$ parametrized with constant speed $F(\gamma,\dot\gamma)\equiv E(\gamma)^{1/2}$. For each $\gamma\in\Lambda$, we denote by 
\[\gamma^m\in\Lambda,\qquad\gamma^m(t)=\gamma(mt)\] its $m$-th iterate, whose energy is $E(\gamma^m)=m^2E(\gamma)$. Clearly, iterates of critical points of $E$ are again critical points. Identifying different iterates of the same closed geodesic detected with global variational methods is the crux of the matter in the closed geodesics problem.

The $C^1$ regularity of $E$ is actually enough to define a smooth pseudo-gradient flow of $E$ on $\Lambda$. It is well known that $E$ satisfies the Palais-Smale condition with respect to a suitable complete Riemannian metric on $\Lambda$, and therefore we can perform the usual deformations of critical point theory. Since $E$ is even $C^{1,1}$, it has a well define Gateaux Hessian $\diff^2E(\gamma)$ at every critical point. However, the $C^2$ regularity would be needed in order to apply the classical Morse-Gromoll-Meyer lemma \cite{Gromoll:1969jy}. A simple way to circumvent the potential lack of $C^2$ regularity and, at the same time, work in a finite dimensional setting consists in employing Morse's finite dimensional approximations of $\Lambda$. We consider the (non-symmetric) Finsler distance
\begin{align}
\label{e:Finsler_distance}
d:M\times M\to[0,\infty),
\qquad
d(x,y) = \min_{\gamma} \int_0^1 F(\gamma(u),\dot\gamma(u))\,\diff u,
\end{align}
where the minimum ranges over all absolutely continuous curves $\gamma:[0,1]\to M$ joining $x$ and $y$.
For each integer $k\geq 2$, we consider the space
\begin{align*}
\Lambda_k=\left\{\xx=(x_0,...,x_{k-1})\in M^{\times k}\ \left|\ \sum_{i\in\Z_k} d(x_i,x_{i+1})^2<\injrad(M,F)^2\quad \forall i\in\Z_k \right.\right\}.
\end{align*}
There is a smooth embedding $\iota:\Lambda_k\hookrightarrow\Lambda$ defined as follows: every $\xx\in\Lambda_k$ is mapped to the curve $\gamma_{\xx}:=\iota(\xx)\in\Lambda$ such that each restriction $\gamma_{\xx}|_{[i/k,(i+1)/k]}$, for $i\in\Z_k$, is the shortest geodesic parametrized with constant speed joining $x_i$ and $x_{i+1}$. In the following, we will identify $\Lambda_k$ with it image $\iota(\Lambda_k)\subset\Lambda$, and indistinctively write $\xx$ or $\gamma_{\xx}$ for the same object. The restriction of the energy to $\Lambda_k$ has the form
\begin{align*}
 E_k=E|_{\Lambda_k}:\Lambda_k\to\big[0,k\,\injrad(M,F)^2\big),
 \qquad
 E_k(\xx)= k \sum_{i\in\Z_k} d(x_i,x_{i+1})^2.
\end{align*}
Since the distance $d$ is smooth away from the diagonal, $E_k$ is smooth on the subspace of those $\xx$ with $x_i\neq x_{i+1}$ for all $i\in\Z_k$. The critical points of $E_k$ are precisely those $\xx$ such that $\gamma_{\xx}$ is a closed geodesic of $(M,F)$ parametrized with constant speed and having energy $E_k(\xx)=E(\gamma_{\xx})< k\,\injrad(M,F)^2$. In particular $E_k$ is smooth on a sufficiently small neighborhood of its critical points with positive energy. For each compact interval $[a,b]\subset (-\infty,k\,\injrad(M,F)^2\big)$, the preimage $E_k^{-1}[a,b]$ is compact, which allows us to apply the gradient flow deformations from critical point theory. For each $a>0$, up to choosing $k\in\N$ large enough, the inclusion of the energy sublevel sets $E_k^{-1}(-\infty,a)\hookrightarrow E^{-1}(-\infty,a)$ admits the homotopy inverse 
\begin{align*}
 r:E^{-1}(-\infty,a)\to E_k^{-1}(-\infty,a), 
 \qquad
 r(\gamma)=(\gamma(0),\gamma(\tfrac1k),...,\gamma(\tfrac{k-1}{k})).
\end{align*}

\subsection{The Morse index and nullity}\label{ss:index}

Let $h:V\times V\to\R$ be a symmetric bilinear form on a vector space $V$. Its index $\ind(h)$ is defined as the supremum of the dimension of the subspaces $W\subset V$ such that $h|_W$ is negative definite. Its nullity $\nul(h)$ is defined as the dimension of $\ker(h)=\{v\in V\ |\ h(v,\cdot)=0\}$. Notice that the sum $\ind(h)+\nul(h)$ is the supremum of the dimension of the subspaces $Z\subset V$ such that $h|_Z$ is negative semi-definite.

Let us consider a closed geodesic $\gamma\in\crit(E)\cap E^{-1}(0,\infty)$, and the associated Gateaux Hessian $h:=\diff^2E(\gamma)$. The Morse index and nullity of $\gamma$ are defined by
\begin{align*}
\ind(\gamma):=\ind(h),\qquad
\nul(\gamma):=\nul(h)-1.
\end{align*}
It is well known that the indices are always finite. The reason for the $-1$ appearing in the definition of the nullity is that $\nul(h)$ is always larger than or equal to 1, as the vector field $\dot\gamma$ belongs to $\ker(h)$. If $x_0:=\gamma(0)$, we denote by \[\Omega:=\{\zeta\in\Lambda\ |\ \zeta(0)=x_0\}\] the space of loops based at $x_0$. The critical points of $E|_\Omega$ are the geodesic loops, that is, those $\zeta\in\Lambda$ whose restriction $\zeta|_{(0,1)}$ is a geodesic parametrized with constant speed. The Morse index and nullity of $\gamma$ in the based loop space are defined as
\begin{align*}
\ind_\Omega(\gamma):=\ind(h|_{\Tan_\gamma\Omega}),\qquad
\nul_\Omega(\gamma):=\nul(h|_{\Tan_\gamma\Omega}).
\end{align*}
The behavior of the Morse indices under iteration of the closed geodesic has been thoroughly studied since the seminal work of Bott \cite{Bott:1956sp}. Without invoking Bott's theory, one has the following properties, which are rather immediate or can be proved as an exercise.

\begin{lem}
\label{l:elementary_ind}
Let $(M,F)$ be a Finsler manifold with a closed geodesic $\gamma\in\crit(E)$. The Morse indices of $\gamma$ satisfy the following properties.
\begin{itemize}

\item[(i)] $\ind(\gamma)\geq\ind_\Omega(\gamma)$.

\item[(ii)] $\ind(\gamma)+\nul(\gamma)\geq\ind_\Omega(\gamma)+\nul_\Omega(\gamma)$.

\item[(iii)] If $\ind(\gamma)>0$, then $\ind(\gamma^m)\to\infty$ as $m\to\infty$.

\item[(iv)] $\ind(\gamma^m)\geq\ind(\gamma)$ and $\nul(\gamma^m)\geq\nul(\gamma)$ for all $m\in\N$.

\item[(v)] $\ind_\Omega(\gamma^m)\geq m\,\ind_\Omega(\gamma)$ and $\ind_\Omega(\gamma^m)+\nul_\Omega(\gamma^m)\geq m\,(\ind_\Omega(\gamma)+\nul_\Omega(\gamma))$ for all $m\in\N$. \hfill\qed

\end{itemize}
\end{lem}

The following proposition summarizes those subtler results concerning the Morse indices of closed geodesics that we will need in the proof of Theorem~\ref{t:multiplicity}. In the literature, most of these results are proved in the Riemannian setting: points~(i--iv) can be found in \cite{Ballmann:1982rz}, point (vi) in \cite{Klingenberg:1995aa}, and point~(vii) in \cite{Bangert:1993wo}. In the Finsler setting, the differences in the proofs are essentially cosmetic, but we include them for the reader's convenience. 

\begin{prop}\label{p:indices}
Let $(M,F)$ be an orientable Finsler manifold, and $\gamma\in\crit(E)\cap E^{-1}(0,\infty)$ a closed geodesic. The indices of $\gamma$ satisfy the following properties.
\begin{itemize}

\item[(i)] $\nul(\gamma)\leq2\dim(M)-2$.

\item[(ii)] $\nul_\Omega(\gamma)\leq\dim(M)-1$.

\item[(iii)] $\ind(\gamma)\leq\ind_\Omega(\gamma)+\dim(M)-1$. \vspace{2pt}

\item[(iv)] $\ind(\gamma)+\nul(\gamma) \leq \ind_\Omega(\gamma)+\nul_\Omega(\gamma)+\dim(M)-1$.

\end{itemize}
Moreover, if $M$ is an orientable surface, they further satisfy the following properties.

\begin{itemize}

\item[(v)] If $\nul_\Omega(\gamma)=1$ then $\ind_\Omega(\gamma^m)=m\,\ind_\Omega(\gamma)+m-1$ and $\nul_\Omega(\gamma^m)=\nul_\Omega(\gamma)$ for all $m\in\N$, \vspace{2pt}

\item[(vi)] If $\nul(\gamma)=2$ then $\nul(\gamma^m)=2$ and $\ind(\gamma^m)$ is odd for all $m\in\N$.

\item[(vii)] If $\ind_\Omega(\gamma^m)>0$ for some $m\geq1$, then $\ind(\gamma)>0$ and indeed there exists a nowhere vanishing smooth vector field $\zeta$ along $\gamma$ that is 1-periodic, everywhere transverse to $\dot\gamma$, and such that $\diff^2E(\gamma)(\zeta,\zeta)<0$.

\end{itemize}
\end{prop}

\begin{proof}
We can assume without loss of generality that $E(\gamma)=1$, so that $F(\gamma,\dot\gamma)\equiv1$. We set 
\[G:\Tan M\to[0,\infty),\qquad G(x,v)=\tfrac12F(x,v)^2,\] which is a $C^{1,1}$ function, smooth outside the zero-section, and fiberwise positively homogeneous of degree 2. The function $G$ defines a 1-form $\lambda$ on $\Tan M$ by
\begin{align*}
\lambda_{(x,v)}(w)=G_v(x,v)\circ\diff\pi(x,v)w,\qquad\forall w\in\Tan_{(x,v)}(TM).
\end{align*}
The 2-form $-\diff\lambda$ is a symplectic form on the complement of the zero-section in $\Tan M$. We treat $G$ as a Hamiltonian, and consider its associated Hamiltonian vector field $X$ defined by $-\diff\lambda(X,\cdot)=\diff G$. We denote by $\phi_t:\Tan M\to\Tan M$ the associated Hamiltonian flow of $X$. Its flow lines are the speed vectors of the geodesics of $(M,F)$ parametrized with constant speed. In particular, the curve $\tilde\gamma(t):=(\gamma(t),\dot\gamma(t))$ is the periodic orbit of $\phi_t$ corresponding to the closed geodesic $\gamma$. 
Since $G$ is autonomous, the Hamiltonian flow $\phi_t$ preserves each level set $G^{-1}(\ell^2)$. The energy level of $\tilde\gamma$ is  
\[G(\tilde\gamma(t))=\tfrac12 F(\tilde\gamma(t))^2=1/2,\] 
and we denote by $SM:=G^{-1}(1/2)=F^{-1}(1)$ the corresponding energy hypersurface, which is the unit tangent bundle of $(M,F)$. The 1-form $\lambda$ restricts to a contact form $\alpha:=\lambda|_{SM}$, and $X$ restricts to the Reeb vector field of $(SM,\alpha)$. Namely $\alpha(X)=1$ and $\diff\alpha(X,\cdot)=0$. In particular $\phi_t^*\alpha=\alpha$. 
We denote by \[\xi:=\ker(\alpha)\subset\Tan(SM)\] 
the contact distribution of $\alpha$. 
Notice that
\begin{align*}
 \diff\pi(\tilde\gamma(t))\xi_{\tilde\gamma(t)} = \ker(G_v(\tilde\gamma(t))).
\end{align*}
Let $L$ be the vector field on $TM$ defined by
\begin{align*}
 L(x,v)=\tfrac{\diff}{\diff t}\big|_{t=1} (x,tv).
\end{align*}
This is a Liouville vector field, meaning that 
$\diff\lambda(L,\cdot)=\lambda$,
and is transverse to $SM$. Over $SM$, the vector bundle $\Tan(\Tan M)$ splits as a direct sum
\begin{align}
\label{e:splitting}
\Tan(\Tan M)|_{SM}=\mathrm{span}\{X,L\}\oplus\xi,
\end{align}
and this decomposition is symplectically orthogonal, i.e.
\begin{align*}
 \diff\lambda(V,W)=0,\qquad\forall V\in\mathrm{span}\{X,L\},\ W\in \xi.
\end{align*}

We recall that a Jacobi vector field $\zeta:\R\to\gamma^*\Tan M$ is a solution of the Jacobi equation, which is the linearization of the Hamiltonian equation $\dot{\tilde\gamma}=X\circ\tilde\gamma$ at $\tilde\gamma$. In local coordinates, the Jacobi equation reads
\begin{align*}
\partial_t(G_{vv}\,\dot\zeta + G_{xv}\,\zeta) - G_{xx}\,\zeta - G_{vx}\,\dot\zeta=0.
\end{align*}
Here and in the following, the second derivatives $G_{xx}$, $G_{xv}$, $G_{vx}$, $G_{vv}$ are meant to be evaluated at $\tilde\gamma(t)$.
We denote by $\Phi_t:=\diff\phi_t(\tilde\gamma(0)):\Tan_{\tilde\gamma(0)}\Tan M\to\Tan_{\tilde\gamma(t)}\Tan M$ the linearized Hamiltonian flow along $\tilde\gamma$. Its flow lines are lifts of Jacobi vector fields $\zeta$, that is, in local coordinates they can be written as
\begin{align*}
\tilde\zeta(t)=\Phi_t(\tilde\zeta(0))=(\zeta(t),\dot\zeta(t)).
\end{align*}
The linearized flow $\Phi_t$ preserves the splitting~\eqref{e:splitting}. Indeed, $\phi_t^*\alpha=\alpha$ implies that $\Phi_t(\xi)=\xi$. Moreover, $\Phi_t(X)=X$ and $\Phi_t(L)=tX+L$, that is, $\Phi_t|_{\mathrm{span}\{X,L\}}$ can be written in the frame $X,L$ as the symplectic matrix
\begin{align}
\label{e:Phi_t_X_L}
\Phi_t|_{\mathrm{span}\{X,L\}}=
\left(
\begin{array}{@{}cc@{}}
    1 & t \\ 
    0 & 1  
\end{array}
\right)
\in\Sp(2).
\end{align}

Let $W^{1,2}(S^1,\gamma^*\Tan M)$ be the Hilbert space of 1-periodic $W^{1,2}$-vector fields along $\gamma$.
The Hessian $h:=\diff^2E(\gamma)$ is the symmetric bilinear form on $W^{1,2}(S^1,\gamma^*\Tan M)$ given by
\begin{align}
\label{e:index_form}
h(\zeta,\eta)
=
2\int_{S^1}
\!\!
\Big(\langle G_{xx}\,\zeta,\eta\rangle + \langle G_{vx}\,\dot\zeta,\eta\rangle + \langle G_{xv}\,\zeta,\dot\eta\rangle + \langle G_{vv}\,\dot\zeta,\dot\eta\rangle \Big)\diff t.
\end{align}
In this expression, we adopt a common abuse of notation: we write the integrand in local coordinates (this can be made precise by splitting the domain of integration $S^1$ as a finite union of intervals over which the local coordinates are available).
A bootstrap argument, together with an integration by parts, implies that the kernel of $h$ is precisely given by the 1-periodic Jacobi vector fields. In particular $\nul(h)=\dim\ker(\Phi_1-I)$, and therefore $\nul(\gamma)=\nul(h)-1=\dim\ker(\Phi_1|_{\xi_{\tilde\gamma(0)}}-I)$, 
which implies point (i).

We consider the subspace 
\begin{align*}
Z=\big\{\zeta\in W^{1,2}(S^1,\gamma^*\Tan M)\ \big|\ G_v(\gamma,\dot\gamma)\zeta\equiv0 \big\}.
\end{align*}
We claim that
\begin{align*}
\ind(h|_{Z})=\ind(h),\qquad
\nul(h|_Z)=\nul(h)-1.
\end{align*}
Indeed, a straightforward computation shows that the $h$-orthogonal 
\[Z^h:=\big\{\zeta\in W^{1,2}(S^1,\gamma^*\Tan M)\ \big|\ h(\zeta,\cdot)|_{Z}=0\big\}\] 
is precisely the space of those $\zeta\in W^{1,2}(S^1,\gamma^*\Tan M)$ of the form $\zeta(t)=f(t)\dot\gamma(t)$ for some $f:S^1\to\R$, and we have $W^{1,2}(S^1,\gamma^*\Tan M)=Z\oplus Z^h$, $\ind(h|_{Z^h})=0$, $\ker(h|_{Z^h})=\mathrm{span}_{\R}\{\dot\gamma\}$, and $\nul(h|_{Z^h})=1$. From now on, we will simply write $h$ for the restriction $h|_Z$, so that in particular
\begin{align*}
\nul(\gamma)=\nul(h)=\dim\ker(\Phi_1|_{\xi_{\tilde\gamma(0)}}-I). 
\end{align*}
Analogously, if we set
\[Z_0=\big\{\zeta\in Z\ \big|\ \zeta(0)=\zeta(1)=0  \big\},\]
we have 
\begin{align*}
\ind_\Omega(\gamma)=\ind(h|_{Z_0}),
\qquad
\nul_\Omega(\gamma)=\nul(h|_{Z_0}).
\end{align*} 
The kernel of $h|_{Z_0}$ is the space of Jacobi fields $\zeta$ such that $\zeta(0)=\zeta(1)=0$ and $G_v(\gamma,\dot\gamma)\zeta\equiv0$, and thus
\begin{align*}
\nul(h|_{Z_0})
\leq
\dim\ker G_v(\gamma(t),\dot\gamma(t))
=
\dim(M)-1,
\end{align*}
which proves point (ii).

Let us reduce the setting to finite dimension. Let $k\geq2$ be a large enough integer such that no restriction $\gamma|_{[a,b]}$ with $b-a<k^{-1}$ contains conjugate points; namely, there are no Jacobi vector fields along $\gamma$ vanishing on more than one point of $[a,b]$. We consider the finite dimensional vector space $V\subset Z$ of those vector fields $\zeta\in Z$ such that, for all $i=0,...,k-1$, each restriction $\zeta|_{[i/k,(i+1)/k]}$ is a Jacobi vector field. The Morse indices of $h$ and $h|_{V}$ are the same
\begin{align*}
 \ind(h|_V)=\ind(h),\qquad\nul(h|_V)=\nul(h).
\end{align*}
Indeed, an integration by parts in~\eqref{e:index_form} shows that the $h$-orthogonal to $V$ is the subspace
\begin{align*}
V^h
=
\big\{\zeta\in Z\ \big|\ h(\zeta,\cdot)|_{V}=0\big\}
=
\big\{\zeta\in Z\ \big|\ \zeta(\tfrac ik)=0\quad \forall i=0,...,k-1\big\},
\end{align*}
and we have $Z=V\oplus V^h$ and $\ind(h|_{V^h})+\nul(h|_{V^h})=0$. Analogously, if we set $V_0:=V\cap Z_0$, we have
\begin{align}
\label{e:Morse_index_thm}
\ind(h|_{V_0})
&=
\ind(h|_{Z_0}) \\
\nonumber
\nul(h|_{V_0})
&=
\nul(h|_{Z_0}) 
\end{align}
From now on, we will simply write $h$ for the restriction $h|_V$.

The $h$-orthogonal $V_0^h=\big\{\zeta\in V\ \big|\ h(\zeta,\cdot)|_{V_0}=0\big\}$ is precisely the space of vector fields $\zeta\in Z$ such that $\zeta|_{(0,1)}$ is a Jacobi vector field. 
We denote by 
$\Ver:=\ker\diff\pi\subset \Tan\Tan M$
the vertical sub-bundle of $\Tan\Tan M$. Each intersection 
\[(\xi\cap\Ver)_{\tilde\gamma(t)}=(SM\cap\Ver)_{\tilde\gamma(t)}\] has dimension $\dim(M)-1$.
For each $\zeta\in V_0^h$, we set
\begin{align*}
\tilde\zeta(t):=(\zeta(t),\dot\zeta(t))=\Phi_t(\tilde\zeta(0^+)),
\qquad
\forall t\in(0,1).
\end{align*}
Notice that there is an isomorphism 
\begin{align*}
 V_0^h\to (\Phi_1-I)|_{\xi_{\tilde\gamma(0)}}^{-1}(\xi\cap\Ver)_{\tilde\gamma(0)},
 \quad
 \zeta\mapsto\tilde\zeta(0^+).
\end{align*}
In particular
\begin{equation}
\label{e:dim_V_0^h}
\begin{split}
\dim(V_0^h) & \leq \dim\ker(\Phi_1-I)|_{\xi_{\tilde\gamma(0)}} + \dim(\xi\cap\Ver)_{\tilde\gamma(0)}\\
& = \nul(h) + \dim(M)-1. 
\end{split}
\end{equation}
Moreover, the evaluation map $V_0^h\to\ker(G_v(\tilde\gamma(0)))$, $\zeta\mapsto\zeta(0)$ is surjective, and its kernel is precisely $\ker(h|_{V_0})$. Therefore
\begin{equation}
\label{e:dim_V_0^h_precise}
\begin{split}
\dim(V_0^h) & = \dim\ker(h|_{V_0}) + \dim \ker(G_v(\gamma(0),\dot\gamma(0)))\\
& = \nul(h|_{V_0}) + \dim(M)-1. 
\end{split}
\end{equation}
The general formula relating the Morse indices of a quadratic form to the ones of its restriction to a subspace (see, e.g., \cite[Section~A.2]{Mazzucchelli:2015zc}) gives
\begin{align}
\label{e:index_subspaces}
\ind(h)=\ind(h|_{V_0}) + \ind(h|_{V_0^h}) + \nul(h|_{V_0^h}) - \nul(h).
\end{align}
In particular, by~\eqref{e:dim_V_0^h}, we have
\begin{align*}
\ind(h) \leq \ind(h|_{V_0}) + \dim(V_0^h) - \nul(h) \leq \ind(h|_{V_0}) + \dim(M) - 1,
\end{align*}
which proves point~(iii). By~\eqref{e:dim_V_0^h_precise} and~\eqref{e:index_subspaces}, we have
\begin{align*}
\ind(h) + \nul(h) \leq \ind(h|_{V_0}) + \dim(V_0^h) \leq \ind(h|_{V_0}) + \nul(h|_{V_0}) + \dim(M) - 1,
\end{align*}
which proves point~(iv). 

From now on, let us now assume that $M$ is an orientable surface. 
The classical index theorem of Morse \cite{Morse:1996ua} allows us to express $\ind(h|_{Z_0})$ and $\nul(h|_{Z_0})$ as
\begin{align*}
\ind(h|_{V_0}) &= \sum_{t\in(0,1)} \dim\big(\Phi_t(\Ver_{\tilde\gamma(0)})\cap \Ver_{\tilde\gamma(t)}\big),\\
\nul(h|_{V_0}) &= \dim\big(\Phi_1(\Ver_{\tilde\gamma(0)})\cap \Ver_{\tilde\gamma(1)}\big).
\end{align*}
Notice that the Liouville vector field $L$ takes values in the vertical sub-bundle $\Ver$, and Equation~\eqref{e:Phi_t_X_L} implies that
\begin{align*}
\Phi_t(\Ver_{\tilde\gamma(0)})\cap \Ver_{\tilde\gamma(t)}
=
\Phi_t((\xi\cap\Ver)_{\tilde\gamma(0)})\cap (\xi\cap\Ver)_{\tilde\gamma(t)},
\qquad
\forall t\neq0.
\end{align*}
Since the fibers of the bundle $\xi\cap\Ver$ have dimension 1, we can express these index formulas by means of a single vector field $\eta$,  as follows. Let us fix an arbitrary non-zero vector $\tilde\eta_0\in (\xi\cap\Ver)_{\tilde\gamma(0)}$, and define
\begin{align*}
\tilde\eta(t)=(\eta(t),\dot\eta(t)):=\Phi_t(\tilde\eta_0),
\end{align*}
so that $\eta$ is a Jacobi field along $\gamma$ such that $\eta(0)=0$ and $G_v(\gamma,\dot\gamma)\eta\equiv0$. Since $M$ is an orientable surface, the normal bundle of $\gamma$ is trivial, and we can find a nowhere-vanishing 1-periodic smooth vector field $\mu$ along $\gamma$ such that $G_v(\gamma,\dot\gamma)\mu\equiv0$, so that we can express $\eta$ as 
\begin{align*}
\eta(t)=f(t)\mu(t)
\end{align*}
for some smooth function $f:\R\to\R$. Notice that, since $\eta$ is a Jacobi vector field that does not vanish identically, it has isolated zeroes, and in particular $\dot f(t)\neq0$ whenever $f(t)=0$. The index theory of Morse reduces to
\begin{align*}
\ind(h|_{Z_0}) 
&= 
\# \{t\in(0,1)\ |\ f(t)=0\}
,\qquad
\nul(h|_{Z_0})
=
\left\{
  \begin{array}{@{}ll}
    1 &   \mbox{if }f(1)=0, \\ 
    0 &  \mbox{if }f(1)\neq0. \\ 
  \end{array}
\right. 
\end{align*}
If $\nul(h|_{Z_0})=1$, then $\eta(1)=0$, and therefore 
\[\tilde\eta(t+1)= \tfrac{\dot f(1)}{\dot f(0)} \tilde\eta(t),\qquad\forall t\in\R \]
This readily implies that $\nul_\Omega(\gamma^m)=1$ and 
\begin{align*}
\ind_\Omega(\gamma^m)
=
\# \{t\in(0,m)\ |\ f(t)=0\}
=
m\,\ind_\Omega(\gamma)+m-1.
\end{align*}
This proves point (v).

With an integration by parts in~\eqref{e:index_form}, we readily see that the quadratic form $h$ on the space $V_0^h$ can be expressed in local coordinates as 
\begin{equation}
\label{e:h_V_0^h}
\begin{split}
h(\zeta,\zeta)
&=\langle G_{vv}\dot\zeta(1^-) + G_{xv}\zeta(1) - G_{vv}\dot\zeta(0^+) - G_{xv}\zeta(0),\zeta(0)\rangle\\
&= \diff\lambda((\Phi_1-I)\tilde\zeta(0^+),\tilde\zeta(0^+)). 
\end{split}
\end{equation}
Let us assume that $\nul(\gamma)=2$, so that $\Phi_1|_{\xi_{\tilde\gamma(0)}}=I$, $(\Phi_1-I)\tilde\zeta(0^+)=0$ for all $\zeta\in V_0^h$, and $\nul(h|_{V_0})=1$. Equation~\eqref{e:h_V_0^h} implies that $h|_{V_0^h}=0$. Since $\dim(V_0^h)=1+\nul(h|_{V_0})=2$, this implies that $\ind(h|_{V_0^h})=0$ and $\nul(h|_{V_0^h})=2$. Therefore, Equation~\eqref{e:index_subspaces} becomes $\ind(h)=\ind(h|_{V_0})$. Since $\Phi_1|_{\xi_{\tilde\gamma(0)}}=I$, in particular the above Jacobi field $\eta$ is (smoothly) 1-periodic, and so is the function $f$. Therefore, since $f$ has non-zero derivative at its zeroes, it must vanish an odd number of times in the open interval $(0,1)$. Equation~\eqref{e:Morse_index_thm} allows to conclude that $\ind(\gamma)=\ind(h)=\ind(h|_{V_0})$ is odd. We can now repeat the same argument for all the iterates $\gamma^m$, since
\[\nul(\gamma^m)=\dim\ker(\Phi_1|_{\xi_{\tilde\gamma(0)}}^m-I)=2,\]
and conclude that $\ind(\gamma)^m$ is odd as well for all $m\in\N$. This proves point (vi).

Finally, let us assume that $\ind_\Omega(\gamma^m)>0$ for some $m\geq1$, which is equivalent to the fact that the Jacobi field $\eta$ introduced above vanishes at some positive time. Let $\tau>0$ be the minimum $t>0$ such that $\eta(t)=0$. Up to replacing $\mu$ with $-\mu$, we can assume that $f|_{(0,\tau)}>0$, so that $\dot f(0)>0$ and $\dot f(\tau)<0$. 
If $\tau\leq1$, we consider the 1-periodic vector field along $\gamma$
\begin{align*}
\theta(t)=
\left\{
  \begin{array}{@{}ll}
    \eta(t), & \mbox{if }t\in[0,\tau], \\ 
    0, & \mbox{if }t\in[\tau,1],
  \end{array}
\right.
\end{align*}
which satisfies $h(\theta,\theta) = 0$ and
\begin{align*}
h(\theta,\mu) 
& = \langle G_{vv}\,\dot\eta(\tau) - G_{vv}\,\dot\eta(0),\mu(0) \rangle\\
& = 
\dot f(\tau)\, \langle G_{vv}\,\mu(\tau),\mu(\tau)\rangle
-
\dot f(0)\,\langle G_{vv}\,\mu(0),\mu(0)\rangle\\
& < 0.
\end{align*}
For each $\epsilon>0$ the piecewise smooth vector field $\theta+\epsilon\mu$ is 1-periodic and everywhere transverse to $\dot\gamma$. Moreover, 
\begin{align*}
h(\theta+\epsilon\mu,\theta+\epsilon\mu)=2\epsilon h(\mu,\theta) + \epsilon^2 h(\mu\mu)
\end{align*}
which is negative if $\epsilon>0$ is sufficiently small. Assume now that $\tau>1$. In this case, there exists $t>0$ such that $f(t)=f(t+1)>0$, $\dot f(t)>0$ and $\dot f(t+1)<0$. We define $\theta$ to be the 1-periodic vector field along $\gamma$ such that $\theta|_{[t,t+1]}=\eta|_{[t,t+1]}$. Notice that $\theta$ is everywhere transverse to $\dot\gamma$, and satisfies
\begin{align*}
 h(\theta,\theta)
 &=
 \langle G_{vv}\,\dot\eta(t+1),\eta(t+1) \rangle 
 -
 \langle G_{vv}\,\dot\eta(t),\eta(t) \rangle\\
 &=
 \langle G_{vv}\,(\dot\eta(t+1)-\dot\eta(t)),\eta(t) \rangle\\
 &=(\dot f(t+1)-\dot f(t))\, f(t)\, \langle G_{vv}\,\mu(t),\mu(t) \rangle\\
 &<0.
\end{align*}
In both cases, we can approximate $\theta$ with a $C^0$-close 1-periodic smooth vector field $\zeta$. Such a $\zeta$ will still be everywhere transverse to $\dot\gamma$ and will still satisfy $h(\zeta,\zeta)<0$. This completes the proof of point~(vii).
\end{proof}

\subsection{Local homology}
\label{ss:local_homology}

The last index that is usually employed in critical point theory is the local homology, whose construction we now recall for closed geodesics of Finsler manifolds $(M,F)$. Actually, the theory of local homology is very general, and essentially does not see the difference between the Riemannian and the Finsler settings. We refer the reader to \cite{Rademacher:1992sw, Bangert:2010ak,  Asselle:2018aa}, and in particular to \cite[Section~3]{Asselle:2018aa},  for a more comprehensive treatment.

For any $\UU\subset\Lambda$, $U\subset\Lambda_k$, and $\ell>0$, we set
\begin{align*}
\UU^{<\ell}:=\big\{\gamma\in\UU\ \big|\ E(\gamma)<\ell^2\big\},
\qquad
U^{<\ell}:=\big\{\xx\in U\ \big|\ E_k(\xx)<\ell^2\big\}.
\end{align*}
Notice that $\UU^{<\ell}$ and $U^{<\ell}$ are sublevel sets of the energy functional $E$, whereas in~\eqref{e:W<ell} we denoted by $\WW^{<\ell}$ a sublevel set of the length functional $L$. Nevertheless, the notation is consistent: $\WW$ was indeed a subset of the space of unparametrized loops $\Pi$, and if we parametrize any $\gamma\in\WW$ with constant speed and period 1 we have $L(\gamma)^2=E(\gamma)$.

The energy functional $E$ is invariant under the circle action 
\[u\cdot\gamma=\gamma(u+\cdot)\in\Lambda,\qquad\forall u\in S^1,\ \gamma\in\Lambda.\] Therefore, every closed geodesic $\gamma\in\crit(E)\cap E^{-1}(\ell^2)$ (with $\ell>0$) belongs to a circle of critical points of $E$
\[
S^1\cdot\gamma:=\big\{\gamma(u+\cdot)\in\Lambda\ \big|\ u\in S^1\big\}
\] 
A closed geodesic $\gamma$ is said to be isolated when the critical circles of its iterates $S^1\cdot\gamma^m$ are isolated in $\crit(E)$. Under this assumption, the local homology of $\gamma$ and of $S^1\cdot\gamma$ are the relative homology groups
\begin{align*}
C_*(\gamma):=H_*(\Lambda^{<\ell}\cup\{\gamma\},\Lambda^{<\ell}),\qquad
C_*(S^1\cdot\gamma):=H_*(\Lambda^{<\ell}\cup S^1\cdot{\gamma},\Lambda^{<\ell}).
\end{align*}
As we already mentioned, we will specify the coefficient field in the notation only when we will need to employ a specific one.

Even though the energy function $E$ may not be $C^2$, the local homology of an isolated closed geodesic $C_*(\gamma)$ is isomorphic to the local homology of a smooth function on a finite dimensional manifold at an isolated critical point of index $\ind(\gamma)$ and nullity $\nul(\gamma)$. Indeed, if $k\in\N$ is large enough so that the isolated closed geodesic $\gamma\in\crit(E)\cap E^{-1}(\ell^2)$ belongs to $\Lambda_k$, the inclusion induces the homology isomorphism
\begin{align*}
H_*(\Lambda_k^{<\ell}\cup\{\gamma\},\Lambda_k^{<\ell})\toup^{\cong}
C_*(\gamma).
\end{align*}
The energy $E_k=E|_{\Lambda_k}$ is smooth in a neighborhood of the critical point $\gamma$ (indeed, $E_k$ is smooth at all those $\zeta\in\Lambda_k$ such that $\zeta(\tfrac ik)\neq \zeta(\tfrac {i+1}k)$ for all $i\in\Z_k$). Let $\Sigma\subset M$ be an embedded hypersurface intersecting $\gamma$ transversely at $x_0:=\gamma(0)$. We define the smooth hypersurface
\begin{align*}
 \Sigma_k:=\big\{ \zeta\in\Lambda_k\ \big|\ \zeta(0)\in\Sigma \big\}\subset\Lambda_k.
\end{align*}
It turns out that $\gamma$ is an isolated critical point of $E_k|_{\Sigma_k}$ of index $\ind(\gamma)$ and nullity $\nul(\gamma)$, and the inclusion induces the homology isomorphism
\begin{align}
\label{e:local_homology_Sigma_k}
H_*(\Sigma_k^{<\ell}\cup\{\gamma\},\Sigma_k^{<\ell})
\toup^{\cong}
C_*(\gamma),
\end{align}
see \cite[Prop.~3.1]{Asselle:2018aa}.

Since $\gamma$ is an isolated critical point of $E_k|_{\Sigma_k}$, its local homology can also be expressed by means of the so-called Gromoll-Meyer neighborhoods \cite{Gromoll:1969jy}: these are suitable arbitrarily small compact path-connected neighborhoods $W\subset\Sigma_k$ of $\gamma$ such that, for some $\delta'>0$ and for all $\delta\in[0,\delta']$, the inclusion induces the homology isomorphisms
\begin{align}
\label{e:Gromoll_Meyer_nbhd}
H_*(\Sigma_k^{<\ell}\cup\{\gamma\},\Sigma_k^{<\ell}) 
\otup^{\cong}
H_*(W^{<\ell}\cup\{\gamma\},W^{<\ell-\delta}) 
\toup^{\cong} 
H_*(W,W^{<\ell-\delta}).
\end{align}
Indeed, a homotopy inverse of these inclusion can be built by suitably ``pushing'' in the direction given by a pseudo-gradient of the energy functional $E$, see e.g.~\cite[Theorem~5.2]{Chang:1993ng}.
Gromoll-Meyer neighborhoods are particularly useful to prove certain technical statements concerning the local homology. For instance the following one that we will employ in the proof of Corollary~\ref{c:conjugate_points}.

\begin{lem}
\label{l:local_homology_mountain_pass}
Let $\gamma\in\crit(E)\cap E^{-1}(\ell^2)$, with $\ell>0$, be an isolated closed geodesic. Assume that, for any sufficiently small open neighborhood $\UU\subset\Lambda$ or $\UU\subset\Lambda_k$ of $\gamma$, the open subset $\UU^{<\ell}$ is not connected. Then the local homology $C_1(\gamma)$ is non-zero.
\end{lem}

\begin{proof}
We prove the lemma in the infinite dimensional setting of $\Lambda$, the proof for the setting of $\Lambda_k$ being entirely analogous. Let $\UU_0\subset\Lambda$ be an open neighborhood of $\gamma$ such that, for every open neighborhood $\UU\subset\UU_0$ of $\gamma$, the open subset $\UU^{<\ell}$ is not connected. Let $\UU\subset\UU_0$ be one such open neighborhood. We consider the connected components $\UU_1,...,\UU_r\subset\UU^{<\ell}$ such that $\gamma\not\in\overline\UU_i$ for all $i=1,...,r$. The subset $\VV:=\UU\setminus(\overline\UU_1\cup...\cup\overline\UU_r)$ is still an open neighborhood of $\gamma$ contained in $\UU_0$, and therefore $\VV^{<\ell}$ is not connected. 

Let $W\subset\Sigma_k$ be a Gromoll-Meyer neighborhood of $\gamma$ that is small enough so that $W\subset\VV$. We claim that, for each connected component $\VV'\subset\VV^{<\ell}$, we have
\begin{align}
\label{e:VV'_cap_W}
 \VV'\cap W\neq\varnothing.
\end{align}
This implies that $W^{<\ell}$ is not path-connected. Since $W$ is path-connected, the long exact sequence
\begin{equation*}
\begin{tikzcd}[row sep=small]
 ... \arrow[r]
 &
 H_1(W,W^{<\ell}) \arrow[r] 
 &
 H_0(W^{<\ell}) \arrow[r]
 &
 H_0(W) \arrow[r] \equaldown
 &
 ... \\
 & & & 0 & 
\end{tikzcd}
\end{equation*}
implies that $H_1(W,W^{<\ell})$ is non-zero. This latter relative homology group is isomorphic to $C_1(\gamma)$, according to~\eqref{e:local_homology_Sigma_k} and~\eqref{e:Gromoll_Meyer_nbhd}.

It remains to establish~\eqref{e:VV'_cap_W}. Since $\gamma\in\partial\VV'$, there exists $\zeta_0\in\VV'$ arbitrarily close to $\gamma$ and such that $\zeta_0(0)\in\Sigma$. For each $s\in(0,1]$, we define $\zeta_s\in\Lambda$ to be the unique loop such that, for each $i=0,...,k-1$, the segment $\zeta_s|_{[i/k,(i+s)/k]}$ is a length-minimizing geodesic, while $\zeta_s|_{[(i+s)/k,(i+1)/k]}=\zeta_0|_{[(i+s)/k,(i+1)/k]}$. Notice that $s\mapsto\zeta_s$ is a continuous path in $\Lambda$, and $\zeta_1\in\Sigma_k$. Up to choosing the initial loop $\zeta_0$ to be sufficiently close to $\gamma$, every $\zeta_s$ is contained in the neighborhood $\VV$. Since $E(\zeta_s)\leq E(\zeta_0)<\ell$, we actually have that every $\zeta_s$ is contained in $\VV'$. Finally, if we choose $\zeta_0$ to be sufficiently close to $\gamma$, we have that $\zeta_1\in W$.
\end{proof}

The local homology groups of the critical circles of closed geodesics are the ``building blocks'' for the homology of the free loop space $\Lambda$. Indeed, if $\gamma\in\crit(E)\cap E^{-1}(\ell^2)$ is an isolated closed geodesic and the interval $(\ell,\ell+\epsilon)$ does not contain critical values of $E$, the inclusion induces an injective homomorphism
\begin{align*}
 C_*(S^1\cdot\gamma)
 \hookrightarrow
 H_*(\Lambda^{<\ell+\epsilon},\Lambda^{<\ell}),
\end{align*}
see, e.g., \cite[proof of Lemma~4]{Gromoll:1969gh}.

The local homology of an isolated closed geodesic $\gamma$ does not vary (up to a shift in degree) under iterations that preserve the nullity. In particular, if $\ind(\gamma)=\ind(\gamma^m)$ and $\nul(\gamma)=\nul(\gamma^m)$, the iteration map $\psi^m:\Lambda\hookrightarrow\Lambda$, $\psi^m(\zeta)=\zeta^m$ induces the local homology isomorphisms
\begin{align*}
 \psi^m_*:C_*(\gamma)\toup^{\cong} C_*(\gamma^m),
 \qquad
 \psi^m_*:C_*(S^1\cdot\gamma)\toup^{\cong} C_*(S^1\cdot\gamma^m).
\end{align*}
This is actually a consequence of a general Morse theoretic result due to Gromoll-Meyer \cite[Lemma~7]{Gromoll:1969jy}.

The local homology of an isolated closed geodesic often embeds into the local homology of its critical circle. More precisely, the following statement holds. A closed geodesic $\gamma\in\crit(E)\cap E^{-1}(0,\infty)$ is said to be prime when $\gamma=\zeta^m$ if and only if $\zeta=\gamma$ and $m=1$.

\begin{lem}
\label{l:Cgamma_emb_CS1gamma}
If $\gamma\in\crit(E)\cap E^{-1}(0,\infty)$ is an isolated  prime closed geodesic, then for all odd numbers $m\in\N$ the inclusion induces an injective homomorphism
\begin{align*}
C_*(\gamma^m;\Q)
\hookrightarrow
C_*(S^1\cdot\gamma^m;\Q).
\end{align*}
\end{lem}

\begin{proof}
The proof of this fact is rather long, and we only provide its main steps. Let $m>0$ be a positive integer. We denote by $\mu:\Lambda\to\Lambda$ the continuous map $\mu(\gamma)=\gamma(\tfrac1m + \cdot)$. There is an exact sequence
\begin{align}
\label{e:les_loc_hom}
 0 
 \longrightarrow 
 (\mu_*-\mathrm{id})C_*(\gamma^m;\Q)
 \longrightarrow
 C_*(\gamma^m;\Q)
 \longrightarrow
 C_*(S^1\cdot\gamma^m;\Q),
\end{align}
where all the homomorphisms are induced by the inclusion, see~\cite[Lemma~3.4]{Asselle:2018aa}. The homomorphism $\mu_*:C_*(\gamma^m;\Q)\to C_*(\gamma^m;\Q)$ turns out to be equal to 
\begin{align}
\label{e:mu*}
\mu_*=(-1)^{\ind(\gamma^m)-\ind(\gamma)}\mathrm{id}, 
\end{align}
see~\cite[Lemma~3.5]{Asselle:2018aa}. Bott's iteration theory \cite{Bott:1956sp} implies that the Morse indices $\ind(\gamma)$ and $\ind(\gamma^m)$ have the same parity provided $m$ is odd. This, together with~\eqref{e:les_loc_hom} and~\eqref{e:mu*}, implies that the inclusion induces an injective homomorphism $C_*(\gamma^m;\Q)
\hookrightarrow
C_*(S^1\cdot\gamma^m;\Q)$ for all odd integers $m>0$.
\end{proof}

We close this section by proving the following proposition relating the local homology in $\Lambda=\W(S^1,M)$ to the one in $\Pi=\Emb(S^1,M)/\Diff(S^1)$ of Section~\ref{s:simply_closed_geodesics}. By applying the proposition to the three min-max values $\ell(h_i)$, for $i=1,2,3$, defined in~\eqref{e:ell_h_i}, we will infer Theorem~\ref{t:LS}(iii). 

\begin{prop}
\label{p:local_homology_min_max}
Let $(M,F)$ be a closed, orientable, reversible Finsler surface, $\rho>0$ a constant, and $h\in H_d(\Pi,\Pi^{<\rho})$ a non-trivial homology class. Assume that there are only finitely many simple closed geodesics of $(M,F)$ having length in a neighborhood of $\ell(h)$.
Then, there exists a simple closed geodesic $\gamma\in\crit(E)\cap E^{-1}(\ell(h)^{1/2})$ with non-zero local homology $C_d(\gamma)\neq0$.
\end{prop}

\begin{proof}
We first apply Lemma~\ref{l:l_critical} and obtain a simple closed geodesic $\gamma$ of length $\ell:=\ell(h)$ and, for every $\epsilon>0$, an open neighborhood $\VV(\gamma,\epsilon)\subset\Pi$ and a constant $\delta\in(0,\epsilon^2)$ such that the homomorphism \[H_*(\VV(\gamma,\epsilon),\VV(\gamma,\epsilon)^{<\ell-\delta})\to H_*(\Pi,\Pi^{<\ell})\] induced by the inclusion is non-zero. 

Let $\Sigma\subset M$ be an embedded open hypersurface (i.e.\ an open segment) intersecting $\gamma$ transversely. We choose $\Sigma$ and $\epsilon_0>0$ small enough so that every $\zeta\in\VV(\gamma,\epsilon_0)$ intersect $\Sigma$ in a single point and, by the implicit function theorem, the map
$\VV(\gamma,\epsilon_0)\to\Sigma$, $\zeta\mapsto\zeta\cap\Sigma$ is continuous. Throughout this section, we uniquely parametrize every $\zeta\in\VV(\gamma,\epsilon_0)$ so that
\begin{align*}
\zeta:S^1\to M,
\qquad
F(\zeta,\dot{\zeta})\equiv L(\zeta),
\qquad 
\zeta(0)\in\Sigma.
\end{align*}
With this choice of parametrizations, we identify $\VV(\gamma,\epsilon_0)$ with a subset of the space of embeddings $\Emb(S^1,M)$, endowed as usual with the $C^\infty$-topology. Notice that $\VV(\gamma,\epsilon_0)$ is relatively compact in the $C^1$ topology. Moreover, every $C^1$-open neighborhood of $\gamma$ contains $\VV(\gamma,\epsilon)$ for a sufficiently small $\epsilon\in(0,\epsilon_0]$.

Let us consider an embedding $M\hookrightarrow\R^3$, which exists since $M$ is an orientable closed surface. Let $U\subset\R^3$ be a tubular neighborhood of $M$ with associated smooth retraction $\pi:U\to M$. We consider a family of mollifiers $\theta_s(u)=\theta(u/s)/s$, where $s\in(0,1)$ and $\theta:S^1\to[0,\infty)$ is a smooth function supported in $(-1/2,1/2)$ and with integral 1. We denote by $*$ the convolution operation. Since $\VV(\gamma,\epsilon_0)$ is relatively $C^1$-compact and $\theta_s$ tends to the Dirac delta as $s\to0$, there exists $s_0>0$ and $\epsilon_1\in(0,\epsilon_0]$ such that we have a well defined continuous map
\begin{align*}
c:[0,s_0]\times\VV(\gamma,\epsilon_1)\to \Emb(S^1,M),
\qquad
c(s,\zeta)(u)=c_s(\zeta)(u)=\pi(\zeta*\theta_s(u)).
\end{align*}
Notice that $c_0(\zeta)=\zeta$ for all $\zeta\in \VV(\gamma,\epsilon_1)$. Since the length function is continuous on the relatively $C^1$-compact subset $\VV(\gamma,\epsilon_1)$, there exists $s_1\in(0,s_0]$ such that
\begin{align}
\label{e:stretch_convolution_0}
 L(c_s(\zeta))<L(\zeta)+\delta/2,
 \qquad
 \forall s\in[0,s_1],\ \zeta\in\VV(\gamma,\epsilon).
\end{align}
By the continuity of the convolution, there exists an open subset $\UU\subset W^{1,2}(S^1,M)$ containing $\VV(\gamma,\epsilon)$ such that $c_{s_1}$ extends as a continuous map 
\begin{align*}
c_{s_1}:\UU\to \Emb(S^1,M),
\qquad
c_{s_1}(\zeta)(u)=\pi(\zeta*\theta_{s_1}(u)),
\end{align*}
and
\begin{align}
\label{e:stretch_convolution}
 L(c_{s_1}(\zeta))<L(\zeta)+\delta,
 \qquad
 \forall \zeta\in\UU.
\end{align}

We consider an integer 
\begin{align*}
 k > \frac{\ell+\epsilon_1^2}{\injrad(M,F)}
\end{align*}
that we will soon fix, and the space $\Sigma_k$ of broken closed geodesics intersecting $\Sigma$ at time 0. We define a continuous homotopy
\begin{align*}
r_t:\VV(\gamma,\epsilon_1)\to W^{1,2}(S^1,M),\ t\in[0,1],
\end{align*}
as follows: we uniquely parametrize every $\zeta\in\VV(\gamma,\epsilon_0)$ so that
\begin{align*}
F(\zeta,\dot{\zeta})\equiv L(\zeta),
\qquad 
\zeta(0)\in\Sigma;
\end{align*}
for all $i=0,...,k-1$, we define \[r_t(\zeta)|_{[i/k,(i+1-t)/k]}:=\zeta|_{[i/k,(i+1-t)/k]},\] and $r_t(\zeta)|_{[(i+1-t)/k,(i+1)/k]}$ as the shortest geodesic of $(M,F)$ parametrized with constant speed and joining its endpoints. We require $k$ to be large enough so that every $r_t$ has image inside the open subset $\UU\subset W^{1,2}(S^1,M)$. Clearly, 
\[E(r_t(\zeta))\leq E(\zeta)=L(\zeta)^2,\qquad \forall t\in[0,1].\] 

We consider a Gromoll-Meyer neighborhood $W\subset\Sigma_k\cap\UU$ of $\gamma=r_1(\gamma)$. Notice that, by~\eqref{e:stretch_convolution}, we have
\begin{align*}
L(c_{s_1}(\zeta))< L(\zeta)+\delta \leq E(\zeta)^{1/2}+\delta,\qquad\forall\zeta\in W,
\end{align*}
and in particular $c_{s_1}(W^{<\ell-\delta})\subset\Pi^{<\ell}$.
We fix a constant $\epsilon_2\in(0,\epsilon_1]$ small enough so that $r_1(\VV(\gamma,\epsilon_1))\subset W$. Overall, we have the homomorphisms
\begin{equation}
\label{e:commuting_diagram_local_homologies}
\begin{tikzcd}[row sep=large]
  H_*(\VV(\gamma,\epsilon_2),\VV(\gamma,\epsilon_2)^{<\ell-\delta}) 
  \arrow[rr,"i_*"] \arrow[dr,"(r_1)_*"'] 
  & &
  H_*(\Pi,\Pi^{<\ell})
  \\
  &
  H_*(W,W^{<\ell-\delta}) \arrow[ur,"(c_{s_1})_*"']
  &
\end{tikzcd}
\end{equation}
where $i_*$ is the non-zero homomorphism induced by the inclusion (see the first paragraph of the proof). All we need to do in order to complete the proof is to show that the diagram~\eqref{e:commuting_diagram_local_homologies} commutes. This is a consequence of the fact that the inclusion $i$ is homotopic to the composition $c_{s_1}\circ r_1$ via the continuous homotopy
\begin{align*}
h_t:\VV(\gamma,\epsilon_2)\to \Pi,\qquad 
h_t=
\left\{
  \begin{array}{@{}lll}
    c_{2ts_1}, &   \mbox{if }t\in[0,1/2], \\ 
    c_{s_1}\circ r_{2t-1}, &   \mbox{if } t\in[1/2,1], 
  \end{array}
\right.
\end{align*}
which satisfies $h_0=i$, $h_1=c_{s_1}\circ r_{1}$, and  $h_t(\VV(\gamma,\epsilon_2)^{<\ell-\delta})\subset\Pi^{<\ell}$ for all $t\in[0,1]$ according to~\eqref{e:stretch_convolution_0} and~\eqref{e:stretch_convolution}.
\end{proof}

\section{Infinitely many closed geodesics}
\label{s:infinitely_many}

\subsection{The Birkhoff map}
Let $(S^2,F)$ be a reversible Finsler sphere, and $SS^2=\{(x,v)\ |\ F(x,v)=1\}$ its Finsler unit tangent bundle with base projection $\pi:SM\to M$, $\pi(x,v)=x$. As we already recalled in the proof of Proposition~\ref{p:indices}, $SS^2$ admits the contact form $\alpha=G_v\,\diff \pi$, where $G(x,v)=\tfrac12F(x,v)^2$, and the associated Reeb vector field $X$ on $SM$ defined by $\alpha(X)\equiv1$ and $\diff\alpha(X,\cdot)\equiv0$. The flow $\phi_t:SM\to SM$ of $X$ is precisely the geodesic flow of $(S^2,F)$.

Let $\gamma:S^1\hookrightarrow S^2$ a simple closed geodesic of $(S^2,F)$. Without loss of generality, let us assume that $F(\gamma,\dot\gamma)\equiv1$. The complement $S^2\setminus\gamma$ is the disjoint union of two open balls $B_0,B_1\subset S^2$. We consider the open annuli
\begin{align*}
A_i:=\big\{(x,v)\in SS^2\ \big|\ x\in\gamma(S^1),\ v\pitchfork\dot\gamma(t)\mbox{ and points inside }B_i\big\},\quad
i=0,1.
\end{align*}
Since the Reeb vector field $X$ is transverse to $A_i$, we readily see that \[\diff\alpha|_{A_i}=(X\lrcorner(\alpha\wedge\diff\alpha))|_{A_i}\]
is a symplectic form on $A_i$. We assume that the first return time
\begin{align*}
 \tau_i:A_i\to(0,\infty],
 \qquad
 \tau_i(x,v):=\inf\big\{t>0\ \big|\ \phi_t(x,v)\in A_{1-i} \big\}\in(0,+\infty]
\end{align*}
is finite for all $(x,v)\in A_i$ (here, we adopt the usual convention $\inf\varnothing=+\infty$). Under this assumption, there is a well defined first return map
\begin{align*}
 \psi_i:A_i\to A_{1-i},
 \qquad
 \psi_i(x,v)=\phi_{\tau_i(x,v)}(x,v),
\end{align*}
which is a diffeomorphism. Since
\begin{align*}
\psi_i^*\alpha-\alpha
=
\phi_{t}^*\alpha|_{t=\tau_i}+\alpha(\partial_t\phi_t(z))|_{t=\tau_i}\diff\tau_i-\alpha
=
\alpha+\alpha(X)\diff\tau_i-\alpha
=
\diff\tau_i,
\end{align*}
the first return map is an exact symplectomorphism $\psi_i:(A_i,\diff\alpha)\to(A_{1-i},\diff\alpha)$. Notice that 
\[\partial A_0=\partial A_1=\{(\gamma(t),\dot\gamma(t)),(\gamma(t),-\dot\gamma(t)) \ |\ t\in S^1\},\] 
and we readily see that $\diff\alpha|_{\overline A_i}$ vanishes on $\partial A_i$.

For each $t\in [0,1)$, we choose a non-zero $w_t\in\ker\diff\pi(\tilde\gamma(0))$ depending smoothly on $t$, and we extend it to a vector field 
\begin{align}
\label{e:eta} 
\tilde\eta_t(s)=(\eta_t(s),\dot\eta_t(s)):=\diff\phi_{s-t}(\tilde\gamma(t))w_t.
\end{align} 
Namely, $\eta_t$ is a non-trivial Jacobi vector field along $\gamma$ such that $G_v(\gamma,\dot\gamma)\eta_t\equiv0$ and $\eta_t(t)=0$. We recall that the points $\gamma(t),\gamma(s)$, with $t\neq s$, are conjugate when $\eta_t(s)=0$. For each $t\in[0,1)$, we set
\begin{align*}
 t_{-1}:=\sup\{s<t\ |\ \eta_t(s)=0\},
 \qquad
  t_1:=\inf\{s>t\ |\ \eta_t(s)=0\},
\end{align*}
and $t_{\pm2}:=(t_{\pm1})_{\pm1}$ (here, once again, we set $\sup\varnothing=-\infty$ and $\inf\varnothing=+\infty$). Namely, $t_{i}$ is the time of the $|i|$-th conjugate point to $\gamma(t)$ after $t$ if $i>0$, or before $t$ if $i<0$.

\begin{lem}
\label{l:extension}
Assume that, for some $t\in [0,1)$, $t_1$ is finite. Then, for all $t\in[0,1)$ both $t_1$ and $t_{-1}$ are finite, and the first return maps $\psi_i$ can be extended as homeomorphisms 
\begin{align}
\label{e:extension}
\psi_i:\overline A_i\to\overline A_{1-i},
\qquad
\psi_i(\gamma(t),\pm\dot\gamma(t))=(\gamma(t_{\pm1}),\pm\dot\gamma(t_{\pm1})).
\end{align}
\end{lem}

\begin{proof}
Let $\mu$ be a nowhere vanishing 1-periodic vector field along $\gamma$ such that $G_v(\gamma,\dot\gamma)\mu\equiv0$. The Jacobi fields $\eta_t$ can be written as $\eta_t(s)=f(t,s)\mu(s)$ for some smooth function $f:\R\times\R\to\R$. Up to replacing $\mu$ with $-\mu$, we can assume that $f(t,t+\epsilon)>0$ for all $t\in S^1$ and $\epsilon>0$ small enough. Since the Jacobi fields $\eta_t$ are non-trivial, we have $\partial_s f(t,s)\neq0$ whenever $f(t,s)=0$. This readily implies that, if $t_1$ is finite for some $t\in\R$, the same is true for all $t\in\R$, and the function $t\mapsto t_1$ is continuous and monotone increasing. Since $(t_1)_{-1}=t$, we infer that the function $t\mapsto t_{-1}$ is well defined, continuous and monotone increasing as well.

We fix an arbitrary $t\in[0,1)$ and $(x,v):=(\gamma(t),\dot\gamma(t))$. In order to complete the proof, we are left to show that, for each sequence $v_n\in S_{x}S^2$ of vectors pointing inside $A_i$ and such that $v_n\to \pm v$, we have $\tau_i(x,v_n)\to \pm(t_{\pm1}-t)$. Indeed, this implies that $\phi_{\tau_i(x,v_n)}(x,v_n)\to (\gamma(t_{\pm1}),\pm\dot\gamma(t_{\pm1}))$, and therefore the extension~\eqref{e:extension} of $\psi_i$ is continuous and bijective. Since the annuli $\overline A_i$ and $\overline A_{i-1}$ are compact and Hausdorff, such an extension is a homeomorphism.

Let us focus on the case $v_n\to v$, the other one being analogous.
We set $\gamma_n(s):=\exp_x((s-t)\,v_n)$ and $\sigma_n:=\tau_i(x,v_n)$. We claim that 
\[
\liminf_{n\to\infty} \sigma_n\geq t_1-t.
\]
Otherwise, we could extract a subsequence such that $\sigma_n\to\sigma\in(0,t_1-t)$; however, since the geodesic $\gamma|_{[t,t+\sigma]}$ has no conjugate points, this would contradict the fact that the exponential map $\exp_x$ is a local diffeomorphism at $\sigma v$. The fact that $f(t_1+\epsilon)<0$ if $\epsilon>0$ is small enough readily implies that $\gamma_n(t_1+\epsilon)\in A_{1-i}$ for all $n$ large enough, and therefore $\sigma_n<t_1-t+\epsilon$. This implies that $\sigma_n\to t_1-t$. 
\end{proof}

We set $A:=A_0$. The previous lemma implies that the annulus $A$ is a surface of section for the geodesic flow: a surface that is transverse to the vector field $X$ on its interior, and whose boundary is the union of periodic orbits of the flow. The composition $\psi:=\psi_1\circ\psi_0:A\to A$ is the first return map of the surface of section $A$, and extends to a homeomorphism of $\overline A$ as
\begin{align*}
\psi(\gamma(t),\pm\dot\gamma(t))=(\gamma(t_{\pm2}),\pm\dot\gamma(t_{\pm2})).
\end{align*}
As customary in the Riemannian literature, we will call $\psi$ the Birkhoff map of $\gamma$.
With a suitable change of coordinates, $A$ becomes the standard symplectic annulus.

\begin{figure}
\begin{center}
\begin{small}
%% Creator: Inkscape 1.0beta1 (32d4812, 2019-09-19), www.inkscape.org
%% Accompanies image file 'sigma.pdf' (pdf, eps, ps)
\begingroup%
  \makeatletter%
  \providecommand\color[2][]{%
    \errmessage{(Inkscape) Color is used for the text in Inkscape, but the package 'color.sty' is not loaded}%
    \renewcommand\color[2][]{}%
  }%
  \providecommand\transparent[1]{%
    \errmessage{(Inkscape) Transparency is used (non-zero) for the text in Inkscape, but the package 'transparent.sty' is not loaded}%
    \renewcommand\transparent[1]{}%
  }%
  \providecommand\rotatebox[2]{#2}%
  \newcommand*\fsize{\dimexpr\f@size pt\relax}%
  \newcommand*\lineheight[1]{\fontsize{\fsize}{#1\fsize}\selectfont}%
  \ifx\svgwidth\undefined%
    \setlength{\unitlength}{113.28373596bp}%
    \ifx\svgscale\undefined%
      \relax%
    \else%
      \setlength{\unitlength}{\unitlength * \real{\svgscale}}%
    \fi%
  \else%
    \setlength{\unitlength}{\svgwidth}%
  \fi%
  \global\let\svgwidth\undefined%
  \global\let\svgscale\undefined%
  \makeatother%
  \begin{picture}(1,0.82715617)%
    \lineheight{1}%
    \setlength\tabcolsep{0pt}%
    \put(0,0){\includegraphics[width=\unitlength,page=1]{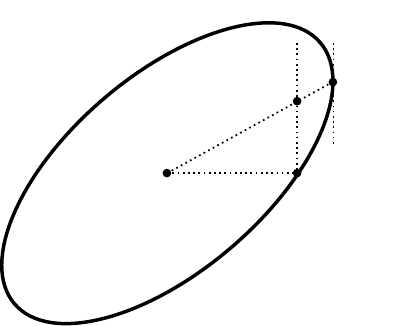}}%
    \put(0.86689061,0.60097931){\color[rgb]{0,0,0}\makebox(0,0)[lt]{\lineheight{1.25}\smash{\begin{tabular}[t]{l}$v$\end{tabular}}}}%
    \put(0.78331526,0.36059106){\color[rgb]{0,0,0}\makebox(0,0)[lt]{\lineheight{1.25}\smash{\begin{tabular}[t]{l}$v_0$\end{tabular}}}}%
    \put(0.63481157,0.57935263){\color[rgb]{0,0,0}\makebox(0,0)[lt]{\lineheight{1.25}\smash{\begin{tabular}[t]{l}$sv$\end{tabular}}}}%
    \put(0.36684054,0.41538721){\color[rgb]{0,0,0}\makebox(0,0)[lt]{\lineheight{1.25}\smash{\begin{tabular}[t]{l}$0$\end{tabular}}}}%
    \put(0.34131421,0.77281257){\color[rgb]{0,0,0}\makebox(0,0)[lt]{\lineheight{1.25}\smash{\begin{tabular}[t]{l}$S_xM$\end{tabular}}}}%
  \end{picture}%
\endgroup%
 
\caption{The value of $s\in[0,1]$ such that $v_0-s\,v\in\ker G_v(x,v)$ can be found geometrically: the parallel to the tangent line $\Tan_v (S_xM)$ passing through $v_0$ intersects the segment joining the origin and $v$ at $sv$.}
\label{f:sigma}
\end{small}
\end{center}
\end{figure}
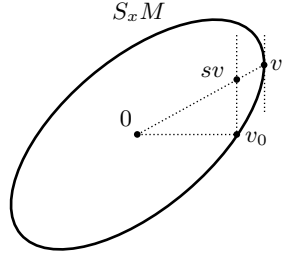

\begin{lem}
\label{l:annulus}
There exists a smooth homeomorphism \[\sigma:S^1\times[-1,1]\to\overline A\] 
of the form 
$\sigma(t,s)=(\gamma(t),\nu(t,s))$,
where $\nu(t,s)\in S_{\gamma(t)}M$ is the unique tangent vector (pointing inside $B_0$ or tangent to $\gamma$) such that \[\dot\gamma(t)-s\,\nu(t,s)\in\ker G_v(\gamma(t),\nu(t,s)).\] 
The map $\sigma$ restrict to a diffeomorphism $\sigma:S^1\times(-1,1)\to A$, and
$\sigma^*\alpha= s\,\diff t$.
\end{lem}

\begin{proof}
We fix $x=\gamma(t)$ and $v_0=\dot\gamma(t)$. For each $v\in S_xM$ there is a unique $s(v)\in[-1,1]$ such that $v_0-s(v)v\in\ker G_v(x,v)$, see Figure~\ref{f:sigma}. Clearly,  $s(v)$ depends smoothly on $v$. We choose an arbitrary parametrization of the fiber \[v:[0,1]\toup^{\cong} \overline A\cap S_xM\] such that $v(0)=v_0$ and $v(1)=-v_0$. Notice that $\dot v(r)\in\ker G_v(x,v(r))$, and there exists $\lambda(r)\in\R$ such that $v_0=s(v(r))v(r)+\lambda(r)\dot v(r)$. By the strict convexity of $S_xM$, we have $\lambda(r)=0$ if and only if $r\in\{0,1\}$. Since \[G_v(x,v)v_0=G_v(x,v)s(v)v=s(v),\] we have
\begin{align*}
 \diff s(v(r))\dot v(r) 
 & = 
 \tfrac{\diff}{\diff r} s(v(r))
 =
 \tfrac{\diff}{\diff r} G_v(x,v(r))v_0
 =
 G_{vv}(x,v(r))\dot v(r)\,v_0\\
 & =
 s(v(r)) G_{vv}(x,v(r))\dot v(r)\,v(r)
 +
 \lambda(r) G_{vv}(x,v(r))\dot v(r)\,\dot v(r)\\
 & =
 s(v(r)) G_{v}(x,v(r))\dot v(r)
 +
 \lambda(r) G_{vv}(x,v(r))\dot v(r)\,\dot v(r)\\
 & =
  \lambda(r) G_{vv}(x,v(r))\dot v(r)\,\dot v(r).
\end{align*}
The last term is non-zero for all $r\in(0,1)$. Therefore, $s:\overline A\cap S_xM\to[-1,1]$ is a diffeomorphism that restricts to a diffeomorphism $s:A\cap S_xM\to(-1,1)$. We set $\sigma(x,\cdot)$ to be the inverse homeomorphism. The obtained map $\sigma:S^1\times[-1,1]\to\overline A$ is thus a homeomorphism that restricts to a diffeomorphism  $\sigma:S^1\times(-1,1)\to A$. The pull-back of the contact form $\alpha$ by $\sigma$ is
\[
 (\sigma^*\alpha)_{(t,s)}
 =
 G_v(\gamma(t),\nu(t,s))\dot\gamma(t)\,\diff t 
 =
 G_v(\gamma(t),\nu(t,s))s\,\nu(t,s)\,\diff t 
 = 
 s\,\diff t.
 \qedhere
\]
\end{proof}

From now on, the annulus $S^1\times[-1,1]$ will be implicitly equipped with the Euclidean area form $\diff s\wedge\diff t$. By means of Lemma~\ref{l:annulus}, we will always consider the Birkhoff map of a simple closed geodesic $\gamma$ as a homeomorphism $\psi:S^1\times[-1,1]\to S^1\times[-1,1]$ that restricts to a symplectomorphism of $(S^1\times(-1,1),\diff s\wedge\diff t)$ and acts on the boundary
as $\psi(t,\pm1)=(t_{\pm2},\pm1)$.

\subsection{Periodic points of twist maps}

Let $\psi:S^1\times[-1,1]\to S^1\times[-1,1]$ be an area preserving homeomorphism preserving the boundary components $S^1\times\{-1\}$ and $S^1\times\{1\}$. Such a $\psi$ is called a \textbf{twist map} when it admits a lift 
\begin{align}
\label{e:lift_twist}
\widetilde\psi:\R\times[-1,1]\to\R\times[-1,1],
\qquad
\widetilde\psi(t,s)=(a(t,s),b(t,s)),
\end{align}
satisfying the twist conditions $a(t,1)<t$ and $a(t,-1)>t$ for all $t\in\R$. If $\psi$ is the Birkhoff map of a simple closed geodesic $\gamma$ of $(S^2,F)$, the set of its periodic orbits in $S^1\times(-1,1)$ is in one-to-one correspondence with the set of closed geodesics intersecting $\gamma$ (other than $\gamma$ itself). In particular, the existence of infinitely many periodic points of $\psi$ implies the existence of infinitely many closed geodesics on $(S^2,F)$

By the celebrated Poincar\'e-Birkhoff theorem \cite{Birkhoff:1966xb}, any twist map has at least two fixed points in the interior of the annulus. Indeed, more is true: any lift~\eqref{e:lift_twist} satisfying the twist condition has at least two fixed points.
A simple argument due to Neumann \cite{Neumann:1977aa} further implies that any twist map $\psi$ has infinitely many periodic points. Indeed, consider the translation 
\[\tau:\R\times[-1,1]\to\R\times[-1,1], \qquad \tau(t,s)=(t+1,s).\]
For each integer $q>0$ there exists another relatively prime integer $p>0$ that is large enough so that
\begin{align*}
p \min_{x\in\R} \big( x-a(x,1) \big) > q.
\end{align*}
This condition guarantees that $\widetilde\phi:=\widetilde\psi^p\circ\tau^{q}$ is a lift of $\phi=\psi^p$ satisfying the twist condition, and therefore has at least a fixed point $z\in\R\times[-1,1]$. Such a $z$ projects to a $p$-periodic point $[z]$ of $\psi$, and since $p,q$ are relatively prime the minimal period of $[z]$ is $p$.

Let us now apply this results to the Birkhoff map of a simple closed geodesic $\gamma$ of $(S^2,F)$. For each $t\in S^1$, we denote $t\cdot\gamma:=\gamma(t+\cdot)$ the closed geodesic $\gamma$ with the parametrization translated by $t$, and by $\ind_\Omega(t\cdot\gamma)$ the Morse index of $t\cdot\gamma$ in its corresponding based loop space $\Omega_t=\{\zeta\in\Lambda\ |\ \zeta(0)=\gamma(t)\}$ (see Section~\ref{ss:index}). 

\begin{thm}
\label{t:twist}
If the simple closed geodesic $\gamma$ satisfies $\ind_\Omega(t\cdot\gamma)\geq2$ for all $t\in S^1$ and has a well defined Birkhoff map $\psi$, then $\psi$ is a twist map, and in particular $(S^2,F)$ has infinitely many closed geodesics.
\end{thm}

\begin{proof}
Let us consider the family of Jacobi fields $\eta_t$ introduced in~\eqref{e:eta}. As we already mentioned in the proof of Proposition~\ref{p:indices}, the classical Morse index theorem \cite{Morse:1996ua} allows to relate $\ind_\Omega(t\cdot\gamma)$ to the zeros of $\eta_t$ by
\begin{align*}
\ind_\Omega(t\cdot\gamma)
=
\#\{s\in(t,t+1)\ |\ \eta_t(s)=0\}
=
\#\{s\in(t-1,t)\ |\ \eta_t(s)=0\}.
\end{align*}
Therefore, $\ind_\Omega(t\cdot\gamma)\geq2$ is equivalent to $t_{2}-t<1$ and $t-t_{-2}<1$. If this holds for all $t\in\R$, we claim that the Birkhoff map $\psi$ is a twist map. Indeed, $\psi$ can be lifted to a continuous map
\begin{align*}
\widetilde\psi:\R\times[-1,1]\to\R\times[-1,1],
\end{align*}
as follows. Let $\sigma:S^1\times[-1,1]\to\overline A$, $\sigma(t,s)=(\gamma(t),\nu(t,s))$ be the homeomorphism of Lemma~\ref{l:annulus}. For each $(t,s)\in \R\times(-1,1)$, we consider the geodesic ray $\zeta$ starting at $\zeta(0)=\gamma(t)$ with speed $\dot\zeta(0)=\nu(t,s)$. Let $a',a''\in(0,1]$ be such that the first intersection of $\zeta$ at positive time with $\gamma$ is at $\gamma(t+a')$, and the second one is at $\gamma(t+a'+a'')$. Let $0<b'<b''$ be the first positive times such that $\zeta(b')=\gamma(t+a')$ and $\zeta(b'')=\gamma(t+a'+a'')$. We denote by $i',i''\in\Z$ the algebraic count of (transverse) self-intersections of the geodesics $\zeta|_{(0,b')}$ and $\zeta|_{(b',b'')}$ respectively; here a double-point intersection is counted positively if and only if $\zeta$ crosses itself from left to right (up to isotoping $\zeta|_{[0,b'']}$ without moving $\zeta(0)$, $\zeta(b')$, and $\zeta(b'')$, we can assume that all the self-intersections of $\zeta|_{[0,b'']}$ are double points). 
\begin{figure}
\begin{center}
\begin{tiny}
%% Creator: Inkscape 1.0beta1 (32d4812, 2019-09-19), www.inkscape.org
%% Accompanies image file 'lift.pdf' (pdf, eps, ps)
\begingroup%
  \makeatletter%
  \providecommand\color[2][]{%
    \errmessage{(Inkscape) Color is used for the text in Inkscape, but the package 'color.sty' is not loaded}%
    \renewcommand\color[2][]{}%
  }%
  \providecommand\transparent[1]{%
    \errmessage{(Inkscape) Transparency is used (non-zero) for the text in Inkscape, but the package 'transparent.sty' is not loaded}%
    \renewcommand\transparent[1]{}%
  }%
  \providecommand\rotatebox[2]{#2}%
  \newcommand*\fsize{\dimexpr\f@size pt\relax}%
  \newcommand*\lineheight[1]{\fontsize{\fsize}{#1\fsize}\selectfont}%
  \ifx\svgwidth\undefined%
    \setlength{\unitlength}{121.82205148bp}%
    \ifx\svgscale\undefined%
      \relax%
    \else%
      \setlength{\unitlength}{\unitlength * \real{\svgscale}}%
    \fi%
  \else%
    \setlength{\unitlength}{\svgwidth}%
  \fi%
  \global\let\svgwidth\undefined%
  \global\let\svgscale\undefined%
  \makeatother%
  \begin{picture}(1,0.80647095)%
    \lineheight{1}%
    \setlength\tabcolsep{0pt}%
    \put(0,0){\includegraphics[width=\unitlength,page=1]{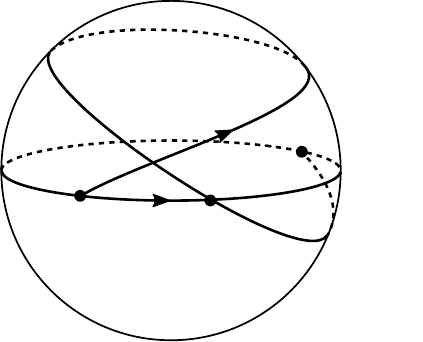}}%
    \put(0.14220831,0.27610225){\color[rgb]{0,0,0}\makebox(0,0)[lt]{\lineheight{1.25}\smash{\begin{tabular}[t]{l}$\gamma(t)$\end{tabular}}}}%
    \put(0.47466111,0.37460668){\color[rgb]{0,0,0}\makebox(0,0)[lt]{\lineheight{1.25}\smash{\begin{tabular}[t]{l}$\gamma(t')$\end{tabular}}}}%
    \put(0.64176689,0.47794835){\color[rgb]{0,0,0}\makebox(0,0)[lt]{\lineheight{1.25}\smash{\begin{tabular}[t]{l}$\gamma(t'')$\end{tabular}}}}%
    \put(0.47378147,0.52412225){\color[rgb]{0,0,0}\makebox(0,0)[lt]{\lineheight{1.25}\smash{\begin{tabular}[t]{l}$\zeta$\end{tabular}}}}%
    \put(0.37615704,0.2637892){\color[rgb]{0,0,0}\makebox(0,0)[lt]{\lineheight{1.25}\smash{\begin{tabular}[t]{l}$\gamma$\end{tabular}}}}%
  \end{picture}%
\endgroup%
 
\end{tiny}
\caption{Example of geodesic ray $\zeta$ on a reversible Finsler sphere intersecting the simple closed geodesic $\gamma$ at subsequent points $\gamma(t)$, $\gamma(t')$ and $\gamma(t'')$, with $t'=t+a'\in(t,t+1]$ and $t''=t'+a''\in(t',t'+1]$, such that $i'=1$ and $i''=0$.}
\label{f:lift}
\end{center}
\end{figure}
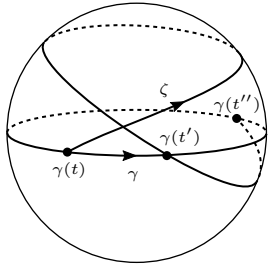
We define the lift
\begin{align*}
\widetilde\psi(t,s)=(a(t,s),b(t,s)),
\end{align*}
by setting the first component to be
\begin{align*}
a(t,s)=t+a'+a''+i'+i''-1.
\end{align*}
It is straightforward to verify that such a function $a$ is continuous.
Since $\ind_\Omega(t\cdot\gamma)\geq2$, if $|s|$ is close to 1 (that is, if $\dot\zeta(0)$ is close to $\pm\dot\gamma(0)$), we have $i'=i''=0$. Moreover, if $s$ is close to $1$ we have $a'+a''\in(0,1)$, whereas if $s$ is close to $-1$ we have $a'+a''\in(1,2)$. Therefore
\begin{align*}
a(t,1)-t  = t_2-t-1\in(-1,0),\qquad
a(t,-1)-t  =1- (t-t_{-2})\in(0,1),
\end{align*}
namely $\widetilde\psi$ satisfies the twist condition.
Since $\psi$ is a twist map, it has infinitely many periodic points corresponding to infinitely many closed geodesics of $(S^2,F)$.
\end{proof}

\subsection{Hingston's theorems}
A celebrated theorem due to Hingston \cite{Hingston:1993ou}, that extends previous results of Bangert \cite{Bangert:1980ho, Bangert:1993wo}, implies the existence of infinitely many closed geodesics on $(S^2,F)$ when there is a simple closed geodesic with non-zero local homology in degree 3 and a Birkhoff map not of twist type. Hingston's original proof was phrased for Riemannian manifolds, but is valid as well in the Finsler setting, and indeed even in the non-reversible Finsler setting. We include the full argument here for the reader's convenience.

\begin{thm}
\label{t:Hingston}
Let $(M,F)$ be a closed Finsler manifold of dimension $d\geq2$, and $\gamma\in\crit(E)$ a closed geodesic satisfying the following two conditions:
\begin{itemize}
\item[(i)] The local homology $C_{i}(\gamma)$ with coefficient in some field is non-zero in degree $i=\ind(\gamma)+\nul(\gamma)$.
\item[(ii)] $\ind(\gamma^m)+\nul(\gamma^m)\leq m(\ind(\gamma)+\nul(\gamma))-(d-1)(m-1)$ for all $m\in\N$.
\end{itemize}
Then $(M,F)$ has infinitely many closed geodesics.
\end{thm}

The proof of Theorem~\ref{t:Hingston} will employ the following arithmetic statement, which is also due to Hingston.

\begin{lem}
\label{l:arithmetic_existence}
Let $(M,F)$ be a Finsler manifold, $\ell>0$ a positive real number, and $\K\subseteq\N$ a $k$-dense subset  for some $k>0$, that is, 
\[(n-k,n+k)\cap\K\neq\varnothing,\qquad\forall n\in\N.\] 
Assume that for all $\epsilon>0$ there exists $\overline m>0$ and, for all $m\in\K$ with $m\geq\overline m$, a closed geodesic $\zeta_m\in\crit(E)$ whose length satisfies
\[m\ell < E(\zeta_m)^{1/2}\leq m\ell+\epsilon.\] Then $(M,F)$ has infinitely many closed geodesics.
\end{lem}

\begin{proof}
Up to replacing the Finsler metric $F$ with $\ell^{-1}F$, we can assume that \[\ell=1.\]
We recall that a closed geodesic $\gamma\in\crit(E)$ is called prime when it is not the iterate of another closed geodesic, that is, $\gamma=\zeta^m$ implies $m=1$ and $\zeta=\gamma$. We prove the lemma by contradiction, by assuming that $(M,F)$ has only finitely many prime closed geodesics $\gamma_1,...,\gamma_r\in\crit(E)$. We denote by $\ell_i:=E(\gamma_i)^{1/2}$ the length of such closed geodesics, and we reorder them so that $\ell_i \not\in\Q$ for all $i=1,...,s$, and $\ell_i=p_i/q_i \in\Q$ for all $i=s+1,...,r$. Here, $p_i$ and $q_i$ are positive integers. 

If $s>0$, that is, there are closed geodesics of irrational length, we set 
\begin{align*}
c:= \frac{(s+1)(2k-2)+1}{\min\{\ell_1,...,\ell_s\}},
\end{align*}
and 
\begin{align*}
\delta_1 & := \min\Big\{|m_1\ell_i-m_2|\ \Big|\ m_1,m_2\in\N\mbox{ with }m_1\leq c,\ i=1,...,s\Big\}>0.
\end{align*}
If $s<r$, that is, there are closed geodesics of rational length, we have
\begin{align*}
 \delta_2 & :=\min\Big\{ \big|m_1\ell_i-m_2\big|\ \Big|\   m_1,m_2\in\N \mbox{ with }m_1\ell_i-m_2\neq0,\ i=s+1,...,r\Big\}\\
 &  \geq \frac1{\max\{q_{s+1},...,q_{r}\}}>0.
\end{align*}
If $s=0$ we simply set $\delta_1:=+\infty$, and likewise if $s=r$ we set $\delta_2:=+\infty$.

Now, we fix $\epsilon\in(0,\min\{\delta_1,\delta_2,1\})$. By the assumptions of the lemma, there exists $\overline m$ and, for all $m\in\K$ with $m\geq\overline m$, a closed geodesic $\zeta_m\in\crit(E)$ such that
\begin{align}
\label{e:final_inequality}
m 
< 
E(\zeta_m)^{1/2}
\leq m+\epsilon.
\end{align}
Since $\gamma_1,...,\gamma_r$ are the only prime closed geodesics in $(M,F)$, up to a shift in the parametrization each closed geodesic $\zeta_m$ must be of the form 
$\zeta_m=\gamma_{i}^{\mu}$ for some $i=i(m)\in\{1,...,r\}$ and $\mu=\mu(m)\in\N$. In particular, the length of $\zeta_m$ is
\begin{align*}
E(\zeta_m)^{1/2} = \mu(m)\ell_{i(m)}.
\end{align*}
We must have $i(m)\leq s$, that is, every $\gamma_{i(m)}$ must have irrational length $\ell_{i(m)}$; indeed, the inequality \eqref{e:final_inequality} implies
\begin{align*}
 0<\mu(m)\ell_{i(m)}-m<\epsilon<\delta_2.
\end{align*}
Since $\K$ is $k$-dense in $\N$, for any $m_1\in\K$ there exists $m_2\in\K$ such that 
\[m_1 < m_2 \leq m_1+2k-2.\]
This, together with the fact that $\gamma_1,...,\gamma_s$ are the only prime closed geodesics with irrational length, implies that we can find two integers $m_1,m_2\in\K$ both larger than $\overline m$ and such that $m_1<m_2\leq m_1+(s+1)(2k-2)$ and $i:=i(m_1)=i(m_2)$. The inequalities in~\eqref{e:final_inequality} applied to these two integers give
\begin{align}
\label{e:final_inequality_2}
m_1
< 
\mu(m_1)\ell_{i} 
\leq m_1+\epsilon,
\qquad
m_2
< 
\mu(m_2)\ell_{i} 
\leq m_2+\epsilon.
\end{align}
Therefore
\begin{align*}
| (\mu(m_2)-\mu(m_1))\ell_{i} - (m_2 - m_1) | \leq \epsilon < \delta_1,
\end{align*}
which implies 
\[\mu(m_2)-\mu(m_1)>c\] 
according to the definition of $\delta_1$. This gives a contradiction, since the inequalities~\eqref{e:final_inequality_2} imply
\[
\mu(m_2)-\mu(m_1)<\frac{m_2-m_1+\epsilon}{\ell_i}<\frac{(s+1)(2k-2)+1}{\min\{\ell_1,...,\ell_r\}}=c.
\qedhere
\]

\end{proof}

\begin{proof}[Proof of Theorem~\ref{t:Hingston}]
Condition~(ii), together with Lemma~\ref{l:elementary_ind}(ii,v), implies that
\begin{align*}
 \ind_\Omega(\gamma)+\nul_\Omega(\gamma)
 \leq
 \tfrac1m\big(\ind_\Omega(\gamma^m)+\nul_\Omega(\gamma^m)\big)
 \leq 
 \ind(\gamma)+\nul(\gamma) - \tfrac{m-1}{m}(d-1).
\end{align*}
In the limit $m\to\infty$ the latter term converges to $d-1$. This, together with the opposite inequality provided by Proposition~\ref{p:indices}(iv), implies 
\begin{align}
\label{e:relation_indices_Hingston}
 \ind_\Omega(\gamma)+\nul_\Omega(\gamma)
 =
 \ind(\gamma)+\nul(\gamma) - (d-1).
\end{align}

Let $\Sigma\subset M$ be an embedded hypersurface diffeomorphic to a compact $(d-1)$-dimensional disk intersecting in its interior the closed geodesic $\gamma$ transversely at $\gamma(0)$. As in Section~\ref{ss:local_homology}, we introduce the space
\begin{align*}
 \Sigma_k:=\{\zeta\in\Lambda_k\ |\ \zeta(0)\in\Sigma\},
\end{align*}
for an integer $k$ large enough so that $\gamma\in\Lambda_k$. Since we are looking for infinitely many closed geodesics, we can assume that $\gamma$ is an isolated closed geodesic.
Therefore, $\gamma$ is an isolated critical point of the restricted energy functional $E|_{\Sigma_k}$ with non-trivial local homology
\begin{align*}
H_{i}(\Sigma_k^{<\ell}\cup\{\gamma\},\Sigma_k^{<\ell})
\cong
H_{i}(\Lambda^{<\ell}\cup\{\gamma\},\Lambda^{<\ell})=C_{i}(\gamma),
\end{align*}
where $\ell^2:=E(\gamma)$. 

Since $i=\ind(\gamma)+\nul(\gamma)$, by the Morse-Gromoll-Meyer lemma \cite{Gromoll:1969jy} there exists a smooth embedded ball $B\subset\Sigma_k$ of dimension $i$ containing $\gamma$ in its interior, and such that $E|_{B\setminus\{\gamma\}}<\ell^2$. Any such ball $B$ represents a generator of the local homology $C_{i}(\gamma)$. This can be easily seen as follows. Let $N\subset\Sigma_k$ be a tubular neighborhood of $B$ diffeomorphic to the normal bundle of $B$ in $\Sigma_k$. The restriction of the energy functional $E$ to any fiber $F$ of $N$ has a non-degenerate local minimizer at $F\cap B$. Thus, the local homology of $E$ at $\gamma$ is isomorphic to the local homology of $E|_B$ at its local maximizer $\gamma$, and the local homology at a local maximizer is generated by the relative cycle covering the whole domain (see, e.g., \cite[Proposition~2.6]{Mazzucchelli:2013co} for a detailed proof of this general Morse-theoretic fact). This argument is independent of the choice of the coefficient field, and in particular $[B]\neq0$ in $C_i(\gamma;\Q)$ as well.

We consider the evaluation map $\ev:B\to \Sigma$, $\ev(\zeta)=\zeta(0)$, whose differential has the form
\[\diff\,\ev(\gamma):\Tan_\gamma B\to\Tan_{\gamma(0)}\Sigma, \qquad\diff\,\ev(\gamma)\xi=\xi(0).\] 
We claim that
$\diff\,\ev(\gamma)$ is surjective.
Indeed, if as usual $\Omega=\big\{\zeta\in\Lambda\ |\ \zeta(0)=\gamma(0)\big\}$ denotes the based loop space, we have
\[\ker(\diff\,\ev(\zeta))=\Tan_\gamma\Omega\cap\Tan_{\gamma}B.\] Since the Hessian $\diff^2 E(\gamma)$ is negative semi-definite on $\Tan_{\gamma}B$, Equation~\eqref{e:relation_indices_Hingston} implies
\begin{align*}
 \dim\ker(\diff\,\ev(\zeta))
 &\leq
 \ind_{\Omega}(\gamma)+\nul_\Omega(\gamma)\\
 &\leq
 \ind(\gamma)+\nul(\gamma)-(d-1)\\
 &=
 \dim(B)-(d-1).
\end{align*}
Since $\dim(\Sigma)=d-1$, we infer that $\diff\,\ev(\gamma)$ is surjective.
By the implicit function theorem, up to shrinking $B$ around $\gamma$, we find a diffeomorphism $\phi:\Sigma\times U\to B$ such that $\ev\circ\phi(x,y)=x$.

If $\zeta_i:[0,\tau_i]\to M$ are continuous paths such that $\zeta_i(\tau_i)=\zeta_{i+1}(0)$ for all $i\in\Z_m$, we define $\zeta:=\zeta_0*...*\zeta_{m-1}\in\Lambda$ to be the 1-periodic curve obtained by first concatenating the $\zeta_i$'s with their original parametrization, and then by linearly reparametrizing the resulting curve so that it becomes 1-periodic. Namely, \[\zeta(t)=\tilde\zeta((\tau_0+...+\tau_{m-1})t),\] 
where
\begin{align*}
\tilde\zeta(\tau_0+...+\tau_{i-1}+u)=\zeta_i(u),
\qquad\forall u\in[0,\tau_i].
\end{align*}
If the $\zeta_i$'s are $W^{1/2}$ paths, the energy of $\zeta_0*...*\zeta_{m-1}$ is 
\begin{align}
\label{e:energy_concatenation}
E(\zeta_0*...*\zeta_{m-1})
=
(\tau_0+...+\tau_{m-1})
\sum_{i=0}^{m-1}
\int_0^{\tau_i} F(\zeta_i,\dot\zeta_i)^2\,\diff t.
\end{align}

We now employ $\phi$ to construct a relative cycle representing a non-zero element of the local homology group of $\gamma^m$. We first define the smooth embedding
\begin{align*}
\phi_m:\Sigma\times U^{\times m}\hookrightarrow\Sigma_{mk},
\qquad
\phi_m(x,y_0,...,y_{m-1})=\phi(x,y_0) * ... * \phi(x,y_{m-1}),
\end{align*}
where $U^{\times m}=U\times...\times U$ denotes the $m$-fold cartesian product. The fact that that $\phi_m$ is a smooth embedding can be easily seen if we identify the loops $\zeta_i=\phi_i(x,y_i)\in\Sigma_k$ with the tuple $\xx_i=(\zeta_i(0),\zeta_i(\tfrac 1k),...,\zeta_i(\tfrac{k-1}k))$ as explained in Section~\ref{ss:energy}: indeed, the curve $\phi_m(x,y_0,...,y_{m-1})\in\Sigma_{mk}$ is then identified with the juxtaposition $(\xx_0,...,\xx_{m-1})$.
The image of $\phi_m$ is a smooth embedded ball 
\[B_m:=\phi_m(\Sigma\times U^{\times m})\subset\Sigma_{mk}\] 
containing $\gamma^m$ in its interior. By assumption~(ii) of the lemma, its dimension is bounded from below as
\begin{align}
\label{e:dim_B_m}
 \dim(B_m) = d-1+m(i-(d-1))\geq\ind(\gamma^m)+\nul(\gamma^m).
\end{align}
Since our $\zeta_i$'s are 1-periodic loops (that is, we consider them as closed paths parametrized on $[0,1]$), Equation~\eqref{e:energy_concatenation} reduces to  \[E(\zeta_0*...*\zeta_{m-1})=m\big(E(\zeta_0)+...+E(\zeta_{m-1})\big).\] Since $E(\zeta_i)<E(\gamma)$, we have
\[E|_{B_m\setminus\{\gamma^m\}}<E(\gamma^m)=m^2\ell^2.\] 
This, together with~\eqref{e:dim_B_m}, implies that $\dim(B_m) =\ind(\gamma^m)+\nul(\gamma^m)=:i_m$. Therefore, as we explained above for $B$, the ball $B_m$ represents a generator of the local homology group 
\[
H_{i_m}(\Sigma_{mk}^{<m\ell}\cup\{\gamma^m\},\Sigma_{mk}^{<m\ell};\Q)
\cong
C_{i_m}(\gamma^m;\Q).
\]

We claim that, for each $\epsilon>0$ sufficiently small, there exists $\overline m=\overline m_\epsilon\in\N$ such that, for all integers $m\geq\overline m$, the homomorphism
\[
C_{i_m}(\gamma^m;\Q)\to H_{i_m}(\Lambda^{<m\ell+\epsilon/\ell},\Lambda^{<m\ell};\Q)
\]
induced by the inclusion is the zero one.
Indeed, let $\epsilon\in(0,1)$ be small enough so that 
\begin{align*}
\max_{\Sigma\times\partial U} E\circ\phi < \ell^2-\epsilon.
\end{align*}
If needed, we shrink $\Sigma$ around $\gamma(0)$ so that
\begin{align*}
\diam(\Sigma):=
\max_{x_1,x_2\in\Sigma} d(x_1,x_2)<\frac{\epsilon}{2(\ell^2+2)},
\end{align*}
where $d:M\times M\to[0,\infty)$ denotes the (possibly non-symmetric) distance \eqref{e:Finsler_distance} induced by the Finsler metric $F$. Let $\delta>0$ be such that
\begin{align*}
\max_{\partial\Sigma\times U} E\circ\phi = \ell^2-\delta,
\end{align*}
and notice that
\begin{align*}
\max_{\partial\Sigma\times U^{\times m}} E\circ\phi_m = m^2( \ell^2-\delta).
\end{align*}
We define the continuous map
\begin{gather*}
\psi_m:\Sigma\times\Sigma\times U^{\times \lfloor m/2\rfloor}\times U^{\times \lceil m/2\rceil}\to\Lambda,
\\
\psi_m(x_1,x_2,y_1,y_2)=\phi_{\lfloor m/2\rfloor}(x_1,y_1)*\gamma_{x_1 x_2}*\phi_{\lceil m/2\rceil}(x_2,y_2)*\gamma_{x_2 x_1},
\end{gather*}
where $\gamma_{x_ix_j}:[0,d(x_i,x_j)]\to M$ is the shortest geodesic parametrized with unit speed joining $x_i$ and $x_j$. Let us compute the composition $E\circ\psi_m$. If we set 
\[\zeta_1:=\phi_{\lfloor m/2\rfloor}(x_1,y_1),\quad \zeta_2:=\phi_{\lceil m/2\rceil}(x_2,y_2),\quad \zeta:=\psi_m(x_1,x_2,y_1,y_2),\] 
we have
\begin{align*}
E(\zeta)
&=
\big(m+d(x_1,x_2)+d(x_2,x_1)\big)
\left( \frac{ E(\zeta_1)}{\lfloor \tfrac m2\rfloor} + d(x_1,x_2)  +\frac{ E(\zeta_2)}{\lceil \tfrac m2\rceil} +d(x_2,x_1) \right)\\
&\leq \big(m+2\diam(\Sigma)\big)\big(m\ell^2+2\diam(\Sigma)\big)\\
&<m^2\ell^2 + 2\diam(\Sigma)(m+\ell^2m + 2\diam(\Sigma))\\
&<m^2\ell^2+m\epsilon
<
(m\ell+\epsilon/\ell)^2
.
\end{align*}
If $y_1\in\partial U^{\times \lfloor m/2\rfloor}$ or $y_2\in\partial U^{\times \lceil m/2\rceil}$, we have the estimate
\begin{align*}
E(\zeta)
&<
\big(m+2\diam(\Sigma)\big)
\left( m\ell^2-\epsilon + 2\diam(\Sigma) \right)\\
&=
m^2\ell^2+ m\big(2\diam(\Sigma)(1+\ell^2)-\epsilon\big) + 2\diam(\Sigma)\big(2\diam(\Sigma)-\epsilon\big)\\
&<m^2\ell^2.
\end{align*}
If instead $x_1\in\partial\Sigma$ or $x_2\in\partial\Sigma$, we have
\begin{align*}
 E(\zeta)
 & \leq
 \big(m+2\diam(\Sigma)\big)
\left(m\ell^2 - \lfloor \tfrac m2\rfloor\delta + 2\diam(\Sigma) \right)\\
 & \leq
m^2\ell^2
+ \big(2\diam(\Sigma)(2+\ell^2)-\lfloor\tfrac m2 \rfloor\delta\big)m\\
 & \leq
 m^2\ell^2 + \underbrace{\big( \epsilon -  \lfloor \tfrac m2\rfloor\delta \big)}_{(*)}m,
\end{align*}
and the term $(*)$ is negative for $m\geq \overline m=2\epsilon/\delta+2$. Summing up, our map $\psi_m$ satisfies 
\begin{align}
\label{e:E_low_boundary}
 E\circ\psi_m|_{\partial (\Sigma\times\Sigma\times U^m)}<m^2\ell^2.
\end{align}
The relative cycle $\mathrm{diag}_\Sigma \times U^m$ is null-homologous in \[H_*(\Sigma\times\Sigma\times U^m,\partial (\Sigma\times\Sigma\times U^m);\Q),\] since it is homologous to a relative cycle contained in $\partial (\Sigma\times\Sigma\times U^m)$, see Figure~\ref{f:cube}. Therefore, \eqref{e:E_low_boundary} implies that the relative cycle $B_m=\phi_m(\Sigma\times U^m)=\psi_m(\mathrm{diag}_\Sigma \times U^m)$ is null-homologous in 
$H_{i_m}(\Lambda^{<m\ell+\epsilon/\ell},\Lambda^{<m\ell};\Q)$, i.e. 
\begin{align}
\label{e:homological_vanishing}
[B_m]=0 \mbox{ in } H_{i_m}(\Lambda^{<m\ell+\epsilon/\ell},\Lambda^{<m\ell};\Q). 
\end{align}

\begin{figure}
\begin{center}
\begin{small}
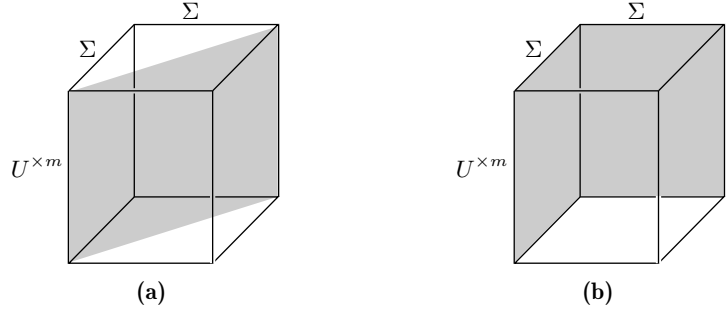 
\caption{The shaded region in \textbf{(a)} is the relative cycle $\mathrm{diag}_\Sigma \times U^m$, which is null-homologous in $H_*(\Sigma\times\Sigma\times U^{\times m},\partial(\Sigma\times\Sigma\times U^{\times m});\Q)$, for instance because it is homologous to the shaded region in \textbf{(b)}.}
\label{f:cube}
\end{small}
\end{center}
\end{figure}

From now on, we assume that our integer $m\geq\overline m$ is odd, and employ Morse theory. Since $m$ is odd, Lemma~\ref{l:Cgamma_emb_CS1gamma} implies that the inclusion induces an injective homomorphism 
$C_*(\gamma^{m};\Q)\hookrightarrow C_*(S^1\cdot\gamma^{m};\Q)$. 
Therefore, the commutative diagram 
\begin{equation*}
\begin{tikzcd}[row sep=large]
  C_{i_{m}}(\gamma^{m};\Q)\, 
  \arrow[r,hookrightarrow] 
  \arrow[dr,"0"'] &
  C_{i_{m}}(S^1\cdot\gamma^{m};\Q)   
  \arrow[d,"\iota_*"] 
  \\
   &
  H_{i_{m}}(\Lambda^{<m\ell+\epsilon/\ell},\Lambda^{<m\ell};\Q)
\end{tikzcd}
\end{equation*}
whose homomorphisms are all induced by the inclusion, implies that $\iota_*$ is not injective. This, in turn, implies that there exists a closed geodesic $\zeta_m\in\crit(E)$ of length 
$E(\zeta_m)^{1/2}\in(m\ell,m\ell+\epsilon/\ell]$. 

Summing up, we showed that for every $\epsilon>0$ small enough there exists $\overline m>0$ and, for all odd integers $m\geq\overline m$, a closed geodesic $\zeta_m\in\crit(E)$ such that $m\ell < E(\zeta_m)^{1/2}\leq m\ell+\epsilon$. We can now invoke Lemma~\ref{l:arithmetic_existence} with $\K$ being the set of odd positive integers, and conclude that $(M,F)$ has infinitely many closed geodesics.
\end{proof}

We now derive the two corollaries that we will need for proving Theorem~\ref{t:multiplicity}.

\begin{cor}
\label{c:Hingston}
Let $(M,F)$ be an orientable Finsler surface, and $\gamma\in\crit(E)$ a closed geodesic such that $\ind_\Omega(\gamma)=\nul_\Omega(\gamma)=1$ and the local homology $C_3(\gamma)$ with some coefficient field is non-zero. Then $(M,F)$ has infinitely many closed geodesics.
\end{cor}

\begin{proof}
Since $C_3(\gamma)$ is non-trivial, we have 
\begin{align*}
 \ind(\gamma)\leq3\leq\ind(\gamma)+\nul(\gamma).
\end{align*}
We now employ Proposition~\ref{p:indices}. Since $\nul_\Omega(\gamma)=1$, we have
\begin{align*}
\ind_\Omega(\gamma^m) = m\,\ind_\Omega(\gamma) + (m-1)\nul_\Omega(\gamma)= 2m-1,
\qquad
\nul_\Omega(\gamma^m)=1.
\end{align*}
Moreover
\begin{align*}
\ind(\gamma)\leq\ind_\Omega(\gamma)+1=2,
\qquad
\nul(\gamma)\geq3-\ind(\gamma)\geq1.
\end{align*}
If $\nul(\gamma)=2$, then $\nul(\gamma^m)=2$ and $\ind(\gamma^m)$ is odd for all $m\in\N$; since $\ind_\Omega(\gamma^m)\leq\ind(\gamma^m)\leq\ind_\Omega(\gamma^m)+1$,
we infer
\begin{align*}
\ind(\gamma^m)=\ind_\Omega(\gamma^m)=2m-1.
\end{align*}
If instead $\nul(\gamma)=1$, then $\ind(\gamma)=2$, and
\begin{align*}
\ind(\gamma^m) + \nul(\gamma^m)
&\leq
\ind_\Omega(\gamma^m) + \nul_\Omega(\gamma^m)+1\\
&=
2m+1\\
&=
m(\ind(\gamma)+\nul(\gamma))-(m-1). 
\end{align*}
In both cases, $\gamma$ satisfies the assumptions of Theorem~\ref{t:Hingston}, and we infer that $(S^2,F)$ has infinitely many closed geodesics.
\end{proof}

The second corollary of Theorem~\ref{t:Hingston} was established in the Riemannian case by Bangert \cite{Bangert:1980ho, Bangert:1993wo}. Even though we present it here as a corollary of Theorem~\ref{t:Hingston}, Bangert's proof came historically earlier than \cite{Hingston:1993ou}. Let us recall, once again, the classical notion of conjugate points: two points $\gamma(t)$ and $\gamma(s)$ along a geodesic $\gamma:[t,s]\to M$ are conjugate when there exists a Jacobi field along $\gamma$ that is not identically zero, but vanishes at both $\gamma(t)$ and $\gamma(s)$. When $\dim(M)=2$ this condition can be expressed in terms of the Jacobi field $\eta_t$ introduced in Equation~\eqref{e:eta}: $\gamma(t)$ and $\gamma(s)$ are conjugate points if and only if $\eta_t(s)=0$. A closed geodesic $\gamma$ on a Finsler surface has no conjugate points if and only if $\ind_\Omega(t\cdot\gamma^m)=\nul_\Omega(t\cdot\gamma^m)=0$ for all $t\in S^1$ and $m\in\N$; equivalently, $\ind(\gamma^m)=0$ for all $m\in\N$, according to Lemma~\ref{l:elementary_ind}(iii) and Proposition~\ref{p:indices}(iii).

\begin{cor}
\label{c:conjugate_points}
Any reversible Finsler 2-sphere with a simple closed geodesic without conjugate points possesses infinitely many closed geodesics.
\end{cor}

\begin{proof}
Let $\gamma\in\crit(E)\cap E^{-1}(\ell^2)$ be a simple closed geodesic without conjugate points in the reversible Finsler 2-sphere $(S^2,F)$. 
We claim that there exists a neighborhood $\UU\subset \Lambda$ of the critical circle $S^1\cdot\gamma:=\{t\cdot\gamma\ |\ t\in S^1\}$ such that every $\zeta\in\UU$ that intersects the curve $\gamma$ has energy $E(\zeta)\geq E(\gamma)$. Indeed, if this were not true, we could find a sequence $\zeta_n\in\Lambda$ such that $\zeta_n(0)=\gamma(t_n)$,  $E(\zeta_n)<E(\gamma)$, $t_n\to t$ and $\zeta_n\to t\cdot\gamma$ as $n\to\infty$. We consider the based loop spaces
\[\Omega_s:=\{\zeta\in\Lambda\ |\ \zeta(0)=\gamma(s)\},\qquad s\in S^1,\] 
and the space of broken closed geodesics $\Lambda_k$ introduced in Section~\ref{ss:energy}. Here $k\in\N$ must be large enough so that $\gamma\in\Lambda_k$. Since $\gamma$ has no conjugate points, $\ind_\Omega(s\cdot\gamma)=\nul_\Omega(s\cdot\gamma)=0$. Therefore, every  $s\cdot\gamma$ is a non-degenerate local minimizer of $E|_{\Lambda_k\cap\Omega_s}$. Since $E_k:=E|_{\Lambda_k}$ is smooth in a neighborhood of the critical circle of $\gamma$, we can apply the parametric Morse lemma, which provides an $\epsilon>0$ and an open neighborhood $U\subset\Lambda_k$ of $\gamma$ such that, for all $t\in(-\epsilon,\epsilon)$, $t\cdot\gamma$ is the unique global minimizer of $E|_{U\cap\Omega_t}$. Let $\gamma_n\in\Lambda_k\cap\Omega_{t_n}$ be the sequence of broken closed geodesics such that $\gamma_n(i/k)=\zeta_n(i/k)$ for all $i\in\Z_k$, which have energy $E(\gamma_n)\leq E(\zeta_n)<E(\gamma)$. Since $\zeta\to t\cdot\gamma$ in $\Lambda$, we would have that $\gamma_n\to\gamma$ in $\Lambda_k$, and in particular $\gamma_n\in U$ for all $n\in\N$ large enough, contradicting the fact that $E|_{U\cap\Omega_{t_n}}$ has a strict global minimizer at $t_n\cdot\gamma$.

We denote by $B_0$ and $B_1$ the connected components of $S^2\setminus\gamma$, and by $\BB_i\subset\Lambda$ the open subset of those $\zeta\in\Lambda$ such that $\zeta(S^1)\subset B_i$. Since we are looking for infinitely many closed geodesics, we can assume that $\gamma$ is an isolated closed geodesic (i.e., the critical circle of each iterate $\gamma^m$ is isolated in $\crit(E)$). We set $\ell^2:=E(\gamma)$. We have two possible cases, which we deal with separately.

\vspace{5pt}

\emph{Case 1:} For every open neighborhood $\VV\subset\Lambda$ of $\gamma$, the intersections $\VV\cap\BB_0\cap\Lambda^{<\ell}$ and $\VV\cap\BB_1\cap\Lambda^{<\ell}$ are both non-empty. If we choose $\VV$ contained in the above open subset $\UU$, we infer that every connected component of $\VV^{<\ell}:=\VV\cap\Lambda^{<\ell}$ is contained in either $\VV^{<\ell}\cap\BB_0$ or $\VV^{<\ell}\cap\BB_1$. In particular $\VV^{<\ell}$ is not connected. Since $\VV$ can be chosen arbitrarily small, Lemma~\ref{l:local_homology_mountain_pass} implies that the local homology $C_1(\gamma;\Q)$ is non-zero. Since $\ind(\gamma^m)=0$, Proposition~\ref{p:indices}(vi) implies that $\nul(\gamma^m)<2$ for all $m\in\N$. Since the local homology $C_1(\gamma;\Q)$ is non-zero, we must have $\ind(\gamma)+\nul(\gamma)\geq1$, and thus $\nul(\gamma)=1$. Since $\nul(\gamma)\leq\nul(\gamma^m)<2$, we infer that $\nul(\gamma^m)=1$ for all $m\in\N$. Therefore, $\gamma$ satisfies the assumptions of Theorem~\ref{t:Hingston}, which implies that $(S^2,F)$ has infinitely many closed geodesics.

\begin{figure}
\begin{center}
\begin{small}
%% Creator: Inkscape 1.0beta1 (32d4812, 2019-09-19), www.inkscape.org
%% Accompanies image file 'annulus.pdf' (pdf, eps, ps)
\begingroup%
  \makeatletter%
  \providecommand\color[2][]{%
    \errmessage{(Inkscape) Color is used for the text in Inkscape, but the package 'color.sty' is not loaded}%
    \renewcommand\color[2][]{}%
  }%
  \providecommand\transparent[1]{%
    \errmessage{(Inkscape) Transparency is used (non-zero) for the text in Inkscape, but the package 'transparent.sty' is not loaded}%
    \renewcommand\transparent[1]{}%
  }%
  \providecommand\rotatebox[2]{#2}%
  \newcommand*\fsize{\dimexpr\f@size pt\relax}%
  \newcommand*\lineheight[1]{\fontsize{\fsize}{#1\fsize}\selectfont}%
  \ifx\svgwidth\undefined%
    \setlength{\unitlength}{320.75531083bp}%
    \ifx\svgscale\undefined%
      \relax%
    \else%
      \setlength{\unitlength}{\unitlength * \real{\svgscale}}%
    \fi%
  \else%
    \setlength{\unitlength}{\svgwidth}%
  \fi%
  \global\let\svgwidth\undefined%
  \global\let\svgscale\undefined%
  \makeatother%
  \begin{picture}(1,0.43494224)%
    \lineheight{1}%
    \setlength\tabcolsep{0pt}%
    \put(0,0){\includegraphics[width=\unitlength,page=1]{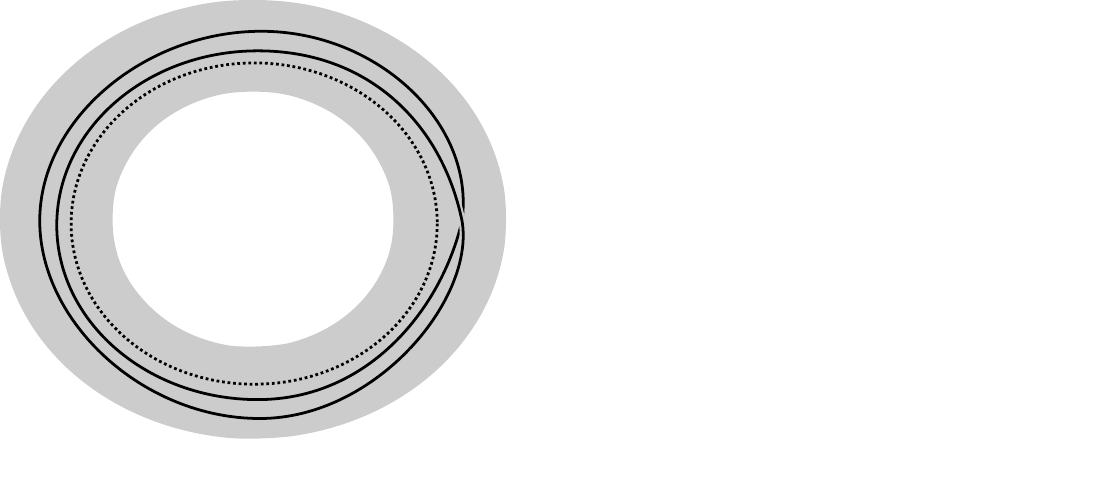}}%
    \put(0.07987228,0.27144286){\color[rgb]{0,0,0}\makebox(0,0)[lt]{\lineheight{1.25}\smash{\begin{tabular}[t]{l}$\gamma$\end{tabular}}}}%
    \put(0.07051935,0.35912652){\color[rgb]{0,0,0}\makebox(0,0)[lt]{\lineheight{1.25}\smash{\begin{tabular}[t]{l}$\zeta$\end{tabular}}}}%
    \put(0,0){\includegraphics[width=\unitlength,page=2]{annulus.pdf}}%
    \put(0.62552669,0.27144286){\color[rgb]{0,0,0}\makebox(0,0)[lt]{\lineheight{1.25}\smash{\begin{tabular}[t]{l}$\gamma$\end{tabular}}}}%
    \put(0.2129693,0.00524884){\color[rgb]{0,0,0}\makebox(0,0)[lt]{\lineheight{1.25}\smash{\begin{tabular}[t]{l}\textbf{(a)}\end{tabular}}}}%
    \put(0.7507626,0.00524884){\color[rgb]{0,0,0}\makebox(0,0)[lt]{\lineheight{1.25}\smash{\begin{tabular}[t]{l}\textbf{(b)}\end{tabular}}}}%
  \end{picture}%
\endgroup%
 
\caption{\textbf{(a)} The simple closed geodesic $\gamma$ (dotted), and a curve $\zeta\in\Lambda$ close to $\gamma^2$ with a self-intersection. \textbf{(b)} The support of the curve $\zeta$ can be decomposed as the union of $\zeta_1$ (dashed curve) and $\zeta_2$ (solid curve), both close to $\gamma$.}
\label{f:intersections}
\end{small}
\end{center}
\end{figure}

\vspace{5pt}

\emph{Case 2:} For some $i\in\{0,1\}$, there exists an open neighborhood $\VV\subset\Lambda$ of $\gamma$ such that $\VV\cap\BB_i\cap\Lambda^{<\ell}=\varnothing$. This implies the analogous property for $\gamma^m$: there exists an open neighborhood $\VV_m\subset\Lambda$ of $\gamma^m$ such that $\VV_m\cap\BB_i\cap\Lambda^{<\ell}=\varnothing$. Indeed, since we are on an orientable surface, a tubular neighborhood of the simple closed geodesic $\gamma$ is diffeomorphic to the annulus $S^1\times(-1,1)$, $\gamma$ being its zero section $S^1\times\{0\}$. Therefore, any curve $\zeta$ sufficiently close to the iterated curve $\gamma^m$ has at least $m-1$ self-intersections counted with multiplicity (see Figure~\ref{f:intersections}). The support of $\zeta$ can be decomposed as the union of the supports of $\zeta_1,...,\zeta_m\in\VV$, and $E(\zeta)=m(E(\zeta_1)+...+E(\zeta_m))$, see for instance  \cite[Lemma~4.2]{Abbondandolo:2015lt}. If $E(\zeta)< E(\gamma^m)=m^2E(\gamma)$, we would have $E(\zeta_j)<E(\gamma)$ for some $j$, contradicting the fact that $\VV\cap\BB_i\cap\Lambda^{<\ell}=\varnothing$.

Since $\gamma$ is an isolated closed geodesic, we can choose $\VV_m$ to be small enough so that $\VV_m\cap\BB_i\cap\crit(E)=\varnothing$. This implies that \[E(\zeta)>E(\gamma^m),\qquad \forall \zeta\in\VV_m\cap\BB_i.\] 
Indeed, by the previous paragraph we know that $E(\zeta)\geq E(\gamma^m)$ for any $\zeta\in\VV_m\cap\BB_i$. If we had equality $E(\zeta)= E(\gamma^m)$ then $\zeta$ would be a local minimizer, and in par\-tic\-u\-lar $\zeta\in\crit(E)$.

Let us fix a homotopy $u:[0,1]\to\BB_i\cup\{\gamma^m\}$, $u(t)=u_t$, such that $u_0=\gamma^m$ and $E(u_1)=0$; namely, $u_t$ defines a contraction of $\gamma$ to a point within the disk $B_i$. We choose an integer 
\[k > \frac{\max\{E\circ u\}}{\injrad(S^2,F)^2},\] 
and consider the space of broken closed geodesics $\Lambda_k$ and the restricted energy functional $E_k=E|_{\Lambda_k}$. We have the associated retraction
\begin{align*}
r:\Lambda^{<\sqrt k\,\injrad(S^2,F)}\to\Lambda_k,
\qquad 
r(\zeta)=\tilde\zeta,
\end{align*}
where $\tilde\zeta\in\Lambda_k$ is the broken closed geodesics such that $\tilde\zeta(i/k)=\zeta(i/k)$ for all $i\in\Z_k$. We recall that $E_k\circ r\leq E$.
Since the boundary of the disk $B_i$ is geodesic, we readily see that $r$ preserves $\BB_i$. Consider an open neighborhood $V\subset\Lambda_k$ of $\gamma^m$ with compact closure $\overline V\subset\VV_m$. Since $\partial V$ is compact, we have
\begin{align*}
b^2:=\min_{\partial V \cap \BB_i} E_k > E_k(\gamma^m).
\end{align*}
Let $W\subset V$ be a small enough open neighborhood of $\gamma^m$ such that
\begin{align*}
E_k(\gamma^m)<a^2:=\max_{\overline W} E_k<b^2.
\end{align*}
We consider
\begin{align*}
 c^2:=\inf_{v} \max \{E_k\circ v\},
\end{align*}
where the infimum ranges over all homotopies $v:[0,1]\to \Lambda_k\cap\BB_i\cup\{\gamma^m\}$, $v(t)=v_t$, such that $v_0=\gamma^m$ and $E_k(v_1)=0$. Notice that the space of such homotopies is non-empty, as it contains $r\circ u$.
We have $c\geq b$, since every such a homotopy $v$ must eventually intersect $\partial V\cap \BB_i$. We fix an arbitrary $d\in(c,k\,\injrad(S^2,F)^{1/2})$. Notice that $E_k^{-1}[a^2,d^2]\cap\BB_i$ is compact, since it is a closed subset of the compact set $E_k^{-1}[a^2,d^2]\setminus W$. Therefore, the classical min-max theorem implies that $c^2$ is a critical value of $E_k$. Since we are looking for infinitely many closed geodesics, we can assume that $(S^2,F)$ has only isolated closed geodesic (i.e. any critical circle is isolated in $\crit(E)\cap E^{-1}(0,\infty)$).
Under this assumption, there exists at least one closed geodesic $\zeta_m\in\crit(E_k)\cap E_k^{-1}(c^2)\cap\BB_i$ such that every open neighborhood $Z\subset\Lambda_k\cap\BB_i$ of $\zeta_m$ has a non-connected intersection $Z^{<c}=Z\cap \Lambda_k^{<c}$. Indeed, if no closed geodesic in $\crit(E_k)\cap E_k^{-1}(c^2)\cap\BB_i$ satisfied this property, we could find a homotopy $v$ as above such that $\max\{E_k\circ v\}<c^2$, contradicting the definition of the min-max value $c^2$. Lemma~\ref{l:local_homology_mountain_pass} implies that the local homology $C_1(\zeta_m;\Q)$ is non-trivial.

Now, either the family $\{\zeta_m\ |\ m\in\N\}$ that we found contains infinitely many geometrically distinct closed geodesics, or there exists a non-iterated closed geodesic $\zeta$, an infinite subset $\K\subset\N$, and a function $\mu:\K\to\N$ such that $\mu(m)\to\infty$ as $m\to\infty$ and $\zeta_m=\zeta^{\mu(m)}$ for all $m\in\N$. Since every iterate $\zeta^{\mu(m)}$ has non-trivial local homology $C_1(\zeta^{\mu(m)})$, we have $\ind(\zeta^{\mu(m)})\leq1$ for all $m\in\K$, and therefore $\ind(\zeta^m)=0$ for all $m\in\N$ according to Lemma~\ref{l:elementary_ind}(iii). We cannot have $\nul(\zeta)=0$, for otherwise $\zeta$ would be a local minimizer of $E_k$, and the same would be true for all its iterates according to analogous argument of Figure~\ref{f:intersections}. Therefore $1\leq\nul(\zeta)\leq\nul(\zeta^m)\leq2$ for all $m\in\N$. By Proposition~\ref{p:indices}(vi), since the Morse indices $\ind(\zeta^m)$ vanish, we must have $\nul(\zeta^m)=1$ for all $m\in\N$. 

Since $\ind(\zeta)=\ind(\zeta^m)$ and $\nul(\zeta)=\nul(\zeta^m)$ for all $m\in\N$, we have an isomorphism of local homology groups
$C_*(\zeta;\Q) \cong
C_*(\zeta^m;\Q)$.
In particular $C_1(\zeta;\Q)$ does not vanish. Therefore, $\zeta$ satisfies the assumptions of Theorem~\ref{t:Hingston}, which implies that $(S^2,F)$ has infinitely many closed geodesics.
\end{proof}

\subsection{Bangert's theorem}

As it turns out, the statements proved so far allow us to conclude the existence of infinitely many closed geodesics on any reversible $(S^2,F)$, except when none of its simple closed geodesics has a well-defined Birkhoff map. We recall that a simple closed geodesic $\gamma$ of a reversible $(S^2,F)$ does not have a well defined Birkhoff map when, for some $x=\gamma(t)$ and $v\in\Tan_x S^2$ transverse to $\dot\gamma(t)$, the geodesic $\zeta(t)=\exp_x(tv)$ does not intersect $\gamma$ at any positive time $t>0$.
In this section, we show that this last case is covered by Corollary~\ref{c:conjugate_points}. For Riemannian 2-spheres, this is a theorem due to Bangert \cite{Bangert:1993wo}.

\begin{thm}\label{t:Bangert}
Any reversible Finsler 2-sphere having a simple closed geodesic without a well-defined Birkhoff map possesses infinitely many closed geodesics.
\end{thm}

The proof is based on the following lemma of independent interest.

\begin{lem}
\label{l:no_conjugate_points}
Let $(M,F)$ be a (not necessarily reversible) Finsler surface, and $\gamma:[-T,T]\to M$ a geodesic parametrized with constant speed. If there exists a sequence of geodesics $\gamma_n:[-T,T]\to M$ parametrized with constant speed, not intersecting $\gamma$, and such that $(\gamma_n(0),\dot\gamma_n(0))\to(\gamma(0),\dot\gamma(0))$ in $\Tan M$, then $\gamma|_{(-T,T)}$ has no conjugate points. 
\end{lem}

\begin{proof}
Since the problem is local to $\gamma$, we can assume without loss of generality that $M=\R^2$ and $\gamma(t)=(t,0)$ for all $t\in[-T,T]$, so that we can write expressions in coordinates. Without loss of generality, we can assume that $F(\gamma,\dot\gamma)\equiv1$. We reparametrize the geodesics $\gamma_n$ so that they have speed $F(\gamma_n,\dot\gamma_n)\equiv1$. By doing this, we change the interval of definition of $\gamma_n$: the reparametrized curve has the form $\gamma_n:[-T_n,T_n]\to M$ with $T_n\to T$ as $n\to\infty$.

We set $(x,v)=(\gamma(0),\dot\gamma(0))=(0,\dot\gamma(0))$, $G=\tfrac12 F^2$, and consider the line
\begin{align*}
\Sigma:=\{y\in\R^2\ |\ y\in\ker G_v(x,v)\}.
\end{align*}
For all $n\in\N$ large enough, the geodesic $\gamma_n:[-T_n,T_n]\to\R^2$ intersects $\Sigma$ in a unique point $x_n$, and clearly $x_n\to x$. We shift the parametrization of $\gamma_n$, so that we have a sequence of geodesics $\gamma_n:[-T_n+\epsilon_n,T_n+\epsilon_n]\to\R^2$ not intersecting $\gamma$ and such that $\epsilon_n\to0$ and $v_n:=\dot\gamma_n(0)\to v$. Up to extracting a subsequence, we can assume that each $\gamma_n$ lies on the same side of $\gamma$. Therefore we have a well defined non-zero vector
\begin{align*}
w':=\frac{x_n-x}{\|x_n-x\|}\in\Sigma
\end{align*}
independent of $n$. Here, $\|\cdot\|$ denotes the Euclidean norm.
We consider the vectors
\begin{align*}
z_n:=\frac{v_n-v}{\|x_n-x\|}.
\end{align*}
Since the geodesics $\gamma_n$ and $\gamma$ do not intersect on the time interval $[-T/2,T/2]$ and the second derivative of $\gamma_n-\gamma$ is uniformly bounded on $[-T/2,T/2]$ independently of $n$, we readily obtain that the sequence $\|z_n\|$ is uniformly bounded from above. In particular, up to extracting a subsequence, we have $z_n\to z'$ as $n\to\infty$.

We set
$\lambda:= (1+\|z'\|^2)^{-1/2}$, $w:=\lambda w'$, and $z:=\lambda z'$,
so that
\begin{align*}
\lim_{n\to\infty} \frac{(x_n,v_n)-(x,v)}{\|(x_n,v_n)-(x,v)\|} = \lim_{n\to\infty} \frac{(x_n,v_n)-(x,v)}{\|x_n-x\| \sqrt{1+\|z_n\|^2}} = (w,z) \in \Tan_{(x,v)}S\R^2,
\end{align*}
where $S\R^2=\{(x',v')\in\Tan\R^2\ |\ F(x',v')=1\}$. Therefore, 
\begin{align*}
\lim_{n\to\infty} \frac{\phi_t(x_n,v_n)-\phi_t(x,v)}{\|(x_n,v_n)-(x,v)\|} = \diff\phi_t(x,v)(w,z).
\end{align*}
The Jacobi field $\zeta:(-T,T)\to\R^2$ along $\gamma$ defined by
\[(\zeta(t),\dot\zeta(t))=\diff\phi_t(x,v)(w,z)\]
satisfies $G_v(\gamma(t),\dot\gamma(t))\zeta(t)=0$.

We claim that $\zeta$ is nowhere vanishing. Indeed, let $\nu:(-T,T)\to\R^{2}$ the smooth vector field along $\gamma$ defined by $\nu(t)\in\ker G_v(\gamma(t),\dot\gamma(t))$, $\|\nu(t)\|=1$, and $\nu(t)$ pointing to the the side of $\gamma$ containing the $\gamma_n$'s. We can write $\zeta(t)=z(t)\nu(t)$ for some continuous function $z=(-T,T)\to\R$. Notice that
\begin{align*}
\zeta(t)
=
\lim_{n\to\infty} \frac{\gamma_n(t)-\gamma(t)}{\|(x_n,v_n)-(x,v)\|}.
\end{align*}
If $\zeta(t)=0$ for some $t$, since $\gamma_n(t)-\gamma(t)$ and $\nu(t)$ point to the same side of $\gamma$, we readily obtain that $z(t)=0$ and $\dot z(t)=0$. But this would imply that $\zeta(t)=\dot\zeta(t)=0$, and since $\zeta$ is a Jacobi field we would conclude that $\zeta$ vanishes identically, contradicting $\zeta(0)=w$.

If $\gamma$ had conjugate points $\gamma(t_1),\gamma(t_2)$ for some $-T<t_1<t_2<T$, there would exists a Jacobi field $\eta:[t_1,t_2]\to\R^2$ such that $\eta(t_1)=\eta(t_2)=0$, $\dot\eta(t_1)\neq0$, and $\eta(t)\in\ker G_v(\gamma(t),\dot\gamma(t))$ for all $t\in[t_1,t_2]$. Since we are on a surface, $\ker G_v(\gamma(t),\dot\gamma(t))$ is 1-dimensional. Therefore, by the Sturm separation theorem, $\eta$ and $\zeta$ would have alternating zeroes, contradicting the fact that $\zeta$ is nowhere vanishing. 
\end{proof}

\begin{proof}[Proof of Theorem~\ref{t:Bangert}]
Let $\gamma_0\in\crit(E)\cap E^{-1}(0,\infty)$ be a simple closed geodesic that does not have a well defined Birkhoff map. We only need to consider the case in which $\gamma_0$ has conjugate points (i.e.\ $\ind_\Omega(\gamma_0^m)>0$ for some integer $m\geq1$), for otherwise the existence of infinitely many closed geodesics is already provided by Corollary~\ref{c:conjugate_points}. The fact that $\gamma_0$ does not have a well defined Birkhoff map means that, for some $x_0=\gamma_0(t_0)$ and $v_0\in S_xS^2$ transverse to $\dot\gamma_0(t_0)$, the geodesic $\zeta:(0,\infty)\to S^2$, $\zeta(t)=\exp_{x_0}(tv_0)$ does not intersect $\gamma_0$ in positive time, and therefore stays trapped in a connected component $B\subset S^2\setminus\gamma_0(S^1)$. We consider the compact subset
\begin{align*}
 K := \bigcap_{t>0} \overline{\zeta[t,\infty)} \subset \overline B.
\end{align*}

We claim that
\[K\cap\gamma_0(S^1)=\varnothing.\] 
Otherwise we can find a sequence $t_n\to\infty$ and $s\in \R$ such that $\zeta(t_n)\to\gamma_0(s)$. Since $\zeta$ does not intersect $\gamma_0$ in positive time, up to extracting a subsequence we must have $\dot\zeta(t_n)\to\dot\gamma_0(s)$. Since $\gamma_0$ has conjugate points, there exists $\delta>0$ such that $\gamma_0|_{(s-\delta,s+\delta)}$ has conjugate points. Lemma~\ref{l:no_conjugate_points} thus provides a contradiction: since $\zeta|_{[t_n-\delta,t_n+\delta]}$ does not intersect $\gamma_0|_{[s-\delta,s+\delta]}$, $\gamma_0|_{(s-\delta,s+\delta)}$ cannot have conjugate points.

Let $U\subset B\setminus K$ be the connected component whose closure contains $\gamma_0(S^1)$. One would expect this open set to be locally geodesically convex. We prove a slightly weaker convexity: for all $x,y\in U$ that can be joined by means of an absolutely continuous curve in $U$ of length strictly less than $\rho:=\injrad(S^2,F)$, the shortest geodesic joining $x$ and $y$ is entirely contained in $U$. Indeed, let $\gamma_{x,y}:[0,1]\to S^2$, $\gamma_{x,y}(t)=\exp_x(t\exp_x^{-1}(y))$ be such a geodesic, and assume by contradiction that some $z=\gamma_{x,y}(s)$ belongs to $K$. Then, by the definition of $K$, there exists a sequence $t_n\to\infty$ such that $\zeta(t_n)\to z$. Up to extracting a subsequence, the sequence $\dot\zeta(t_n)$ converges to some $w\in S_zS^2$ that is transverse to $\dot\gamma_{x,y}(s)$, since the geodesic $\theta:\R\to S^2$, $\theta(t)=\exp_z(tw)$ is entirely contained in $K$. We denote the geodesic balls centered at $z$ by
\begin{align*}
B(z,r):= \big\{z'\in S^2\ \big|\ d(z,z')<r \big\},\qquad r>0,
\end{align*}
where $d:S^2\times S^2\to[0,\infty)$ is the distance~\eqref{e:Finsler_distance} induced by the Finsler metric $F$. The points $x$ and $y$ are contained in different connected components of $B(z,\rho)\setminus\zeta(-\rho,\rho)$. Therefore, every continuous curve $\theta:[0,1]\to U$ such that $\theta(0)=x$ and $\theta(1)=y$ must leave the geodesic ball $B(z,\rho)$; since $d(x,z)+d(z,y)<\rho$, we readily obtain that the length of such a $\theta$ is larger than $\rho$, contradicting our assumption on $x,y$.

We consider the space
\begin{align*}
 W:=\big\{\gamma\in\Lambda\ \big|\ \gamma(S^1)\subset U,\ \gamma\mbox{ not contractible in }U,\ E(\gamma)<E(\gamma_0)  \big\}
\end{align*}
We claim that $W$ is not empty. Indeed, since $\gamma_0$ has conjugate points, by Proposition~\ref{p:indices}(vii) there exists a nowhere vanishing 1-periodic vector field $\xi$ along $\gamma$ such that $\diff^2E(\gamma)[\xi,\xi]<0$, and $\xi(t)$ points inside $U$ for all $t\in S^1$. We define $\gamma_s\in \Lambda$ by 
\[\gamma_s(t)=\exp_{\gamma_0(t)}(s\xi(t)).\] 
If $s>0$ is small enough, then $\gamma_s$ is contained in $U$, non-contractible in $U$ (since it is homotopic to $\gamma_0$ within $U\cup\gamma_0(S^1)$), and since
\begin{align*}
E(\gamma_s) \leq E(\gamma_0) + \tfrac12 s^2\,\diff^2E(\gamma)[\xi,\xi] + o(s^2)
\end{align*}
we have $E(\gamma_s)<E(\gamma_0)$. Thus any such $\gamma_s$ belongs to $W$.

We fix $k\in\N$ large enough so that $\gamma_0$ is contained in the space of broken closed geodesics $\Lambda_k\subset\Lambda$. We define the continuous map $r:W\to\Lambda_k$ by  $r(\gamma)(\tfrac ik)=\gamma(\tfrac ik)$ for all $i\in\Z_k$. The above convexity property of $U$ implies that, for each $\gamma\in W$, $r(\gamma)$ is a curve contained in $U$ and homotopic to $\gamma$ within $U$. Therefore, $r$ is a retraction $r:W\to W\cap\Lambda_k$. Since $E(r(\gamma))\leq E(\gamma)$, we have
\begin{align*}
\ell^2 :=\inf_W E = \inf_{W\cap\Lambda_k} E .
\end{align*}

We choose a sequence $\gamma_n\in W\cap\Lambda_k$ such that $E(\gamma_n)\to \ell^2$. We can assume that each $\gamma_n$ is without self-intersection. Indeed, if $\gamma_n$ has self-intersections, we can find an interval $[a,b]\subsetneq[0,1]$ such that $\gamma_n|_{[a,b]}$ is a non-contractible loop in $U$. If $i_0,i_1$ are positive integers such that
\[
[\tfrac{i_0+1}{k},\tfrac{i_1-1}{k}]
\subseteq
[a,b]
\subseteq
[\tfrac{i_0}{k},\tfrac{i_1}{k}],
\]
we define $\tilde\gamma_n\in W\cap\Lambda_k$ by setting $\tilde\gamma_n(\tfrac{i}{k})=\gamma_n(a)$ for all $i\in\{0,...,i_0\}\cup\{i_1,...,k-1\}$, and $\tilde\gamma_n(\tfrac{i}{k})=\gamma_n(\tfrac{i}{k})$ for all $i\in\{i_0+1,...,i_1-1\}$. The curve $\tilde\gamma_n$ has less self-intersections than $\gamma_n$, and energy $E(\tilde\gamma_n)\leq E(\gamma_n)$. Since a broken closed geodesic has only finitely many self intersections, by repeating this procedure a finite number of times we eliminate all of them.

Since $\overline {W\cap\Lambda_k}$ is compact, up to extracting a subsequence we have that \[\gamma_n\to\gamma\in\overline {W\cap\Lambda_k},\]
and $E(\gamma)=\ell^2$.
We claim that $\gamma$ is a closed geodesic. This is clear if $\gamma$ is contained in  $W\cap\Lambda_k$, for in this case it would be a critical point of the energy functional $E$. Assume now that $\gamma\in\partial (W\cap\Lambda_k)$, and consider the unique $\theta,\theta_n\in \Lambda_k$ such that \[\theta(\tfrac ik)=\gamma(\tfrac {i + 1/2}k),\quad\theta_n(\tfrac ik)=\gamma_n(\tfrac {i + 1/2}k),\qquad\forall i\in\Z_k.\] 
Clearly, $\theta_n\to\theta$. Moreover, $E(\theta)\leq E(\gamma)$ with equality if and only if $\gamma$ is a closed geodesic. The above convexity property of $U$ implies that $\theta_n\in W\cap\Lambda_k$. Therefore $E(\theta_n)\geq \inf E|_{W}=E(\gamma)$ and $E(\theta)=E(\gamma)$, and we conclude that $\gamma$ is a closed geodesic.

Since the approximating loops $\gamma_n$ are without self-intersections, $\gamma$ is a simple closed geodesic. Therefore, the union $\gamma_0(S^1)\cup\gamma(S^1)$ bounds an open annulus $A\subset U$. Since $E(\gamma)=\inf E|_{W}$, in particular there is no $\tilde\gamma\in\Lambda$ freely homotopic to $\gamma$ with energy $E(\tilde\gamma)<E(\gamma)$ and support $\tilde\gamma(S^1)\subset A$. Therefore, by applying Proposition~\ref{p:indices}(vii) as above, we infer that $\gamma$ has no conjugate points. Corollary~\ref{c:conjugate_points} implies that $(S^2,F)$ has infinitely many closed geodesics.
\end{proof}

\begin{proof}[Proof of Theorem~\ref{t:multiplicity}]
By Theorem~\ref{t:LS}, if $(S^2,F)$ has only finitely many simple closed geodesics, there exists at least one simple closed geodesic $\gamma\in\crit(E)$ with non-zero local homology $C_3(\gamma;\Z_2)$. If $\gamma$ does not have a well defined Birkhoff map, Theorem~\ref{t:Bangert} implies that there are infinitely many closed geodesics. Assume now that $\gamma$ has a well defined Birkhoff map. Since $C_3(\gamma;\Z_2)$ is non-zero, $C_3(t\cdot\gamma;\Z_2)$ is non-zero as well, and
\[
\ind(t\cdot\gamma)\leq3\leq\ind(t\cdot\gamma)+\nul(t\cdot\gamma),
\qquad\forall t\in S^1.
\]
By Proposition~\ref{p:indices}(iv), we have
\begin{align}
\label{e:the_last}
\ind_\Omega(t\cdot\gamma)+\nul_\Omega(t\cdot\gamma)
\geq
\ind(t\cdot\gamma)+\nul(t\cdot\gamma)-1
\geq
2,
\qquad
\forall t\in S^1.
\end{align}
Since $\nul_\Omega(t\cdot\gamma)\leq1$ according to Proposition~\ref{p:indices}(ii), the inequality in~\eqref{e:the_last} implies that $\ind_\Omega(t\cdot\gamma)\geq1$. If $\ind_\Omega(t\cdot\gamma)\geq2$ for all $t\in S^1$, Theorem~\ref{t:twist} implies that there are infinitely many closed geodesics. If instead $\ind_\Omega(t\cdot\gamma)=1$ for some $t\in S^1$, the above inequality implies that $\nul_\Omega(t\cdot\gamma)=1$, and Corollary~\ref{c:Hingston} implies that there are infinitely many closed geodesics.
\end{proof}

\bibliography{_biblio}
\bibliographystyle{amsalpha}

\end{document}